\documentclass[11pt]{article}

\usepackage{amsthm}
\usepackage{amsmath}
\usepackage{amssymb}
\usepackage[margin=1in]{geometry}
\usepackage{enumerate}
\usepackage{hyperref}
\usepackage{mathrsfs}
\usepackage{color}
\usepackage{bm}
\usepackage{mathtools}
\mathtoolsset{showonlyrefs=true}

\title{On the norm equivalence of Lyapunov exponents for regularizing linear evolution equations}
\author{Alex Blumenthal and Sam Punshon-Smith}
\date{\today}

\newcommand{\C}{\mathbb{C}}

\newcommand{\E}{\mathbb{E}}

\newcommand{\N}{\mathbb{N}}
\newcommand{\R}{\mathbb{R}}
\newcommand{\T}{\mathbb{T}}

\newcommand{\Z}{\mathbb{Z}}
\newcommand{\Hbf}{\mathbf{H}}
\newcommand{\Sb}{\mathbb{S}}

\newcommand{\dee}{\mathrm{d}}
\newcommand{\ds}{\dee s}
\newcommand{\dt}{\dee t}

\newcommand{\Bc}{\mathcal{B}}
\newcommand{\Fc}{\mathcal{F}}

\newcommand{\Ac}{\mathcal{A}}
\newcommand{\Sc}{\mathcal{S}}

\renewcommand{\a}{\alpha}

\newcommand{\e}{\epsilon}

\DeclareMathOperator{\Div}{\mathrm{div}}

\usepackage{bbm}

\newcommand{\ep}{\epsilon}

\renewcommand{\P}{\mathbf{P}}

\newcommand{\embeds}{\hookrightarrow}
\newcommand{\leqc}{\lesssim}

\newcommand{\pd}{\partial}

\newcommand{\dist}{\operatorname{dist}}
\newcommand{\codim}{\operatorname{codim}}

\newcommand{\Leb}{\operatorname{Leb}}

\newcommand{\Gr}{\operatorname{Gr}}

\DeclareMathOperator{\Vol}{\mathrm{Vol}}

\newtheorem{theorem}{Theorem}[section]
\newtheorem{atheorem}{Theorem}

\newtheorem*{theorem*}{Informal Theorem}
\newtheorem{proposition}[theorem]{Proposition}
\newtheorem{corollary}[theorem]{Corollary}
\newtheorem*{corollary*}{Informal Corollary}
\newtheorem{acorollary}[atheorem]{Corollary}
\newtheorem{lemma}[theorem]{Lemma}
\newtheorem*{lemma*}{Lemma}
\newtheorem{claim}[theorem]{Claim}

\theoremstyle{definition}
\newtheorem{definition}[theorem]{Definition}

\newtheorem*{question*}{Question}

\theoremstyle{remark}
\newtheorem{remark}[theorem]{Remark}
\newtheorem{example}[theorem]{Example}

\begin{document}
\maketitle

\begin{abstract}
We consider the top Lyapunov exponent associated to a dissipative linear evolution equation posed on a separable Hilbert or Banach space. In many applications in partial differential equations, such equations are often posed on a scale of nonequivalent  spaces mitigating, e.g., integrability ($L^p$) or differentiability ($W^{s, p}$). In contrast to finite dimensions, the Lyapunov exponent could apriori depend on the choice of norm used. In this paper we show that under quite general conditions, the Lyapunov exponent of a cocycle of compact linear operators is independent of the norm used. We apply this result to two important problems from fluid mechanics: the enhanced dissipation rate for the advection diffusion equation with ergodic velocity field; and the Lyapunov exponent for the 2d Navier-Stokes equations with stochastic or periodic forcing.
\end{abstract}

\section{Introduction}\label{sec:Intro}

Consider the linear evolution equation 
\begin{align}\label{eq:abstractLinearPDE}
\frac{\dee}{\dt} v(t) = L(t) v(t) \, , \quad v(0) = v_0 \, , 
\end{align}
posed on a Banach space $(B, \| \cdot\|_B)$, where $L(t)$ is a time-varying, closed linear operator (potentially unbounded). Let us assume \eqref{eq:abstractLinearPDE} is globally well-posed and it gives rise to an evolution semigroup $S(t)$, namely, a bounded family of solution operators $S(t) : B \to B$, $t \geq 0$, such that $v(t) = S(t) v_0$ is the unique solution to \eqref{eq:abstractLinearPDE} for all fixed initial $v_0 \in B$. 

The Multiplicative Ergodic Theorem (MET) is a powerful tool for characterizing the asymptotic behavior of systems such as \eqref{eq:abstractLinearPDE} in the 
case when $L(t)$ depends on the value of some auxiliary \emph{stationary process}, e.g., when $L(t)$ is random with probabilistic law independent of $t$. In this setting and under some mild 
conditions on the operators $S(t)$, the MET asserts that for `typical' realizations
of $t \mapsto L(t)$, there exists a value $\lambda_1 \in [-\infty, \infty)$ and a finite-codimensional subspace $F \subset B$ such that for all $v_0 \in B \setminus F$
\begin{align}\label{eq:lyapdef}
\lambda_1 = \lim_{t \to \infty} \frac1t \log \|S(t) v_0\|_B \, .
\end{align}
 As $B \setminus F$ is open and dense in $B$, it follows that the growth rate $\lambda_1$ is experienced by `typical' initial $v_0$. For further details and a review of the MET in this setting, see Section \ref{subsec:backgroundMET2} below.

Linear evolution equations such as \eqref{eq:abstractLinearPDE} cover a broad variety of time-dependent dissipative linear PDE. In this paper, we consider the following example settings: 
\begin{itemize}
\item[(i)] the advection diffusion equation for a passive scalar advected by a velocity field evolving according to a ``statistically stationary'' evolution equation, e.g., the 2d Navier-Stokes equations with either time-periodic or stochastic forcing; and 
\item[(ii)] the first variation (linearization) of the 2d Navier-Stokes equations with either time-periodic or stochastic forcing. 
\end{itemize}
We will discuss both of these examples in more detail in Section \ref{subsec:appsIntro} below. Other relevant examples that can be treated (though not discussed) in this setting are: the kinematic dynamo equation governing the advection and diffusion of a magnetic field in a flow, as well as the first variation equation of a wide class of forced dissipative semilinear parabolic problems, including reaction diffusion equations, magnetohydrodynamics (MHD), and various dissipative wave equations (see \cite{Temam1997-xm} and \cite{Chicone1999-mi} for examples and descriptions of these models and more.)

\subsubsection*{Lyapunov exponents along scales of norms}

For globally well-posed linear evolution equations as above, it is common to have global well-posedness on a \emph{scale} of Banach spaces $(B_\alpha, \| \cdot\|_{B_\alpha})$ for $\alpha \in [a,b]\subset \R$ such that $B_\beta$ is embedded in $B_\alpha$ for all $\alpha < \beta$. Such scales of spaces might capture varying degrees of integrability, e.g., $L^p$ spaces, or spatial regularity, e.g., Sobolev spaces $W^{r,p}$ or Besov spaces $B^r_{p,q}$. Each $B_\alpha$ comes equipped with its own norm $\|\cdot\|_{B_\alpha}$ with respect to which one can compute Lyapunov exponents, and so apriori, the same evolution equation \eqref{eq:abstractLinearPDE} might possess an entire range of Lyapunov exponents 
\begin{equation*}
\lambda_1(B_\alpha) = \lim_{t \to \infty} \frac1t \log \|v(t)\|_{B_\alpha}
\end{equation*}
 depending on the scale parameter $\alpha$. 

It is evident from \eqref{eq:lyapdef} that for a {\em finite-dimensional} Banach space, the choice of norm has no effect on the value of $\lambda_1$: in this case, local compactness implies that all norms on $B$ are equivalent up to a multiplicative constant which vanishes under the limit of $\frac1t \log$. Thus it is often said that in the finite dimensional setting, Lyapunov exponents are \emph{intrinsic to the underlying system} in that they do not depend on the choice of norm.
 However, local compactness is false in infinite dimensions, and so 
 it is possible that $\lambda_1(B_\alpha)$ could depend nontrivially on the scale parameter $\alpha$.  
{\it This casts doubt on the `intrinsic' nature of Lyapunov exponents in the infinite dimensional setting}, especially when there is no natural or otherwise physically relevant choice for the space $B_\alpha$. 

The following is an informal statement of the main result of this paper. 
\begin{theorem*}
Let $(B_\alpha)_{\alpha \in [a,b]}$ be a nested family of Banach spaces, each with separable dual. Assume that
\begin{itemize}
	\item[(a)] $B_\beta \subset B_\alpha$ is dense for all $a \leq \alpha < \beta \leq b$;  
	\item[(b)] $S(t) : B_\alpha \to B_\alpha$ is a compact (hence bounded) linear operator for all $\alpha \in [a,b]$; and
	\item[(c)] $\sup_{t \in [0,1]} \| S(t)\|_\alpha$ satisfies a logarithmic moment condition with respect to the stationary law governing $L(t)$. 
\end{itemize}
Then, for all $\beta \in [a,b]$ and $v_0 \in B_{\beta} \setminus \{ 0 \}$, the limit
\[
\lambda(v_0) = \lim_{t \to \infty} \frac1t \log \|S(t) v_0\|_{B_\alpha}
\]
exists and does not depend on $\alpha$. 
\end{theorem*}
This result affirms the idea that {\it for such systems, the Lyapunov exponent is an intrinsic feature of the system}, independent of the choice of norm $\| \cdot\|_{B_\alpha}$. 
For full statements, see Section \ref{subsec:statementOfResults}. See Section \ref{subsec:prevWork} for a literature review of prior work on this topic.

\subsubsection*{Rates at which exponents are realized}

It is also desirable to understand the rate at which Lyapunov exponents are realized between various Banach spaces. In general, there is a period of transient growth or decay before the one sees the exponential rate given by $\lambda_1$. To capture this, we consider the {\em Lyapunov regularity functions} for a given Banach space $B$ and $\ep >0$
\[
    \overline{D}_\ep^B := \sup_{t \geq 0}\frac{\|S(t)\|_B}{e^{t(\lambda_1 +\ep)}},\quad\text{and}\quad \underline{D}_\ep^B := \sup_{t\geq 0}\sup_{v\in B \backslash F\atop \|v\|_B = 1} \frac{e^{t(\lambda_1 -\ep)}\sin\angle^B(v,F)}{\|S(t)v\|_B},
\]
where the angle $\angle^B(v,F)$ is defined via \eqref{defn:angleHilbSpace}, then we have the following upper and lower bounds on $\|S(t)v\|_B$ for $\|v\|_B = 1$,
\[
    (\underline{D}^B_\ep)^{-1}\sin\angle^B(v,F)e^{t(\lambda_1 -\ep )} \leq \|S(t)v\|_B \leq \overline{D}^B_\ep e^{t(\lambda_1 +\ep )}.
\]
By the MET (c.f. Section \ref{subsubsec:MET} below), the quantity $D_\ep^B := \max\{\overline{D}_\ep^B,\underline{D}_\ep^B\}$ is finite almost surely, so the above upper and lower bounds are non-trivial. It follows that if $\sin\angle^B(v,F) = 1$, then
\[
     \left|\frac{1}{t}\log{\|S(t)v\|_B} - \lambda_1\right|\leq  \frac{\log D_\ep^B}{t} +\ep,
\]
so that when $t \sim \ep^{-1}\log D^B_\ep$ we have that $\frac{1}{t}\log{\|S(t)v\|_B}$ is $\ep$ close to $\lambda_1$. This means that $\log{D^B_\ep}$ can be interpreted as the (random) amount of time it takes for the exponent $\lambda_1$ to be realized within an $\epsilon$ error. In other words $\log{D^B_\ep}$ is is roughly the duration of the transient phase. 

It is natural to wonder how $ D^B_\ep$ changes between different Banach spaces. As a corollary of our main theorem, given a dense subspace $V\subset B$ we are able to relate $\overline D^V_{\ep}, \underline D^V_\ep$ to $\overline D^B_\ep, \underline D^B_\ep$ under a slightly stronger regularization condition on $S(t)$. See Section \ref{subsubsec:rate-realized} and Corollary \ref{cor:reguarlityFunctions} in Section \ref{subsubsec:uniformity-sets} for more precise statements on these results.

\subsection{Applications}\label{subsec:appsIntro}

This paper contains several applications of the main result to systems of interest in fluid dynamics in Sobolev spaces. We will discuss these applications and their physical relevance below, deferring detailed statements to Section \ref{sec:Applications}. While several simplifying assumptions are made, e.g., working with periodic domains without boundary, we note that many of the main ideas discussed below remain valid in broader generality. Moreover, while we work with Hilbert regularity spaces $H^s$, much of what we show can also be done in Banach regularity spaces like $W^{s,p}$ and $B^s_{p,q}$, which often carry more precise regularity information.

\subsubsection{Lyapunov exponents for passive scalar advection}\label{subsubsec:passiveScalarIntro}

Let $d > 1$ and let $u:  [0,\infty) \times \T^d \to \R^d$ be a time dependent, incompressible velocity field on the torus $\T^d$, which for the purposes of this discussion will be assumed to be $C^\infty$ in $x$ locally uniformly in $t$. Let $f(t, x)$ be a solution to the passive scalar advection equation
    \begin{align}\label{eq:passiveScalar}
    \partial_t f + u \cdot \nabla f = \kappa \Delta f \, , \quad f(0, x)= f_0(x)
    \end{align}
for a given mean-zero initial scalar $f_0 : \T^d \to \R$. This equation models the advection of the scalar density $f_0$ (e.g., a dilute chemical concentration or small temperature variation) by a fluid with velocity field $u(t, x)$ taking into account molecular diffusivity $\kappa > 0$. Equation \eqref{eq:passiveScalar}  is globally well-posed on the Sobolev space $H^r$ for any $r \in \R$, and so gives rise to a linear (nonautonomous) semiflow $S(t) : H^r \to H^r$ of bounded linear operators. 
Here, $H^r$ is viewed as a Hilbert space of mean-zero functions (or distributions, if $r < 0$) $g : \T^d \to \R$ with the homogeneous Sobolev norm $\|f\|_{H^r} = \|(-\Delta)^{r/2}f\|_{L^2}$ and corresponding inner product $(\cdot, \cdot)_{H^r}$. 

    When $t \mapsto u(t, \cdot)$ evolves according to some ergodic, stationary process, e.g., the Navier-Stokes equations with spatially regular, time-periodic or stochastic forcing, the Lyapunov exponent 
    \[
    \lambda_1(H^r) = \lim_{t \to \infty} \frac1t \log \|f(t)\|_{H^r}
    \]
    exists in $[-\infty,\infty)$ with probability 1 for all sufficiently regular initial velocity profiles $u_0$ and for an open and dense set of initial scalars $f_0$. 
    
    In the mathematical literature on advection diffusion, special interest is often taken in interpretations of the growth or decay of $\| f(t)\|_{H^r}$
    for various values of $r$. When $\kappa =0$, the $H^1$ norm is naturally connected to the strength of shear-straining in the fluid (see \eqref{eq:shearStrain} below), while the $H^{-1}$ norm and other negative Sobolev norms measure the degree to which the scalar $f_t$ has been uniformly ``mixed'' into the fluid -- see, e.g., \cite{mathew2005multiscale, doering2006multiscale, lin2011optimal, shaw2007stirring} -- and are related to the decay of correlations of the Lagrangian flow associated to the velocity $u$. The addition of diffusion $(\kappa >0)$ somewhat complicates these interpretations: when advection generates small scales, diffusion can effect decay of the $L^2$ norm on time scales faster than the diffusive one. This is known in the mathematics literature as {\em enhanced dissipation}; this effect has been studied in 
    both the physics \cite{bernoff1994rapid, dubrulle1994scaling, latini2001transient, lundgren1982strained, rhines1983rapidly} and (somewhat more recent) mathematics literature \cite{BBPS19II, BCZ2017enhanced, beck2013metastability, bedrossian2016invariant, constantin2008diffusion, vukadinovic2015averaging, zlatovs2010diffusion, zelati2020relation, Feng_2019}.

Despite such varied interpretations of the measurement of various norms, in this manuscript we prove the following (see Theorem \ref{thm:allEqualLE} below for a precise statement). 

\begin{theorem*}
Assume $\kappa > 0$ and that $u(t, \cdot)$ is an ergodic, stationary process and $\int_0^1 \| u(t, \cdot) \|_{H^\gamma}\, dt$ satisfies a moment condition for some $\gamma > \frac{d}{2} + 1$. Then, 
$\lambda(H^s)$ exists for all $s \in [-\gamma, \gamma]$ and does not depend on $s$. In particular, 
\begin{align}\label{eq:equalSubspaceLEPSAIntro}
\lambda_1(H^1) = \lambda_1(L^2) = \lambda_1(H^{-1}) \,. 
\end{align}
\end{theorem*}
For a full statement, see Theorem \ref{thm:ratesTheSamePSA} in Section \ref{subsec:passivescalar} below. 
To our knowledge ours is the first proof of this fact for passive scalar advection, although we note that
it has been predicted before, e.g., by numerical evidence in the recent paper \cite{oakley2021mix} as well as in \cite{MC18}. 

\subsubsection*{Additional discussion and context}

\noindent {\bf The case $\kappa = 0$. }Key to the validity of \eqref{eq:equalSubspaceLEPSAIntro} is compactness of the solution linear operators $S^t : H^r \to H^r$ for \eqref{eq:passiveScalar} when $\kappa > 0$. A clear example is provided
in the case $\kappa = 0$; in this case, \eqref{eq:passiveScalar} is still globally well-posed on $H^r$ for all $r$, and by the method of characteristics one has
\[
f(t) = S(t) f_0 = f_0 \circ ( \varphi^t)^{-1}
\]
for all $f_0 \in L^2$, where $\varphi^t:\T^d \circlearrowleft$ is the Lagrangian flow associated to the velocity field $u$. 
Note that in this case, $S(t) : L^2 \to L^2$ cannot be compact, as it is unitary: $(S(t) f, S(t) g)_{L^2} = (f, g)_{L^2}$ for all $f, g \in L^2$ by incompressibility. In particular, 
for the $L^2$ Lyapunov exponent, $\| f(t)\|_{L^2} = \| f_0\|_{L^2}$ for all $t \geq 0$, hence $\lambda_1(L^2) = 0$. On the other hand, 
\begin{align}\label{eq:shearStrain}
\| f(t)\|_{H^1} = \| (D \varphi^t)^{- T} \nabla f\|_{L^2} \, 
\end{align}
by incompressibility. When the Lagrangian flow $\varphi^t$ associated to $u$ has a positive Lyapunov exponent on a positive volume, i.e., 
\begin{align}\label{posLELagrangian}
\Leb\left\{ x \in \T^d : \limsup_{t \to \infty} \frac{1}{t} \log \| D_x \varphi^t\| > 0 \right\} > 0 \, , 
\end{align}
 then $\| f(t)\|_{H^1}$ can grow exponentially fast, hence $\lambda_1(H^1) > 0$. 
It was recently shown\footnote{Despite a wealth of numerical evidence, in the absence of noise it is a notoriously challenging open problem to prove positivity of Lyapunov exponents for incompressible systems of practical interest. 
This is already the case for low-dimensional discrete-time toy models \cite{CrisantiEtAl1991} of Lagrangian flow such as the Chirikov standard map \cite{chirikov1979universal}, for which the analogue of \eqref{posLELagrangian} is a wide-open problem -- see, e.g., the discussion in \cite{crovisier2019problem, blumenthal2017lyapunov}.} by the authors and J. Bedrossian in \cite{bedrossian2018lagrangian} that when $u$ solves the stochastic Navier-Stokes equations with nondegenerate, white-in-time forcing, the LHS of \eqref{posLELagrangian} has full Lebesgue measure with probability 1. 

\medskip

\noindent {\bf $L^2$ and $H^1$ decay rates. }When $\kappa > 0$, standard heat equation energy estimates for the $L^2$ norm immediately imply
\[
\lambda_1(L^2) \leq - \kappa < 0\,,
\]
and therefore our result implies the same holds true for $\lambda_1(H^1)$. At first glance, this might be surprising in light of the tendency of \eqref{eq:passiveScalarEq} to form large gradients. It is however consistent with the energy estimate 
\[
    \int_0^\infty \kappa \|\nabla f(s)\|_{L^2}^2\ds <\infty,
\]
requiring time integrability of $\|f_t\|_{H^1}$ over $[0,\infty)$ for $\kappa >0$. 
We emphasize, though, that $\lambda_1(H^1) < 0$ refers only to \emph{time asymptotic behavior}, and does not rule out transient growth in $H^1$ on some ($\kappa$-dependent, potentially quite long) time scale after which the diffusion dominates (see Sections \ref{subsubsec:rate-realized} and \ref{subsubsec:uniformity-sets} for more discuss on this)

\medskip

\noindent {\bf $H^{-1}$ decay rates. } In
\cite{bedrossian2018lagrangian, bedrossian2019almost, BBPS19II}, the authors proved the following exponential decay estimate for \eqref{eq:passiveScalarEq} when $u$ solves the 2d stochastic Navier-Stokes equation (or any of a large class of noisy evolution equations): 
\[
    \|f(t)\|_{H^{-1}} \leq D_\kappa e^{-\gamma t}\|f_0\|_{H^1},
\]
 where the deterministic constant $\gamma > 0$ is {\em independent} of $\kappa$ and the random variable $D_\kappa \geq 1$ has $\kappa$-independent expectation. Using that $S(t)$ instantly regularizes $H^{-1}$ to $H^1$ for $t>0$, this readily implies that
 \[
 \lambda_1(H^{-1}) \leq -\gamma < 0 \,. 
 \]
 In light of our main result and the results of \cite{BBPS19II}, we conclude that, for the stochastic Navier-Stokes equations and related models, {\em all} Sobolev norms (including $L^2$) eventually decay no slower than the uniform-in-$\kappa$ exponential decay rate $\gamma>0$ (perhaps after an initial $\kappa$-dependent period of  transient growth). 
This $\kappa \to 0$ singular limit bears a striking similarity to the stochastic stability of the so-called Ruelle-Pollicott resonances associated to stationary hyperbolic flows \cite{Dyatlov2015-uj}.

\begin{remark}
We emphasize that we are not the first to apply the MET and related ideas to passive scalar advection. 
Froyland et al. have developed data-driven algorithms for identifying coherent structures in incompressible fluids \cite{froyland2010coherent}, with applications in the forecasting of oceanic features such as persistent gyres in the Atlantic ocean \cite{van2012origin}. 
Justifying the use of these algorithms required extending the MET for compositions of possibly noninjective linear operators, addressed in \cite{froyland2010semi} in finite dimensions and, e.g., \cite{gonzalez2014semi} in infinite dimensions. Additional 
applications of the MET in this vein include the exploration of almost-sure statistical
properties for random compositions of mappings \cite{dragivcevic2018spectral, dragivcevic2020spectral}. 
\end{remark}

\subsubsection{Lyapunov exponents for the Navier-Stokes equations}

Let $u(t, x)$ be a mean-zero divergence free velocity field solving the Navier-Stokes equations on the periodic box $\mathbb T^2$, 
    \[
    \partial_t u + (u \cdot \nabla) u = \nu \Delta u - \nabla p  + F \, , \quad \Div u = 0 \, ,
    \]
    where $F$ is some spatially smooth, white-in-time or time-periodic forcing term, $\nu > 0$ is fixed, and $p$ denotes the pressure that enforces the divergence-free condition. Under appropriate conditions on the forcing, for all $r \geq 0$ this nonlinear evolution equation gives rise to a stochastic semiflow of $C^1$ Frechet-differentiable mappings $\Phi^t_\omega : H^r \to H^r$, where $H^r$ denotes the Sobolev space of $H^r$ (weakly) divergence-free fields ($r=0$ corresponding to $L^2$). Here, $\omega$ denotes the history of the driving path. Given an initial $u_0 = u(0, \cdot)$ and an initial divergence-free $v_0 \in H^r$, the $v_0$ derivative 
    $v_t = (D_{u_0} \Phi^t)v_0$ solves the linearized Navier-Stokes equations
    \begin{equation}\label{eq:linearized-ns}
    \partial_t v + (u\cdot\nabla) v + (v\cdot\nabla) u =  \nu \Delta v -\nabla q\,, \quad \Div v = 0\,,
    \end{equation}
    with initial data $v_0$. When $F$ is deterministic and time-periodic or when $F$ is stochastic and white-in-time, the MET applies: under mild additional conditions, the $H^r$ Lyapunov exponent\footnote{When 
$F$ is white-in-time and satisfies mild nondegeneracy conditions (e.g., those in \cite{HM06}), 
the value $\lambda_1(H^r)$ does not depend on $u_0$. When $F$ is 
time-periodic it is possible that $\lambda_1(H^r)$ depends on $u_0$. For more details and discussion, see {Section \ref{sec:Applications}} below. }
\[
\lambda_1(H^r) = \lim_{t \to \infty} \frac1t \log \| v_t\|_{H^r}
\]
exists with probability 1 and for `typical' initial velocity fields $u_0$, where $v_0$ is drawn from an open and dense subset of $H^r$. Since we are working with a first variation equation, 
the value $\lambda_1(H^r)$ represents the asymptotic exponential
rate at which nearby trajectories converge ($\lambda_1 < 0$) or diverge ($\lambda_1 > 0$) in the $H^r$ norm as time progresses.

In the study of the 2d Navier-Stokes equations it is often useful to formulate the equation in terms of vorticity $w = \mathrm{curl}\,u$, 
\[
    \partial_t w + (u\cdot\nabla) w = \nu \Delta w + \mathrm{curl}\, F \, , 
\]
or in terms of the stream function $\psi = \Delta^{-1}\mathrm{curl}\,u$, which is the Hamiltonian for the velocity field $u = \nabla^\perp \psi = (-\partial_y\psi,\partial_x\psi)$. Depending on the variable considered, it is natural to study the associated growth of the perturbation in $L^2$ of the associated variable. Hence, measuring the linearization in $H^{1},L^2$ or $H^{-1}$ corresponds to measuring the linearization in $L^2$ for the vorticity, velocity, or stream function formulations of the equation. 

In the inviscid case ($\nu =0$), it is known that the stability of the equation is strongly dependent on the whether one is considering $L^2$ of vorticity, velocity or the stream function with some perturbation being stable in $L^2$ of the stream function, but not in $L^2$ of velocity or vorticity due to the generation of high-frequencies due to mixing effects.
For the viscid problem $\nu > 0$, we prove in this paper the following: 

\begin{theorem*}
Assume that $u(t, \cdot)$ solve the Navier-Stokes equations with forcing $F$ (either time-periodic or white-in-time) and that the resulting process on velocity fields is (statistically) stationary and ergodic. Assume $\int_0^1 \| u(t, \cdot) \|_{H^{\gamma + 2}}\, \dt$ has finite moments for some $\gamma > 2$. Then, $\lambda(H^s)$ exists for all $s \in [- \gamma + 1, \gamma + 1]$ and does not depend on $s$. In particular, 
\[
\lambda_1(H^1) = \lambda_1(L^2) = \lambda_1(H^{-1}) \,.
\]
\end{theorem*}
For full details, see Theorem \ref{thm:ratesTheSameNSE} in Section \ref{subsec:NavierStokes}. 

There is a long and extensive literature on the linear stability or instability of stationary (time independent) solutions to the Euler and Navier-Stokes equations; see, e.g., the textbooks \cite{chandrasekhar2013hydrodynamic, drazin2004hydrodynamic, schmid2002stability, yaglom2012hydrodynamic, I_Udovich1989-jm}. Lyapunov exponents, which can be viewed as analogous to spectra for nonstationary flows, have been employed extensively in the study of semilinear parabolic problems such as Navier-Stokes, for instance in providing upper bounds on the dimension of the global attractors-- see, e.g., \cite{temam2012infinite, foias2001navier, constantin1983global}. 
Ruelle and Takens proposed dynamical chaos, of which a positive Lyapunov exponent is a natural hallmark, as a mechanism involved in the 
transition to turbulence \cite{ruelle1971nature, mackay1991appraisal}. To this end, Ruelle established an extension of smooth ergodic theory to dissipative parabolic problems such as Navier-Stokes \cite{ruelle1982characteristic} (see also, e.g., \cite{lian2010lyapunov, blumenthal2017entropy, lu2013strange}). For numerical studies of Lyapunov exponents in turbulent regimes, see, e.g., \cite{crisanti1993intermittency, yamada1988lyapunov}. 

\subsubsection*{Plan for the paper}

Section \ref{sec:Results} covers necessary background from ergodic theory and a full statement of our main abstract result, Theorem \ref{thm:allEqualLE}, the full proof of which is given in Section \ref{sec:METsameExp}. Full statements and proofs of the assertions in Section \ref{subsec:appsIntro} above are given Section \ref{sec:Applications}. 

\subsubsection*{Acknowledgments} AB was supported by National Science Foundation grant DMS-2009431. SPS was supported by National Science Foundation grant DMS-2205953. SPS is also grateful to the Institute for Advanced Study for their generous support and hospitality during 2021-2022 academic year when this paper was being written.

\section{Abstract setting and statement of results}\label{sec:Results}

\subsection{Background on the Multiplicative Ergodic Theorem (MET)}\label{subsec:backgroundMET2}

The MET is a theorem in ergodic theory, the study of measure-preserving 
transformations (mpt's) of a probability space. Here we briefly 
recall a few basic definitions and the statement of the MET itself, and will
afterwards provide the full statement of our main result. Additional context and a brief review of literature is given at the end of Section \ref{subsec:backgroundMET2}. 

\subsubsection{Setting}\label{subsubsec:ergTheoryBackground}

Let $(X, \mathscr F, m)$ be a probability space: here $X$ is a set, $\mathscr F$ a $\sigma$-algebra of subsets of $X$, and $m$ a probability measure. We say that a measurable 
transformation $T : X \to X$ (possible noninvertible) is an {\em mpt}
if $m \circ T^{-1} = m$, i.e., $m(T^{-1} A) = m(A)$ for all $A \in \mathscr F$. 
We can interpret the \emph{invariant measure} $m$ as characterizing ``equilibrium statistics'' for the 
dynamics described by $T$: if $x_0$ is an $X$-valued random variable with law $m$, and given any observable $\varphi : X \to \R$, then the random variables
\[
\varphi(x_0), \, \varphi \circ T(x_0), \, \cdots, \, \varphi \circ T^k(x_0), \,\cdots
\]
all have the same law, i.e., $\{ \varphi \circ T^k(x_0)\}_{k \geq 0}$ a stationary sequence. 

We say that $T : (X , \mathscr F, m) \circlearrowleft$ is \emph{ergodic} if, for any $A \in \mathscr F$, the invariance relation $T^{-1} A = A$ implies $m(A) = 0$ or $1$. 
Ergodicity is a form of irreducibility: 
the phase space $X$ cannot be partitioned into two
pieces of positive $m$-mass which never exchange trajectories.

\begin{example}\label{exam:simpleInvMeasSetup}
Let $B$ be a separable Banach space and let $T : B \to B$ be a continuous mapping, e.g., the time-1 solution mapping to a possibly nonlinear, well-posed evolution equation on $B$. If $\mathcal A \subset B$ is a compact, $T$-invariant subset\footnote{ We call $\mathcal A$ a $T$-invariant set if $T^{-1} \mathcal A \supset \mathcal A$.}, e.g., a global attractor for $T$, then there exists at least one $T$-invariant, Borel probability measure supported on $\mathcal A$. Indeed, for any fixed $x_0 \in \mathcal A$, any subsequential weak$^*$ limit of the sequence\footnote{That such weak$^*$ limits exist follows by compactness of $\mathcal A$. That such limiting measures are $T$-invariant is straightforward to check: see, e.g., Lemma 2.2.4 of \cite{viana2016foundations}. The above procedure is often referred to as the \emph{Krylov-Bogolyubov argument} for the existence of $T$-invariant measures \cite{kryloff1937theorie}. }
\[
\frac1n \sum_{i = 0}^{n-1} \delta_{T^i x_0} 
\]
is $T$-invariant. By standard arguments\footnote{E.g., Proposition 4.3.2 of \cite{viana2016foundations} and the Krein-Milman Theorem, paragraph I.A.22 in \cite{wojtaszczyk1996banach}.}, it follows that there also exist \emph{ergodic} $T$-invariant probability measures supported on $\mathcal A$. 
\end{example}

\subsubsection{The MET}\label{subsubsec:MET}

Let $B$ be a separable Banach space with norm $\| \cdot\|_B$. The MET concerns \emph{cocycles} of operators, which for our purposes
are compositions of the form 
\[
A^n_x = A(T^{n-1} x) A(T^{n-2} x) \cdots  A(T x)  A(x) \,, \quad n \geq 1, x \in X \, , 
\]
where $A : X \to L(B)$, $L(B)$ the space of bounded operators on $B$, is the \emph{generator} of the cocycle. We view
the composition $A^n_x$ as being ``driven'' by the dynamics $T : X \to X$. 
The MET describes the asymptotic exponential growth rates
\begin{align}\label{eq:defnFTLE}
\lambda(x, v) := \lim_{t \to \infty} \frac1t \log\|A^n_x v\|_B \, ,
\end{align}
where they exist, 
as $x$ ranges over $m$-typical initial conditions in $X$ and $v \in B \setminus \{ 0 \}$. 
For simplicity, we assume below that $A_x$ is compact for all $x \in X$; otherwise, we make no additional assumptions, e.g., on the injectivity of $A_x$ (we follow the convention that 
$\log 0 = -\infty$). 
Many proofs of the MET in this setting exist; the following is taken from \cite{lian2010lyapunov}; see also \cite{schaumloffel1991multiplicative}. 

\begin{theorem}[MET for compact cocycles]\label{thm:nonInvMET}
Let $T : (X, \mathscr F, m) \circlearrowleft$ be an mpt. Assume that $A : X \to L(B)$ is 
strongly measurable\footnote{When $B$ is separable, we say that $x \mapsto A_x$ is strongly measurable if it is Borel measurable w.r.t. the strong operator topology on $L(B)$, or equivalently, when 
$x \mapsto A_x v$ is a Borel measurable mapping for each fixed $v \in B$. For a summary of alternative measurability requirements for the MET, see, e.g., \cite{varzaneh2021oseledets}. } and that $A(x)$ is a compact linear operator on $B$ for all $x \in X$. Lastly, assume the log-integrability condition
\begin{equation}\label{eq:log-int-cond}
\int \log^+ \| A(x)\|_B \,\dee m(x) < \infty \,. 
\end{equation}
Then, for every $\lambda_c > -\infty$, there exists a (i) function $r_{\lambda_c} : X \to \Z_{\geq 0}$; (ii) for each $i \geq 1$, a function $\lambda_i : \{x\,:\, r_{\lambda_c}(x) \geq i\} \to \R$ satisfying
\[
\lambda_1(x) > \cdots > \lambda_{r_{\lambda_c} (x)} (x) \geq  \lambda_c \, ; 
\]
and (iii) at $m$-a.e. $x \in X$ a filtration 
\[
B =: F_1(x) \supsetneq F_2(x) \supsetneq \cdots \supsetneq F_{r_{\lambda_c}(x) }(x) \supsetneq  \bar F_{\lambda_c}(x) 
\]
by closed, finite-codimensional, measurably varying\footnote{Throughout we consider the space of closed subspaces of $B$ with the Hausdorff metric $d_{Haus}$ of unit spheres; see \eqref{eq:defnHausDist} for details. Here, we are asserting that $x \mapsto F_i(x)$ is Borel measurable w.r.t. the topology induced by $d_H$.} subspaces $F_i(x), \bar F_{\lambda_c}(x)$ 
such that 
\begin{align*}
\lambda(x, v) &= \lambda_i(x) \quad \text{ for all } 1 \leq i \leq r_{\lambda_c}(x)-1 \text{ and } v \in F_i(x) \setminus F_{i + 1}(x) \, , \\
\lambda(x, v) &= \lambda_{r_{\lambda_c}(x)}(x) \quad \text{ for all } v \in F_{r_{\lambda_c}(x)}(x) \setminus \bar F_{\lambda_c}(x) \, , 
\end{align*}
and
\[
\lim_n \frac1n \log \| A^n_x|_{\bar F_{\lambda_c}} \|_B \leq \lambda_c 
\]
for $m$-a.e. $x \in X$. 

The functions $r_{\lambda_c}(x), \lambda_i(x)$ are constant along $m$-a.e. trajectory, as are the codimensions $M_i(x) := \codim F_{i+1}(x) < \infty$. Moreover, when $T : (X, \mathscr F, m) \circlearrowleft$ is ergodic, $r_{\lambda_c}$ and the values 
$\lambda_1, \cdots, \lambda_{r_{\lambda_c}}$ are constant over $m$-a.e. $x \in X$. 
\end{theorem}

It is immediate that the $F_i(x)$ are invariant\footnote{Some authors refer to the property \eqref{eq:equivariant} as \emph{equivariance}.} in the sense that
\begin{align}\label{eq:equivariant}
A_x (F_i(x)) \subset F_i(T x) \quad \text{ for } m-\text{a.e. x.}
\end{align}
Note that we allow the inclusion to be strict. Observe also that 
\begin{align}\label{eq:multDefn}
d_i(x) := M_i(x) - M_{i - 1}(x) = \codim F_{i+1}(x) - \codim F_i(x)
\end{align}
 is the codimension of $F_{i + 1}(x)$ in $F_i(x)$; we refer to $d_i(x)$ as the \emph{multiplicity} of $\lambda_i(x)$. 

\subsubsection{Lyapunov exponents}\label{subsubsec:LE2}

The values $\{\lambda_i\}$ are called \emph{Lyapunov exponents}, while the collection of them
is referred to as the \emph{Lyapunov spectrum}, in analogy with the spectrum of a single closed operator. The value $\lambda_c$ is a cutoff, past which we do not resolve the spectrum further, while adjusting the value $\lambda_c$ lower can potentially `uncover' additional Lyapunov spectrum (i.e., $r_{\lambda_c}$ increases as $\lambda_c$ decreases). Define
\begin{align}\label{eq:defineR2}
r(x) = \lim_{\lambda_c \to -\infty} r_{\lambda_c}(x) = \sup_{\lambda_c \in \R} r_{\lambda_c}(x) \in \Z_{\geq 0} \cup \{ \infty\} \, . 
\end{align}

To simplify the discussion below, assume $T : (X, \mathscr F, m) \circlearrowleft$ is ergodic, so that $r$ and the $\{ \lambda_i\}$ are constants. We distinguish three scenarios: 
\begin{itemize}
	\item[(a)] No Lyapunov exponents are uncovered ($r_{\lambda_c} = 0$ for all values of cutoff $\lambda_c$). In this case, 
	\[
	\lambda(x, v) = -\infty
	\]
	for a.e. $x \in X$ and all $v \in B$. When this occurs, we follow the convention 
	$\lambda_1 = -\infty$, $F_1(x) :=B$, $r = 0$. 
	\item[(b)] Finitely many Lyapunov exponents $\lambda_1 > \cdots > \lambda_r > -\infty, r \in \Z_{\geq 1}$ are uncovered. Each exponent corresponds to a member of the filtration
	\[
	B =: F_1(x) \supsetneq F_2(x) \supsetneq \cdots \supsetneq F_r(x) \supsetneq  F_{r + 1}(x) \supset \{ 0 \}
	\]
	such that $\lambda(x, v) = \lambda_i$ for all $i \leq r, v \in F_{i}(x) \setminus F_{i + 1}(x)$, while $\lim_n \frac1n \log \| A^n_x|_{F_{r + 1}(x)}\| = -\infty$. 
	In this case, we follow the convention $\lambda_{r + 1} = -\infty$. 
	\item[(c)] Infinitely many Lyapunov exponents $\lambda_1 > \lambda_2 > \cdots$ are uncovered. In this case, compactness of $A(x), x \in X$ (see, e.g., discussion after Corollary 2.2 in \cite{ruelle1982characteristic}) implies that $\lim_i \lambda_i = -\infty$, and each exponent corresponds to a member of the filtration
	\[
	B =: F_1(x) \supsetneq F_2(x) \supsetneq \cdots \supsetneq F_i(x) \supsetneq \cdots 
	\]
	for which $\lambda(x, v) = \lambda_i$ for all $i \geq 1, v \in F_i(x) \setminus F_{i + 1}(x)$. The (possibly trivial) closed space $F_\infty(x) := \cap_i F_i(x)$ has the property that $\lim_n \frac1n \log \| A^n_x|_{F_\infty(x)}\| =-\infty$. In this case we follow the convention $r = \infty$. 
\end{itemize}
We note that in all three scenarios, the codimension $\codim F_i$ is constant along trajectories $\{ T^k x\}_{k \geq 0}$, while if $(T, m)$ is ergodic, $\codim F_i(x)$ is constant $m$-almost surely. 

\begin{remark}\label{rmk:nonergodic2}
When $T : (X, \mathscr F, m) \circlearrowleft$ is nonergodic, the limiting value $r(x)$ in \eqref{eq:defineR2} depends on $x \in X$. In particular, $X$ can be subdivided into 
the $T$-invariant (possibly empty) sets $\{ r(x) = 0\}, \{ 1 \leq r(x) <\infty\}$ and $\{ r(x) = \infty\}$ along which each of scenarios (a) -- (c) holds, respectively. 
\end{remark}

\begin{remark}
These scenarios are analogous to the situation for the spectrum $\sigma(K)$ of a 
compact linear operator $K$ on $B$: (a) when $\sigma(K) = \{ 0\}$ (e.g., $K$ is a compact shift operator); (b) when $\sigma(K)$ is a finite set containing $\{0 \}$ (e.g., $K$ is finite rank); and (c) when $\sigma(K)$ is countable
and accumulates only at $\{ 0 \}$ (e.g., $K = \Delta^{-1}$ is the inverse Laplacian on $L^2([0,1])$ with Dirichlet boundary conditions). 
Indeed, when $A_x \equiv K$ is a fixed compact operator not depending on $x$, the
$\lambda_i$ are precisely the logarithms of the absolute values of the elements of $\sigma(K)$, while the $F_i$ are direct sums of the corresponding generalized eigenspaces. 
\end{remark}

\begin{example}\label{exam:simpleDiffSetup}
Let $T : B \to B$ be a continuous mapping as in Example \ref{exam:simpleInvMeasSetup} admitting a compact invariant set $\mathcal A \subset B$ and an invariant Borel probability $m$. Assume in addition that $T$ is $C^1$ Frechet differentiable, and that the derivative $D_x T$ is a compact linear operator (as is the case for a broad class of dissipative parabolic evolution equations \cite{temam2012infinite}). Theorem \ref{thm:nonInvMET} applies to the cocycle 
generated by $A(x) = D_x T \in L(B)$ (note that by our assumptions, $x \mapsto \log^+ |D_x T|$ is a continuous function and $\mathcal A$ is compact, so \eqref{eq:log-int-cond} holds automatically). It follows that for $m$-a.e. $x \in X$ and for all $v \in B$, the limit
\[
\lambda(x, v) = \lim_{n \to \infty} \frac1n \log \| D_x T^n v\|_B \in [-\infty, \infty)
\]
exists, and if finite, equals one of the values $\lambda_i(x)$. 
\end{example}

\subsubsection{Rate at which Lyapunov exponents are ``realized''}\label{subsubsec:rate-realized}

The MET guarantees convergence of the exponential rates $\lambda(x, v)$ for $x \in X, v \in B$ as in equation \eqref{eq:defnFTLE}, but this convergence can be badly nonuniform in $x \in X$. 
While little can be said at this level of generality, we can at least quantify this nonuniformity as we show below. 

To fix ideas, assume $(T, m)$ is ergodic and $r > 0$ (scenarios (b) or (c) in Section \ref{subsubsec:LE2}). Fix an $i \in \{ 1, \dots, r\}$, so by the statement of the MET we have that $\lambda(x, v) = \lambda_i$ for all $v \in F_{i}(x) \setminus F_{i + 1}(x)$. 
With additional work, it is possible to show (see \cite{blumenthal2017entropy}) that for any $\epsilon > 0$, one has
\begin{align}\label{eq:overlineD}
\| A^n_x v\|_B \leq \overline{D}_\e(x) e^{n (\lambda_i + \e)} \| v \|_B \, , 
\end{align}
where
\[
\overline{D}_\epsilon(x) : = \sup_{n \geq 0} \frac{\| A^n_x|_{F_i(x)}\|_B}{e^{n (\lambda_i + \epsilon)}} 
\]
is finite for $m$-a.e. $x \in X$.

For corresponding lower bound, note that the convergence of $\lambda(x, v)$ to $\lambda_i$ should be slower as $v$ approaches $F_{i + 1}(x)$. To account for this, given $v \in B \setminus \{0\}$ and a closed subspace $F \subset V$, write $\angle^B(v, F) $ for the unique ``angle'' in $[0,\pi/2]$ such that 
\begin{align}\label{defn:angleHilbSpace}
\sin \angle^B(v, F) = \inf_{w \in F}\frac{\|v - w\|_B}{\|v\|_B}\,. 
\end{align}
Then, 
\begin{align}\label{underlineD}
\frac{\| A_x^n v\|_B}{\| v \|_B} \geq (\underline D_\e(x))^{-1} e^{n (\lambda_i - \epsilon)} \sin \angle^B(v, F_{i + 1}(x)) \, , 
\end{align}
where
\[
	\underline{D}_\e(x) : = \sup_{n \geq 0} \sup_{\substack{v \in B \setminus F_{i+1}(x) \\ \| v \|_B = 1}} \frac{e^{n (\lambda_i - \epsilon)} \sin\angle^B(v, F_{i+1}(x)) }{\| A^n_x v\|_B} 
\]
is again $m$-almost surely finite. 

Define 
\[
D_\e := \max\{ \overline{D}_\e, \underline{D}_\e\} \, , \quad \text{ and } \quad \Gamma_\ell := \{ D_\e \leq \ell\}, \, \ell > 1 \,. 
\]
The sets $\Gamma_\ell \subset X$ are sometimes referred to as \emph{uniformity sets} or \emph{Pesin sets}: when $\epsilon$ is chosen sufficiently small, for any fixed $\ell > 1$ we have 
\[
\| A^n_x v\|_B \approx_\ell e^{n \lambda_i} \| v \|_B
\]
uniformly over all $x \in \Gamma_\ell$ and $v \in F_i(x) \setminus F_{i + 1}(x)$ with $\angle^B(v, F_{i + 1}(x))$ bounded away from $0$, up to the multiplicative 
constant $\ell$ and ignoring the slowly-growing factors $e^{n \epsilon}$. 
This can be very useful, e.g., in smooth ergodic theory where the exponential expansion/contraction along various directions of $B$ is used to construct stable/unstable manifolds of smooth systems (see references below). Unfortunately, despite their importance, little else can be said about $D_\e$ without additional assumptions. 

\subsubsection{Additional background and context for the MET}

The MET for stationary compositions of $d \times d$ matrices was first proved by Oseledets \cite{oseledets1968multiplicative} in the late 60's, although investigations on the properties of IID products of $d \times d$ matrices
date from the early 60's \cite{furstenberg1960products, furstenberg1963noncommuting}.  
There are now many proofs available-- see, e.g., \cite{raghunathan1979proof, ruelle1979ergodic, walters1993dynamical} and \cite{froyland2010semi}. Since then the MET has been extended 
in several directions, e.g., to the asymptotic behavior of random walks on semisimple Lie groups \cite{kaimanovich1989lyapunov} and on spaces of nonpositive
curvature \cite{karlsson1999multiplicative}. 

One of the most significant impacts of the MET has been in smooth ergodic theory, 
the study of the ergodic properties of differentiable mappings. 
For such systems, the MET implies
the existence of stable and unstable subspaces in the moving frames along 
``typical'' trajectories of the dynamics. Pesin discovered \cite{pesin1977characteristic, ruelle1979ergodic} soon after that these could be used in the construction of 
stable and unstable manifolds, generalizing directly from the classical
theory of stable/unstable manifolds for equilibria and periodic orbits.
This development is at the core of our contemporary understanding of 
chaotic dynamical systems and the fractal geometry of strange attractors
\cite{eckmann1985ergodic}. For more discussion, see, e.g., the the textbook \cite{barreira2002lyapunov}
or the surveys \cite{wilkinson2017lyapunov, young2002srb, pesin2010open}. 

A part of Ruelle's work in \cite{ruelle1982characteristic} was a version of the MET for 
stationary compositions of Hilbert space operators. 
By now, there are many works extending the MET to various infinite-dimensional settings. 
Highlights include extensions to stationary products of compact linear operators on a Banach space \cite{mane1983lyapounov}, dropping the compactness assumption \cite{thieullen1987fibres}, versions of the MET suited to the first variation equations of SPDE \cite{schaumloffel1991multiplicative, lian2010lyapunov},
and a version of the MET for compositions of operators drawn from a vonNeumann algebra \cite{bowen2021multiplicative}. See also, e.g., \cite{gonzalez2014concise, blumenthal2015volume, varzaneh2021oseledets}.

\subsection{Statement of main results}\label{subsec:statementOfResults}

To start, we will assume
\begin{itemize}
\item[(a)] $(B, \| \cdot\|_B)$ is a Banach space and $V \subset B$ is a dense subspace; 
\item[(b)] The space $V$ is equipped with its own norm $\| \cdot \|_V$ such that
\[
\| \cdot\|_B \leq \| \cdot\|_V \, ; 
\]
\item[(c)] Both $(B, \| \cdot\|_B)$ and $(V, \| \cdot\|_V)$ have separable duals. In particular, $(B, \| \cdot\|_B)$ and $(V, \| \cdot\|_V)$ are separable by a standard argument. 
\end{itemize}
Additionally, 
\begin{itemize}
	\item[(1)] $T : (X, \mathscr F, m) \circlearrowleft$ is an mpt; 
	\item[(2)] $A : X \to L(B)$ is strongly measurable and $A(x) \in L(B)$ is compact for all $x \in X$; 
	\item[(3)] The restriction $A(x)|_V$ has range contained in $V$ and is a compact linear operator $V \to V$, and moreover, $x \mapsto A(x)|_V$ is strongly measurable $X \to L(V)$; and finally, 
	\item[(4)] The operator $A(x)$ satisfies the log-integrability condition regarded on both $(B, \| \cdot\|_B)$ and $(V, \| \cdot \|_V)$, i.e., 
	\begin{equation}\label{eq:regularizingCond}
	\begin{aligned}
		&\int \log^+\|A(x)\|_B \dee m(x) < \infty \\
		&\int \log^+ \| A(x)|_V \|_V \dee m(x) < \infty
	\end{aligned}
	\end{equation}
\end{itemize}
Under (1) -- (4), the MET as in Theorem \ref{thm:nonInvMET} applies to $A^n_x$
regarded as a cocycle on either $B$ or $V$. Below, for either of $W = B$ or $V$ 
we write $\lambda_i^W, r^W$ and $F_i^W$ for the objects in Theorem \ref{thm:nonInvMET}
applied to $A^n_x$ regarded as a cocycle on $W$. 

\begin{atheorem}\label{thm:allEqualLE}
Under (1) -- (4), we have that $r^V(x) = r^B(x)$ for $m$-a.e. $x \in X$. Writing $r(x)$ for this common value, the following holds for all $i \geq 1$ and $m$-a.e. $x \in \{ r \geq i\}$: 
 \[
\lambda_i^V (x) = \lambda_i^B (x)  \quad \text{ and } \quad F_i^V(x) = F_i^B(x) \cap V \,. 
\]

\end{atheorem}

In particular, for $m$-a.e. $x \in X$ and all $v \in V$, we have that
\[
\lim_{n \to \infty} \frac1n \log \| A^n_x v\|_B = \lim_{n \to \infty} \frac1n \log \| A^n_x v\|_V \, . 
\]
 Note also that scenarios (a), (b) and (c) above are carried over from $B$ to $V$. For instance, in the ergodic case, $\lambda_1^B = -\infty, r^B = 0$ (our convention for scenario (a)) holds if and only if $\lambda_1^V = -\infty, r^V = 0$. 

\begin{example} \ 
\begin{itemize}
\item[(i)] Let $T : B \to B$ be a $C^1$ Frechet differentiable mapping with compact invariant set $\mathcal A \subset B$ as in Examples \ref{exam:simpleInvMeasSetup} and \ref{exam:simpleDiffSetup}. Assume $V \subset B$ is a dense embedded subspace, $\| \cdot \|_V \geq \| \cdot\|_B$. If $A(x) = D_x T$ satisfies assumptions (1) -- (4) above, then Theorem \ref{thm:allEqualLE} applies: writing $\lambda^W(x, v) = \lim_{n \to \infty} \frac1n \log \| D_x T^n v\|_W$ for $x \in \Ac, v \in B$ and $W = V$ or $B$, it follows that 
\[
\lambda^V(x,v) = \lambda^B(x,v) \quad \text{ for } m\text{ a.e. } x \in \Ac, \text{ and all } v \in B \,. 
\]
\item[(ii)] When $T$ is the time-1 mapping for a dissipative semilinear parabolic problem, e.g., the 2d Navier-Stokes equations, one typically works with a scale of Banach spaces $B_\alpha, \alpha \in [a,b] \subset \R, B_\beta \subset B_\alpha$ for $\alpha< \beta$, e.g., the Sobolev spaces $B_\alpha = W^{\alpha, 2} = H^\alpha$. Fixing $B = B_\alpha$ and $V = B_\beta$, it is often the case that $D_x T : B \to B$ is compact, and that $D_x T$ restricted to $V$ maps into $V$ and is similarly compact;  this is a consequence of \emph{parabolic regularity} for the first variation (linearization) equation. In particular, Theorem \ref{thm:allEqualLE} applies. See Section \ref{sec:Applications} for more details in the case of 2d Navier-Stokes. 
\end{itemize}
\end{example}

\subsubsection{Comparison of uniformity sets}\label{subsubsec:uniformity-sets}

It is natural to attempt to compare the rate at which Lyapunov exponents are realized
between the norms of $B$ and $V$. 
Fix $i \geq 1$ and assume $m\{ r \geq i\} > 0$. Fix $\e > 0$ and let $\overline{D}_\e^W, \underline{D}_\e^W, W = B, V$ be as in \eqref{eq:overlineD}, \eqref{underlineD}, respectively (note that we have not assumed $(T, m)$ is ergodic, so $\lambda_i(x)$ can depend on $x$). 
Since $\| \cdot\|_B \leq \| \cdot \|_V$, it is of interest to bound $\overline{D}_\e^V$ from above by $\overline{D}_\e^B$ and $\underline{D}^B_\e$ from above by $\underline{D}^V_\e$. We prove such a comparison under the following additional assumption: 
\begin{itemize}
	\item[(5)] For all $x \in X$, the range of $A_x : B \to B$ is contained in $V$ and is bounded as a linear operator $(B, \| \cdot\|_B) \to (V, \| \cdot \|_V)$. Moreover, assume that for some $p > 3$ we have that
	\begin{equation}\label{eq:moment-regularizing}
\log^+ \| A_x\|_{B \to V} \in L^p(m) \, . 
\end{equation}
\end{itemize}

\begin{acorollary}\label{cor:reguarlityFunctions}
Assume the setting of Theorem \ref{thm:allEqualLE} and additionally that \eqref{eq:moment-regularizing} above holds.  For any $\delta > 0$, there exists a function $K_\delta : \{ r \geq i\} \to \R_{\geq 1}$ such that
	\[
	\overline{D}_\e^V(x) \leq K_\delta(x) \overline{D}_{\e + \delta}^B(x) \quad \text{ and } \quad 
	\underline{D}_\e^B(x) \leq K_\delta(x) \underline{D}_{\e + \delta}^V(x)
	\]
hold for any $\e > 0$ and $m$-a.e. $x \in \{ r \geq i\}$. The function $K_\delta$ satisfies the moment estimate
\begin{align}
\int_{\{ r \geq i \}} (\log^+ K_\delta)^{q} \, \dee m \lesssim_{p, q} \delta^{- (p-q)} (1 + \| \log^+ \varphi\|_{L^p(m)}^{p})  \quad \text{ for all } \quad q < \frac{p(p - 3)}{p -1} \, , 
\end{align}
where $\varphi (x) := \| A_x\|_{B \to V}$. 
\end{acorollary}
That is, while usually one has little control over $\overline{D}_\e^W, \underline{D}_\e^W$, these terms are comparable between $W = B$ and $V$, in a way that 
can be made explicit in terms of the $L^p$-norm of $\log^+ \| A_\cdot\|_{B \to V}$. 
Viewing $V$ as a ``higher regularity'' subspace of $B$ (c.f. the discussion in Section \ref{subsec:appsIntro}), 
condition (5) has the connotation that $A_x$ \emph{regularizes} initial data from $B$ into $V$. 
This condition is natural for linear cocycles derived from dissipative parabolic PDE, and holds for all the applications covered in Section \ref{sec:Applications} below. 

\begin{proof}[Proof of Corollary \ref{cor:reguarlityFunctions} assuming Theorem \ref{thm:allEqualLE}]
We restrict attention to the lower bound for $\underline{D}_\e^B$; the upper bound on $\overline{D}_\e^V$ is easier and omitted for brevity. Moreover, we assume below that $(T, m)$ is ergodic, so the value $r \geq i$ and the Lyapunov exponents $\lambda_i$ are  almost-surely constant (the non-ergodic case is treated similarly and is omitted for brevity).

To start, observe that for $v \in F_i^V(x)$, we have
\begin{align*}
\| A^n_x v\|_B & \geq \| A^{n + 1}_x v\|_V \| A_{T^n x}\|_{B \to V}^{-1} \\
& \geq \left(\underline{D}^V_\e(x) \| A_{T^n x} \|_{B \to V} \right)^{-1} e^{n (\lambda_i - \e)} \| v \|_V \sin \angle^V(v, F_{i + 1}^V(x)) \,. 
\end{align*}
Since $\| \cdot \|_B \leq \| \cdot\|_V$, it holds directly from  \eqref{defn:angleHilbSpace} that
\[
\| v \|_V \angle^V(v, F_{i + 1}^V(x)) \geq \| v \|_B \angle^B(v, F_{i + 1}^V(x)) \,. 
\]
Since $V$ is dense in $B$ and $F_{i + 1}^V = V \cap F_{i + 1}^B$, the closure of $F_{i + 1}^V(x)$ in $B$ coincides with $F_{i + 1}^B(x)$, hence $\angle^B(v, F_{i + 1}^V(x)) = \angle^B(v, F_{i + 1}^B(x))$. We conclude
\[
\| A^n_x v\|_B \geq \left( \underline{D}^V_\e(x)\| A_{T^n x} \|_{B \to V}  \right)^{-1} e^{n (\lambda_i - \e)} \cdot \| v \|_B \sin \angle^B(v, F_{i + 1}^B(x)) \, . 
\]
The above estimate holds uniformly over $v \in F_i^V(x)$, and so by density of $F_i^V(x)$ in $F_i^B(x)$ we conclude the same holds for $v \in F_i^B(x)$.

It remains to bound $\| A_{T^n x} \|_{B \to V}$ from above. Below, we write ``$\lesssim_p, \lesssim_{p,q}$'' for bounds up to a universal multiplicative constant depending only on $p, $ and/or $q$ and independent of all other parameters, e.g., $\delta$. 
With $G_n := \{ x \in X : \| A_{T^n x} \|_{B \to V} > e^{n \delta}\}$, we have 
\[
m(G_n) = m \{ \log^+ \| A_{T^n x} \|_{B \to V} > n \delta\} \leq \frac{\|\log^+ \varphi\|_{L^p(m)}^p}{(n \delta)^p}
\]
by Chebyshev's inequality, where $\varphi(x):= \| A_{T^n x} \|_{B \to V}$. Since $p > 3 > 2$, it holds that $\sum_n m(G_n) < \infty$, hence 
\[
N_\delta (x) := \max\{ n \geq 0 : \|A_{T^n x}\|_{B \to V} > e^{n \delta} \}\, 
\]
is almost-surely finite by the Borel-Cantelli Lemma, with the tail estimate
\[
m \{ N_\delta > n \} \leq \sum_{\ell = n+1}^\infty m(G_\ell) \lesssim_p \delta^{-p} n^{- p + 1} \| \log^+ \varphi \|_{L^p(m)}^p \,. 
\]
For $m$-a.e. $x$, we now have
\[
\| A_{T^n x} \|_{B \to V} \leq e^{n \delta} \cdot \left( 1 \vee \max_{0 \leq i \leq N_\delta}\| A_{T^i x} \|_{B \to V} \right)  =: e^{n \delta} K_\delta(x) \, .
\]
Plugging $K_\delta$ into our previous estimate, we conclude $\underline{D}^B_\e \leq K_\delta \underline{D}^V_{\e + \delta}$, as desired. 

It remains to estimate the $q$-th moment of $\log^+ K_\delta$, where from here on $q < p$ is fixed. 
We have
\begin{gather*}
(\log^+ K_\delta(x))^q \leq \sum_{i \leq N_\delta} (\log^+ \| A_{T^i x} \|_{B \to V})^q \, , \quad \text{ hence} \\
\int (\log^+ K_\delta)^q \,\dee m  \leq \sum_{n =0}^\infty \sum_{i = 0}^n \int_{\{ N_\delta = n \}} (\log^+ \| A_{T^i x} \|_{B \to V})^q \,\dee m (x)\, . 
\end{gather*}
Using H\"older's inequality on each summand and that $m \circ T^{-1} = m$, we obtain
\begin{align*}
\int (\log^+ K_\delta)^q \dee m & \leq 
 \| \log^+ \varphi \|_{L^p}^{q} \sum_{n =0}^\infty (n + 1) (m\{ N_\delta = n \})^{1-\frac{q}{p}} \\
& \lesssim_{p, q} \delta^{- (p - q)} (1 + \| \log^+ \varphi \|_{L^p}^{p}) \sum_{n = 1}^\infty n^{1 + (1 - p)(\frac{p-q}{p}) } \, . 
\end{align*}
The sum in the RHS is finite iff $q < \frac{p(p-3)}{p -1}$ (note the right-hand quantity is $> 0$ iff $p > 3$). 
\qedhere

\end{proof}

\subsection{Comments on existing results}\label{subsec:prevWork}

To the authors' knowledge, the first result on the dependence of Lyapunov exponents on the norm was given in \cite{froyland2013metastability}, which considered two potentially nonequivalent norms on the same Banach space. 

During the preparation of this manuscript, the authors discovered that Theorem 37 in Appendix A of \cite{gonzalez2021stability} is a version of the main result Theorem \ref{thm:allEqualLE} of this manuscript. On the other hand, the (short and elegant) proof given in \cite{gonzalez2021stability} relies on the invertibility of the base mpt $T : (X, \Bc, m) \circlearrowleft$, while the proof given here, although longer, is inherently ``one-sided'' and does not rely at all on invertibility of $T$. We also note that the setting of \cite{gonzalez2021stability} requires only a ``quasi-compactness'' assumption on the cocycle, not compactness as we assume here. However, in view of our intended applications to dissipative parabolic PDE, we have opted for the sake of simplicity to limit the proof of Theorem \ref{thm:allEqualLE} to the compact case. While not all details have been checked, the authors are confident an approach analogous to that 
given here will work in the quasi-compact setting. 
Corollary \ref{cor:reguarlityFunctions} appears to be new.

Lastly, we note that a version Theorem \ref{thm:allEqualLE} for a class of linear delay-differential equations appears in the paper \cite{mierczynski2020lyapunov}.
\section{Proof of Theorem \ref{thm:allEqualLE}} \label{sec:METsameExp}

In Section \ref{subsec:Grassmanian} 
we collect some preliminary regarding the Grassmanian of closed subspaces
and a notion of determinant on finite-dimensional subspaces of Banach spaces. In Section \ref{subse:cmainProp2} we prove an intermediate result. We complete the proof of Theorem \ref{thm:allEqualLE} in Section \ref{subsec:proofOfEqualLE}. 

\subsection{Preliminaries}\label{subsec:Grassmanian}

Let $(B, \| \cdot\|_B)$ be a Banach space. Let $\Gr(B)$ denote the Grassmanian of $B$, i.e., the set of closed subspaces of $B$. For $k \in \N$, write $\Gr_k(B)$ for the set of $k$-dimensional subspaces of $B$ and $\Gr^k(B)$ for the set of closed, $k$-codimensional subspaces. 

Throughout, $\Gr(B)$ is endowed with the metric topology coming from the Hausdorff distance
\begin{align}\label{eq:defnHausDist}
d_{Haus}^B(E, E') = \max \bigg\{ \sup_{e \in E, \|e\|_B = 1} \dist^B(e, S_{E'}) , \sup_{e' \in E', \|e'\|_B = 1} \dist^B(e', S_{E}) \bigg\}
\end{align}
Equation \eqref{eq:defnHausDist} is the usual Hausdorff distance between two closed subsets of a metric space. In this case, we are taking the usual Hausdorff distance of the unit spheres
$S_E := \{ e \in E : \|e\|_B = 1\}$. Note that 
for $v \in B$ and $S \subset B$ we write $\dist^B(v, S) = \inf_{s \in S} \|v - s\|_B$ for the minimal distance between $v$ and $S$ in the $\|\cdot\|_B$ norm. We note that in $d_{Haus}^B$, the sets $\Gr_k(B), \Gr^k(B)$ are clopen in $\Gr(B)$ for all $k \geq 1$. 
For additional background, see Section IV.2 of \cite{kato2013perturbation}. 

\subsubsection{Norm comparison}

Let $(V, \| \cdot\|_V)$ be another Banach space such that $V \subset B$ and $\|\cdot\|_B \leq \| \cdot\|_V$. 
\begin{definition}
For $E \in \Gr(V)$, define
\begin{align}\label{eq:defineAlpha}
\a(E) = \sup_{v \in E \setminus \{ 0 \}} \frac{\|v\|_V}{\|v\|_B} \, ,
\end{align}
noting $1 \leq \a(E) < \infty$ automatically by compactness of $S_E$ when $\dim E < \infty$ 
(in particular, the $\sup$ is a $\max$). 
\end{definition}
For our purposes, we will require 
some degree of control over $\a(E)$ as $E$ varies: 

\begin{lemma}\label{lem:equivNorm}
Let $k \geq 1$ and $E_0 \in \Gr_k(V)$. Then, $E \mapsto \alpha(E)$ is upper semi-continuous at $E = E_0$: for any $\epsilon > 0$ there exists $\delta = \delta(E_0, \epsilon) > 0$ such that if $d_{Haus}^V(E, E_0) < \delta$, then 
\[
\a(E) \leq \a(E_0) + \epsilon\,. 
\]
\end{lemma}
It is not hard to check that if $d_{Haus}^V(E, E_0) < 1$, then $\dim E = \dim E_0$ (Corollary 2.6 of Section IV.2 in \cite{kato2013perturbation}), 
hence $\a(E) < \infty$ automatically. Lemma \ref{lem:equivNorm} goes further, asserting 
that the $B$ and $V$ norms are \emph{uniformly equivalent} for all $E$ close enough to $E_0$ in $d_{Haus}$. 
\begin{proof}[Proof of Lemma \ref{lem:equivNorm}]
Let $E \in \Gr(V)$ and assume $d_{Haus}^V(E, E_0) < \delta$ for some $\delta > 0$ to be specified. Let $v \in E, \|v\|_V = 1$ be so that
\[
\a(E) = \frac{1}{\|v\|_B} \, .
\]
Let $v_0 \in E_0, \|v_0\|_V = 1$ be such that $\|v - v_0\|_V \leq \delta$. 
We see that 
\[
\a(E) =  \frac{\|v_0\|_B}{\|v\|_B} \frac{1}{\|v_0\|_B}\leq  \frac{\|v_0\|_B}{\|v\|_B} \a(E_0) \, 
\]
by definition of $\a(E_0)$. Now, $\|v - v_0\|_B \leq \|v - v_0\|_V < \delta$ and so 
 $\|v\|_B \geq \|v_0\|_B - \delta$. Using that $\|v_0\|_B \geq \alpha(E_0)^{-1}$, we see that 
\[
\frac{\|v_0\|_B}{\|v\|_B} \leq \frac{\|v_0\|_B}{\|v_0\|_B - \delta} 
= \frac{1}{1 - \delta \|v_0\|_B^{-1}} \leq \frac{1}{1 - \delta \alpha(E_0)} \,. 
\]
On taking $\delta > 0$ sufficiently small (depending only on $\alpha(E_0)$ and $\epsilon > 0$), we conclude $\a(E) \leq  \a(E_0) + \epsilon$ as desired. 
\end{proof}

It is also useful to compare the Hausdorff distances $d_{Haus}^B, d_{Haus}^V$: 

\begin{lemma}\label{lem:distComparison}
For any $E, E' \in \Gr_k(V), k < \infty$, we have that 
\[
d^B_{Haus}(E, E') \leq 2 \max\{ \alpha(E), \alpha(E')\} d^V_{Haus}(E, E') \,. 
\]
\end{lemma}
\begin{proof}
It is straightforward to check that $\dist^B(e', E) \leq \dist^V(e', E)$ for all $e' \in E'$. 
Since $\dist^B(a e', E) = |a| \dist^B(e', E)$ holds for all $a \in \R$, we see that
\begin{align*}
\max_{\substack{e' \in E', \| e'\|_B = 1}} \dist^B(e', E) 
& \leq \max_{\substack{e' \in E', \| e'\|_B = 1}} \| e'\|_V \dist^V(\| e'\|_V^{-1} e' , E) \\
& \leq \alpha(E') \max_{\substack{e' \in E', \| e'\|_V = 1}} \dist^V(e', E) \, .
\end{align*}
Reversing the roles of $E, E'$, we conclude $\delta^B(E, E') \leq \max\{ \alpha(E), \alpha(E')\} \delta^V(E, E')$, where 
\[
\delta^B(E, E') = \max \bigg\{ \sup_{e \in E, \|e\|_B = 1} \dist^B(e, E') , \sup_{e' \in E', \|e'\|_B = 1} \dist^B(e', E) \bigg\} \, .
\]
The desired conclusion now follows from the following standard inequality (c.f. Section IV.2 of \cite{kato2013perturbation}): 
\[
\delta^B(E, E') \leq d_{Haus}^B(E, E') \leq 2 \delta^B(E, E') \,.  \qedhere
\]
\end{proof}

\subsubsection{Determinants in Banach spaces}

The following is an assignment to each $E \in \Gr_k(B)$ a ``volume element'' along $E$. 

\begin{definition} \
\begin{itemize}
\item[(a)] For $E \in \Gr_k(B), k \geq 1$, we define $\Vol_E^B$ to be the Lebesgue measure on $E$ 
normalized so that 
\begin{align}\label{3normalize}
\Vol_E^B\{ v \in E : \| v\|_B \leq 1\} = 1 \,. 
\end{align}
\item[(b)] For a linear operator $A : B \to B$ and $E \in \Gr_k(B), k \geq 1$, we define the \emph{determinant} $\det_B(A|E)$ of $A|_E : E \to B$ by
\[
\det{}_B(A|E) = 
\begin{cases}
\frac{\Vol_{A(E)}^B(A (S))}{\Vol_E^B(S)} & A|_E \text{ injects} \, , \\
0 & \text{else,}
\end{cases}
\]
where $S \subset E$ is any Borel set of positive, finite $m_E^B$-measure.
\end{itemize}
\end{definition}

We note that the measure $\Vol_E^B$ on $E$ is characterized\footnote{The measure $\Vol_E^B$ is sometimes called the Busemann-Hausdorff measure \cite{busemann1947intrinsic} and appears naturally in Finsler geometry.} uniquely by (i) translation invariance and (ii) the normalization \eqref{3normalize}. It follows from this unique characterization that $\det_B(A | E)$ is well-defined irrespective of the choice $S \subset E$. Below we recall some basic properties of $\det_B(\cdot | \cdot)$; for proofs and additional background, see \cite{blumenthal2015volume}. 

In what follows, we define the \emph{minimum-norm} of $A|_E$ by
\begin{equation}\label{eq:min-norm}
\mathfrak m_B(A|_E) := \inf\{ \| A v \|_B : v \in E, \| v \|_B = 1\}. 
\end{equation}
Note that if $A|_E$ is injective, then finite-dimensionality of $E$ implies that $A|_E : E \to A(E)$ is invertible, and therefore we have $\mathfrak m_B(A|_E) = \| (A|_E)^{-1}\|^{-1}$.

\begin{lemma}\label{lem:detProps}
Let $E \in \Gr_k(B), k \geq 1$. 
\begin{itemize}
	\item[(a)] For bounded linear operators $A_1, A_2 : B \to B$, we have that
	\[
	\det{}_B(A_1 A_2 | E) = \det{}_B(A_1|A_2(E)) \det{}_B(A_2|E) \,. 
	\]
	\item[(b)] If $A : B \to B$ is a bounded linear operator and $A|_E$ injects, then
	\[
	\mathfrak m_B(A|_E)^k \leq \det{}_B(A|E) \leq \| A\|_B^k.
	\]
\end{itemize}
\end{lemma}

The following compares $\det_B$ and $\det_V$. 

\begin{proposition}\label{prop:detComparisonBV}
Let $V \subset B$ be an embedded Banach space satisfying $\| \cdot\|_B \leq \| \cdot\|_V$. Let $E \in \Gr_k(V), k \geq 1$. Let $A : B \to B$ be a bounded linear operator for which $A(E) \subset V$ and $A|_E$ injects.  Then, 
\begin{align*}
\det{}_B(A|E) & \leq \alpha(E)^k \det{}_V(A|E) \, ,  \quad \text{ and} \\
\det{}_V(A|E) & \leq \alpha(A(E))^k \det{}_B(A|E) \,. 
\end{align*}
\end{proposition}
\begin{proof}
Let $D^B_E$ denote the unit ball of $E$ in the $\| \cdot\|_B$ norm. 
First, we show that for all Borel $K \subset E$ we have
\[
	\Vol_E^B(K) \leq \Vol^V_E(K)\quad \text{and}\quad \Vol^V_E(K) \leq \alpha^{k}(E)\Vol^B_E(K).
\]
To see this, observe that by uniqueness of Haar measure, there exists $c > 0$ such that
that $\Vol_E^V = c \, \Vol_E^B$. Let $D_E^W$ denote the unit ball of $E$ in the $W$ norm for $W = V, B$. To estimate $c$, we have
\[
1 = \Vol_E^V(D^V_E) = c \, \Vol_E^B(D_E^V) \leq c \, \Vol_E^B(D_E^B) = c \, , 
\]
hence $c \geq 1$, while
\[
1 = \Vol_E^B(D^B_E) = c^{-1} \Vol_E^V(D^B_E) \leq c^{-1} \alpha(E)^k \Vol_E^V(D^V_E) 
= c^{-1} \alpha(E)^k \, , 
\]
hence $c \leq \alpha(E)^k$.

To obtain the first inequality, observe that by definition of $\alpha(E)$, $\Vol^B_E(D^E_V) \geq \alpha(E)^{-k}$ and therefore
\begin{align*}
\det{}_B(A| E) &= \frac{\Vol_{A(E)}^B(A (D^V_E))}{\Vol_E^B(D^V_E)} \leq \Vol_{A(E)}^B(A (D_E^V)) \alpha(E)^k \\
& \leq \alpha(E)^k \Vol_{A(E)}^V(A (D_E^V)) = \alpha(E)^k \det{}_V(A | E) \, .
\end{align*}
For the second inequality, since $\Vol^V(D^B_E) \geq \Vol^V(D^V_E) = 1$, we find
\begin{align*}
\det{}_V(A| E) &= \frac{\Vol_{A(E)}^V(A (D^B_E))}{\Vol_E^V(D^B_E)} \leq \Vol_{A(E)}^V(A (D_E^B)) \\
&  \leq \alpha(A(E))^k \Vol_{A(E)}^B(A (D_E^B)) = \alpha(A(E))^k \det{}_B(A | E) \, .
\end{align*}
\end{proof}

\subsubsection{Determinants and Lyapunov exponents}

Our proof below uses the following characterization of Lyapunov exponents
in terms of asymptotic growth rates of determinants. 
Assume $(B, \| \cdot\|_B)$ is a separable Banach space, $T:(X, \mathscr F, m) \circlearrowleft$ is an mpt, and $A : X \to L(B)$ is a strongly measurable such that $A(x)$ is compact for all $x \in X$ (the setting of Theorem \ref{thm:nonInvMET}). 
Let $x$ be an $m$-generic point and let $r(x)$ be as in \eqref{eq:defineR2}. 
Let $\lambda_i(x), d_i(x)$ denote the Lyapunov exponents and corresponding multiplicities at $x$, and $B =: F_1 \supset F_2(x) \supset \cdots$ denote the corresponding filtration. Recall (see \eqref{eq:multDefn}) that 
$d_i(x)$ is the codimension of $F_{i + 1}(x)$ in $F_i(x)$, and $M_i(x) = d_1(x) + \cdots + d_i(x)$ is the codimension of $F_{i+1}(x)$ in $B$. 

Let us define $\chi_1(x) \geq \chi_2(x) \geq \cdots$ to be the
Lyapunov exponents counted \emph{with multiplicity}, i.e., 
\begin{gather}\label{eq:charExpWMult}
\chi_{M_{i-1}(x) + 1} = \cdots = \chi_{M_{i}(x)} = \lambda_i(x) 
\end{gather}
for all $i \leq r(x)$ (here $M_0(x) := \codim F_1 = 0$ by convention). If $r(x) = 0$, we adopt the convention that $\chi_j (x) = -\infty$
for all $j \geq 1$, while if $0 < r(x) < \infty$, we define 
$\chi_j(x) = -\infty$ for all $j > M_{r(x)}(x)$. 
For $k \geq 1$ we define
\[
\Sigma_k(x) := \chi_1(x) + \cdots + \chi_k(x) \, , 
\]
with the convention that $\Sigma_k(x) = -\infty$ if $r(x) < \infty$ and $k > M_{r(x)}(x)$. 

\newcommand{\Vc}{\mathcal V}

Below, for a linear operator $A : B \to B$ and $k \geq 1$ we define 
\[
\Vc_k^B(A) = \sup\{ \det{}_B(A |E) : \dim E = k\} \,. 
\]
The following characterizes the sums $\Sigma_k(x)$ in terms of the maximal $k$-dimensional volume growth $\Vc^B_k(x)$. 
\begin{proposition}[\cite{blumenthal2015volume}]\label{prop:volGrowAndLE}
Assume the setting of Theorem \ref{thm:nonInvMET}. 
\begin{itemize}
\item[(a)]
For $m$-a.e. $x \in X$ and for all $k \geq 1$
\[
\Sigma_k(x) = \lim_{n \to \infty} \frac1n\log \Vc_k^B(A_x^n) \, . 
\]
\item[(b)] For all $i, k \geq 1$ and for $m$-a.e. $x \in \{ r \geq i\} \cap \{M_{i-1} < k \leq M_{i}\}$, it holds that if $E \in \Gr_k(B)$ and $E \cap F_{i + 1}(x) = \{ 0\}$, then 
\begin{align}\label{eq:topGrowRate3}
\Sigma_{k}(x) = \lim_{n \to \infty} \frac1n \log \det{}_B (A^n_x| E) \,. 
\end{align}
Equation \eqref{eq:topGrowRate3} also holds for $m$-a.e. $x \in \{ r = i\} \cap \{ M_r < k\}$ and all $E \in \Gr_k(B)$. 
\end{itemize}
\end{proposition}

\subsection{Main Proposition}\label{subse:cmainProp2}

Let us assume the setting of Theorem \ref{thm:allEqualLE}: 
$(B, \| \cdot\|_B)$ and $(V, \| \cdot\|_V)$ are Banach spaces with separable
duals, where $V \subset B$ is dense and $\| \cdot\|_B \leq \| \cdot \|_V$. 
We are given an mpt $T : (X, \mathscr F, m) \circlearrowleft$ and 
a strongly measurable mapping $A : X \to L(B)$ satisfying (1) -- (4) in 
Section \ref{subsec:statementOfResults}. For simplicity and to spare heavy notation, we will assume below that $(T, m)$ is ergodic, so that Lyapunov exponents are constant in $x$. The nonergodic case requires superficial changes, which we comment on in Remark \ref{rmk:nonergCase3} below.

For $W = B$ or $V$, let $r^W$ denote the number of distinct Lyapunov exponents $\lambda_i^W$ with multiplicities $d_i^W, M_i^W := d_1^W + \cdots + d_i^W$. Let 
$\chi_j^W$ denote
the Lyapunov exponents counted with multiplicity as in \eqref{eq:charExpWMult}, and let $\Sigma_k^W = \sum_{j = 1}^k \chi_j^W$. 

The following is the main step in the proof of Theorem \ref{thm:allEqualLE}. 
\begin{proposition}[Main Proposition]\label{prop:mainProof} Let $i \geq 1$. 
\begin{itemize}
\item[(a)] If $i\leq r^V $, then $\Sigma_{M_i^V}^V = \Sigma_{M_i^V}^B$
\item[(b)] If $i\leq r^B$, then $\Sigma_{M_i^B}^V = \Sigma_{M_i^B}^B$
\end{itemize}
\end{proposition}

The proof of Proposition \ref{prop:mainProof} occupies the remainder of Section \ref{subse:cmainProp2}; we complete the proof of Theorem \ref{thm:allEqualLE} in 
Section \ref{subsec:proofOfEqualLE}. 

\subsubsection{Measurable selection of projectors}

Next, we will need the following Lemmas on the geometry of Banach spaces. 
Below, given a splitting $B = E \oplus F$ into closed subspaces $E, F \subset V$, 
we define $\pi_{E // F}$ to be the unique oblique projection onto $E$ with kernel $F$. 

\begin{lemma}\label{lem:measSelectCompl} Let $(X, \mathscr F)$ be a measurable space and let $F : X \to \Gr^k(B)$ be a measurable family. 

Then, there exists a measurable family $E : X \to \Gr_k(B)$ such that for all $x \in X$, (i) $V = E(x) \oplus F(x)$, (ii) the mapping $x \mapsto \pi_{E(x) // F(x)}$ is strongly measurable; and (iii) we have that $\| \pi_{E(x) // F(x)}\|_B \leq C_k := 3\sqrt{k} + 2$ for all $x \in X$.
\end{lemma}

\begin{lemma}[Corollary III.B.11 in \cite{wojtaszczyk1996banach}]\label{lem:boundedComplement}
Let $F \in \Gr^k(B)$ for $k \geq 1$, then there exists $E\in \Gr_k(B)$with $B= E\oplus F$ such that 
\[
\| \pi_{E // F}\|_B \leq \sqrt{k}\, , \quad \text{ and } \quad 
\| \pi_{F // E} \|_B \leq \sqrt{k} + 1 \, .
\]
\end{lemma}

The following is a version of Lemma \ref{lem:boundedComplement} for a measurable family $F(x), x \in X,$ of finite-codimensional spaces. 

\begin{proof}

Separability of $B^*$ implies separability of $\Gr^k(B)$ (Lemma B.12 in \cite{gonzalez2014semi}; see also 
Chapter IV, \S2.3 of \cite{kato2013perturbation}). 
Let $\{ F_n\} \subset \Gr^k(B)$ be a countable dense sequence and for each $n$ 
let $E_n$ be a $k$-dimensional subspace of $B$ such that $\| \pi_{E_n // F_n}\|_B \leq \sqrt{k}$ as in Lemma \ref{lem:boundedComplement}. Set $\epsilon = \frac{1}{2 (\sqrt{k }+1)}$ and recursively define
\[
S_1 = \{ x \in X : d_{Haus}^B(F(x), F_1) < \epsilon\} \, , \quad 
S_n = \{ x \in X : d_{Haus}^B(F(x), F_n) < \epsilon \} \setminus \bigcup_{m=1}^{n-1} S_m \,. 
\]
By the density of $\{ F_n\}$ and measurability of $x \mapsto F(x)$, it holds that $\{ S_n\}$ is a countable partition of $X$ by $\mathscr F$-measurable sets. Now, define
\[
E(x) := E_n \quad \text{ for } \quad  x \in S_n \,. 
\]
It is immediate that $x \mapsto E(x)$ is measurable. Measurability of $x \mapsto \pi_{E(x) // F(x)}$ follows from Lemma B.18 of \cite{gonzalez2014semi}. To estimate $\| \pi_{E(x) // F(x)}\|_B$, Proposition 2.7 of \cite{blumenthal2015volume} implies
that for $x \in S_n$, 
\[
\| \pi_{E_n // F(x)} |_{F_n}\|_B \leq \frac{2 d^B_{Haus}(F_n, F(x))}{\| \pi_{F_n // E_n}\|_B^{-1} - d^B_{Haus}(F_n, F(x))} \leq 2
\]
for our choice of $\epsilon$. Therefore, for $x \in S_n$, $e_n \in E_n, f_n \in F_n$, 
\begin{align*}
\| \pi_{E(x) // F(x)} (e_n + f_n)\|_B 
& = \| e_n\|_B + \| \pi_{E_n // F(x)} |_{F_n}\|_B \| f_n\|_B \\
& \leq (3 \sqrt{k} + 2) \| e_n + f_n\|_B \,. 
\end{align*}
\end{proof}

Below we record various additional measurability properties, used freely and without further mention in Section \ref{subsec:proofOfEqualLE} below. 
\begin{lemma}
Let $(B, \| \cdot\|_B)$ be a separable Banach space and let 
$M : X \to L(B)$ be a strongly measurable mapping. 
\begin{enumerate}
\item (Lemma B.16 of \cite{gonzalez2014semi}) If $F : X \to \Gr(B)$ is measurable, then $x \mapsto \| M(x)|_{F(x)}\|_B$ is measurable. Consequently, the function $x \mapsto \| M(x)\|_B$ is measurable.
\item (Lemma A.5 of \cite{gonzalez2014semi}) If $M' : (X, \mathscr F) \to L(B)$ is another strongly measurable mapping, then 
$M' \circ M : (X, \mathscr F) \to L(B), x \mapsto M'(x) \circ M(x)$ is also strongly measurable. 
\end{enumerate}
\end{lemma}
For additional discussion of strong measurability, see, e.g.,  Appendices A, B of \cite{gonzalez2014semi}. 

\subsubsection{Quotient cocycle construction}

Assume $r^V \geq i$ for some fixed $i \geq 1$, and define $F(x) := F_{i+1}^V(x)$. By Lemma \ref{lem:measSelectCompl}, there exists a measurably-varying family $x \mapsto E(x)$ of closed, finite-dimensional complements in $V$ to $F(x)$, equipped with a measurably-varying family of projectors $x \mapsto \pi^\perp_x := \pi_{E(x) // F(x)}$ with $\| \pi_x^\perp\|_V \leq C_k := \sqrt{k} + 1$ for all $x \in X$, where $k = M_i^V = \codim F(x)$ is constant in $x$. 
 Note that $\dim E(x) = \codim F(x)  = k$. 

We define $\hat A_x : E(x) \to E(T x), \hat A^n_x : E(x) \to E(T^n x)$ as follows: 
\[
\hat A_x := \pi_{T x}^\perp A_x|_{E(x)} \, , \quad \hat A_x^n = \hat A_{T^{n-1} x} \circ \cdots \circ \hat A_x \,. 
\]
We note that invariance of $F(x)$ under $A_x$ as in \eqref{eq:equivariant} implies the identities
\begin{gather}\label{eq:formQuotCocycle}
\begin{gathered}
\hat A^n_x = \pi^\perp_{T^n x} A^n_x|_{E(x)} \, , \\
\hat A^{n+1}_x = \hat A_{T^n x} \circ A^n_x |_{E(x)} \,. 
\end{gathered}
\end{gather}

Below we write $\det{}_V(\hat A^n_x)$ below for the determinant of $\hat A^n_x : E(x) \to E(T^n x)$. In view of the expression \eqref{eq:formQuotCocycle}, we have that 
\begin{align}\label{detFormHatA}
\det{}_V(\hat A^n_x) = \det{}_V(\pi^\perp_{T^n x} \circ A^n_x | E(x)) \,. 
\end{align}

\begin{lemma}\label{lem:quotientProjVoLGrowh}
For $m$-a.e. $x \in X$ we have that
\[
\Sigma_{M_i^V}^V = \lim_{n \to \infty} \frac1n \log \det{}_V(\hat A^n_x) \,.
\]
\end{lemma}
The proof uses the following corollary to the MET, which we recall here. Recall that $\angle^V(v, F) \in [0,\pi/2]$ is the \emph{angle} between $v \in V \setminus \{ 0 \}$ and a subspace $F \subset V$. For a nontrivial subspace $E \subset V$, we write
\[
\angle^V_{\rm min}(E, F) := \min\left\{ \angle^V(v, F) : v \in E \setminus \{ 0 \} \right\} \,. 
\]
Note that if $E$ and $F$ share a nontrivial subspace $E\cap F$, then $\angle^V_{\rm min}(E,F) = 0$.

\begin{corollary}\label{cor:anglesDegenerateSlowly}
	For $m$-a.e. $x \in X$ and for any complement $E$ to $F(x)$ in $V$, 
	we have that
	\[
	\lim_{n \to \infty} \frac1n \log \sin \angle_{\rm min}^V(A^n_x(E), F(T^n x)) = 0 \,. 
	\]
\end{corollary}
The proof of Corollary \ref{cor:anglesDegenerateSlowly} is contained in, e.g., paragraph ``Proof of (h)'' in the proof of Theorem 16 in \cite{gonzalez2015concise}. 

\begin{proof}[Proof of Lemma \ref{lem:quotientProjVoLGrowh}]
To start, by \eqref{detFormHatA} and multiplicativity of the determinant (Lemma \ref{lem:detProps}(a)), we have that 
\[
\det{}_V(\hat A^n_x) = \det{}_V (\pi^\perp_{T^n x} | A^n_x(E(x)) ) \det{}_V (A^n_x| E(x)) \,. 
\]
It suffices to check that 
\[
\lim_{n \to \infty} \frac1n \log \det{}_V ( \pi^\perp_{T^n x} | A^n_x(E(x)) = 0 \, .
\]
By Lemma \ref{lem:detProps}(b), we have
\[
\mathfrak m_V(\pi^\perp_{T^n x}|_{A^n_x(E(x))})^k \leq \det{}_V ( \pi^\perp_{T^n x} | A^n_x(E(x)) \leq \| \pi^\perp_{T^n x}\|_V^k \, , 
\]
where $k := M_i^V$ and $\mathfrak m_V(\pi^\perp_{T^n x}|_{A^n_x(E(x))})$ is the minimum norm defined by \eqref{eq:min-norm}. The RHS is uniformly bounded in $x, n$ from above by a constant in $k$, so it remains to bound the LHS from below. 
For this, we use the following estimate which related the norm of $\pi_{E//F}v$ to the distance between $v$ and $F$.
\begin{claim}
Let $E, F$ be complementary closed subspaces of the Banach space $V$. Then 
\[
\| \pi_{E // F}v\|_V \geq \sin \angle^V(v, F)\,\| v \|_V. 
\]
\end{claim}
\begin{proof}[Proof of Claim]
Let $v \in E', \| v \|_V = 1$. Write $v = u + w$ where $u \in E, w \in F$. Then, 
\[
\sin \angle^V(v, F) = \inf_{\hat w \in F}\frac{\|v - \hat w\|_V}{\|v\|_V} \leq \frac{\| v - w\|_V}{\| v \|_V} =\frac{\| u \|_V}{\|v\|_V} \,. \qedhere
\]
\end{proof}
Consequently, the above claim implies that
\[
\mathfrak m_V(\pi^\perp_{T^n x}|_{A^n_x(E(x))} ) \geq \sin \angle_{\rm min}^V(A^n_x(E(x)), F(T^n x)) \,. 
\]
Applying Corollary \ref{cor:anglesDegenerateSlowly} completes the proof. 
\end{proof} 

\subsubsection*{Proof of Proposition \ref{prop:mainProof}(a)}

To start, we check unconditionally that 
\begin{align}\label{eq:uncondVolBd}
\Sigma_k^B \leq \Sigma_k^V \, \quad \text{for all} \quad k \geq 1 \, . 
\end{align}
Using Proposition \ref{prop:volGrowAndLE}(b) and the density of $V \subset B$, 
we can choose a $k$-dimensional $E \subset V$ such that $\Sigma_k^B = \lim_{n \to \infty} \frac1n \log \det{}_B(A^n_x|E)$. Applying now \eqref{prop:detComparisonBV}, we have
\begin{align*}
\Sigma_k^B & = \lim_{n \to \infty} \frac1n \log \det{}_B(A^n_x|E) \\
& \leq \liminf_{n \to \infty} \left( \frac{k}{n} \log \alpha(E) + \frac1n \log \det{}_V(A^n_x|E) \right) \\
& \leq \lim_{n \to \infty} \frac1n \log \Vc_k(A^n_x) = \Sigma_k^V \,. 
\end{align*}

It remains to check that $\Sigma_{M_i^V}^V \leq \Sigma_{M_i^V}^B$. For this, 
observe that with the measurable selection $x \mapsto E(x)$ as above, we have that $\alpha(E(x)) < \infty$ $m$-almost everywhere. Choose $C_0 > 0$ large enough so that 
\[
m (U) \geq \frac{99}{100} \, , \quad \text{ where}  \quad U := \{ \alpha(E(x)) \leq C_0\} \subset X \,. 
\]
Observe that $m(T^{-1} U \cap U) \geq \frac{98}{100} > 0$ by $T$-invariance of $m$. 
By the Poincar\'e Recurrence Theorem, for $m$-a.e. $x \in U \cap T^{-1} U$, there is a sequence of times $n_k \to \infty$ such that $T^{n_k} x \in U \cap T^{-1} U$ for all $k$. Fixing such an $x$ and sequence $(n_k)$, we estimate
\begin{equation}
\begin{aligned}
\Sigma^V_{M_i^V} & = \lim_{n \to \infty} \frac1n \log \det{}_V(\hat A^n_x)  \\
& \leq \liminf_{n \to \infty} \frac{M_i^V}{n+1} \log \alpha(E(T^{n+1}x)) + \liminf_{n \to \infty} \frac{1}{n+1} \log \det{}_B(\hat A^{n+1}_x) 
\end{aligned}
\end{equation}
where in the first line we used Lemma \ref{lem:quotientProjVoLGrowh} and in the second we used Proposition \ref{prop:detComparisonBV}. The first $\liminf$ goes to $0$, as $\alpha(E(T^{n+1} x)) \leq C_0$ along the sequence of times $n = n_k$. For the second term, we estimate
\begin{align*}
\det{}_B(\hat A^{n+1}_x) &= \det{}_B(\hat A_{T^n x}) \det{}_B(A^n_x|_{E(x)}) \\
& \leq \alpha(E(T^n x))^{M_i^V} \det{}_V(\hat A_{T^n x}) \Vc^B_{M_i^V}(A^n_x) \\
& \leq \left( \alpha(E(T^n x)) \| A_{T^n x}\|_V\right)^{M_i^V} \Vc^B_{M_i^V}(A^n_x)
\end{align*}
using Lemma \ref{lem:detProps}(a) and \eqref{eq:formQuotCocycle} in the first line, Proposition \ref{prop:detComparisonBV} in the second line, and 
Lemma \ref{lem:detProps}(b) in the third line. By construction, along $n = n_k$
we have $\alpha(E(T^nx)) \leq C_0$. To control the $\| A_{T^{n} x} \|_V$ term, define
$g(x) := \log^+ \|A_x\|_{V}$. By our assumptions, we have $\log^+ g \in L^1(m)$. By a standard corollary of the Birkhoff Ergodic Theorem (see, e.g., Theorem 1.14 in \cite{walters2000introduction}), it holds that
\begin{align}\label{eq:integrability3}
\lim_n \frac1n g(T^n x) = 0  \quad m\text{-a.e.} \, .
\end{align}
In all, we conclude
$\Sigma^V_{M_i^V} \leq \Sigma^B_{M_i^V}  = \Sigma^B_{M_i^V}$, which in conjunction with \eqref{eq:uncondVolBd} implies $\Sigma^V_{M_i^V} = \Sigma^B_{M_i^V}$. 

\subsubsection*{Proof of Proposition \ref{prop:mainProof}(b)}

Since \eqref{eq:uncondVolBd} holds for all $k$, it suffices to check that 
\[
\Sigma^V_{M_i^B} \leq \Sigma^B_{M_i^B} \,. 
\]
The proof below is parallel to that of Proposition \ref{prop:mainProof}(a), the most notable change being that we use a quotient cocycle parallel to the space $\tilde F(x) := F_{i + 1}^B(x)$. The following is an analogue of Lemma \ref{lem:measSelectCompl} above.  
\begin{claim}
There exists a measurable selection $x \mapsto \tilde E(x)$ of complement to $\tilde F(x)$ with the properties that for a.e. $x$, (a) $\tilde E(x) \subset V$; 
(b) we have that
\begin{align}\label{eq:correctGrowth12}
\Sigma_{M_i^B}^V = \lim_{n \to \infty} \frac1n \log \det{}_V(A^n_x | \tilde E(x)) \, \quad \text{ and } \quad\Sigma_{M_i^B}^B = \lim_{n \to \infty} \frac1n \log \det{}_B(A^n_x | \tilde E(x)) \, ; 
\end{align}
and (c) the projector $\tilde \pi^\perp_x := \pi_{\tilde E(x) // \tilde F(x)}$ 
satisfies $\| \tilde \pi^\perp_x\|_B \leq C_{M_i^B}$, where $C_k = \sqrt{k} + 1$ for $k \geq 1$. 
\end{claim}

The proof requires the following. 
\begin{lemma}\label{lem:simultComplement}
For all $m, k \geq 1$ there exists $C_{m, k} > 0$ such that the following holds. Let $\{ V_i\}_{i = 1}^m \subset \Gr^k(B)$. Then, there exists a common complement $E \in \Gr_k(B)$ to each of the $V_i, 1 \leq i \leq m$, such that $\| \pi_{E // V_i} \|_B = \sin \angle_{\rm min}^B(E, V_i)^{-1} \leq C_{m, k}$. 
\end{lemma}
\begin{proof}[Proof of Lemma \ref{lem:simultComplement}]
We induct on the codimension $k$. The base case $k = 1$ is Corollary 2.5 in \cite{noethen2022well}. Assume the induction hypothesis for $k$ and fix $\{ V_1, \cdots, V_m\}$ of codimension $k + 1$. For each $i \leq m$, let $\hat V_i$ be an arbitrary extension of $V_i$ to a $k$-codimensional space. Using the induction hypothesis, fix $E \in \Gr_k(B)$ such that $\sin \angle_{\rm min}^B(E, V_i) \geq \sin \angle_{\rm min}^B(E, \hat V_i) \geq C_{m, k}^{-1}$. 
Define the hyperplanes $V_i' = E + V_i$ and, using the base case $k = 1$ let $v \in B$ be a unit vector with $\sin \angle^B(v, V_i') \geq C_{m, 1}^{-1}$ for all $i \leq m$. 
Let us now bound $\sin \angle_{\rm min}^B(E', V_i)$ from below, where $E' := E + \langle v \rangle$ and $\langle v \rangle$ is the line spanned by $v$. 
We compute:
\begin{align*}
\pi_{E' // V_i} & = \pi_{E // V_i \oplus \langle v \rangle} + \pi_{\langle v \rangle // E \oplus V_i}\\
& = \pi_{E // V_i}|_{V_i'} \circ \pi_{V_i' // \langle v \rangle} + 
\pi_{\langle v \rangle // V_i'} \\
\implies \| \pi_{E' // V_i} \|_B & \leq C_{m, k} C_{m, 1} + C_{m, 1} =: C_{m, k + 1} \,. 
\end{align*}
\end{proof}
\begin{proof}[Proof of Claim]
The proof is parallel to that of Lemma \ref{lem:measSelectCompl}.  
Assume for now that either $r^V = \infty$ or $r^V < \infty$ and $M_{r^V}^V \geq M_i^B$; we address the alternative case at the end. 
Fix $j \leq r^V$ so that $M_{j-1}^V < M_i^B \leq M_j^V$ (following the convention $M_0^V = 0$). 
For short, set $k := M_i^B, \ell := M_j^V$ and $F(x) := F_{j + 1}^V(x)$. 
 
 Set $F(x) := F_{j+1}^V(x) \subset V$. Let $(F_m), (\tilde F_n)$ denote countable dense sequences in $\Gr^{M_j^V}(V), \Gr^{M_i^B}(B)$, respectively. Set $\epsilon = \frac{1}{2 (C_{2, k} + 1)}$, with $C_{2, k}$ as in Lemma \ref{lem:simultComplement}, and 
\[
\tilde {S}_{m, n} = \{ x \in X : d_{Haus}^V (F(x) , F_m) < \epsilon \, \text{ and } \, 
d_{Haus}^B(\tilde F(x), \tilde F_n) < \epsilon \} \,. 
\]
Refine to a partition $S_{m, n}, m, n \geq 1$ of $X$ such that 
$S_{m, n} \subset \tilde S_{m, n}$ for all $m, n$. 

Form a $B$-closed $k$-codimensional extension $F_m'$ of $F_m$. 
Apply Lemma \ref{lem:simultComplement} to obtain a complement $\tilde E_{m, n}' \in \Gr_k(B)$ to both $\tilde F_n, F_m'$ with 
\[
\sin \angle_{\rm min}^B(\tilde E_{m, n}' , F_m) \geq \sin\angle_{\rm min}^B(\tilde E_{m, n}' , F_m')  \geq C_{2, k}^{-1} \, ,  \quad \sin\angle_{\rm min}^B(\tilde E_{m, n}' , \tilde F_n) \geq C_{2, k}^{-1} \,. 
\]
Finally, using density of $V \subset B$, fix $\tilde E_{m, n} \in \Gr_k(V)$ so that 
$\sin \angle_{\rm min}^B(\tilde E_{m, n}, F_m) \geq (C_{2, k} + 1)^{-1}$ and 
$\sin \angle_{\rm min}^B(\tilde E_{m, n}' , \tilde F_n) \geq (C_{2, k} + 1)^{-1}$. This step can be justified using, e.g., continuity of $E \mapsto \pi_{E // F}$ as $E$ ranges over the set of complements to $F\in \Gr^k(B)$ (Lemma B.18 in \cite{gonzalez2014semi}). 

One now sets
\[
\tilde E(x) = E_{m, n} \quad \text{ for } x \in S_{m, n} \,. 
\]
the estimate on the $B$-norm of $\pi^\perp_x = \pi_{\tilde E(x) //\tilde F(x)}$ 
is completely parallel to that in Lemma \ref{lem:measSelectCompl} and is omitted. 
That $\tilde E(x) \cap F(x) = \{ 0 \}$ follows from a similar argument. 

The proof is now complete when either $r^V = \infty$ or $r^V < \infty$ and $M_{r^V}^V \geq M_i^B$. When $r^V < \infty$ and $M_{r^V} < M_i^B$ the proof is simpler: every $E_0 \in \Gr_k(V)$ satisfies the right-hand equation in \eqref{eq:correctGrowth12} by Proposition \ref{prop:volGrowAndLE}(b). So, in this case it suffices to apply Lemma \ref{lem:measSelectCompl} directly to obtain a complement $\tilde E(x)$ to $\tilde F(x)$.
\end{proof}

Form the quotient cocycle
\[
\tilde A_x := \tilde \pi^\perp_{T x} A_x |_{\tilde E(x)} \, , \quad \tilde A^n_x = \tilde A_{T^{n-1} x} \circ \cdots \circ \tilde A_x \, , 
\]
noting as before that $\tilde A^n_x = \tilde \pi^\perp_{T^n x} A^n_x|_{\tilde E(x)}$ by invariance of $\tilde F(x)$ (equation \eqref{eq:equivariant}). The proof is now largely the same as before with opposite signs: one checks that 
\[
\Sigma^B_{M_i^B} = \lim_{n \to \infty} \frac1n \log \det{}_B(\tilde A^n_x) 
\]
as in Lemma \ref{lem:quotientProjVoLGrowh} (no changes needed), and estimates
\begin{align}
\Sigma^B_{M_i^B} & = \lim_{n \to \infty} \frac1n \log \det{}_B(\tilde A^n_x) \\
& \geq \limsup_{n \to \infty} \frac{1}{n+1} \log \det{}_V(\tilde A^{n+1}_x) -  \liminf_{n \to \infty} \frac{M_i^B}{n+1} \log \alpha(E(T^{n+1}))  
\end{align}
By a recurrence argument parallel to before, the $\liminf$ term is 0 for a positive $m$-measure set of $x \in X$, while the $\limsup$ term is bounded by
\begin{align}
\det{}_V(\tilde A^{n+1}_x) &= \det{}_V(\tilde A_{T^n x}) \det{}_V(A^n_x|E(x)) \\
& \geq \left( \alpha(\tilde E(T^{n } x))^{-1}  \mathfrak m_B(A_{T^n x} |_{\tilde E(T^n x)} ) \right)^{M_i^B} \det{}_V(A^n_x|E(x)) \, , 
\end{align}
using the analogue of \eqref{eq:formQuotCocycle} for $\tilde A^{n+1}_x$ in the first line, and Lemma \ref{lem:detProps}(b) and Proposition \ref{prop:detComparisonBV} in the second line. 
By another recurrence argument, we can bound the above parenthetical 
term from below by a fixed positive constant along an infinite sequence of times $n_k \to \infty$ for an $m$-positive measure set of $x$. Overall, we conclude
$\Sigma^B_{M_i^B} \geq \lim_{n \to \infty} \frac1n \log \det{}_V(A^n_x | \tilde E(x)) = \Sigma^V_{M_i^V}$.

\subsection{Completing the proof of Theorem \ref{thm:allEqualLE}}\label{subsec:proofOfEqualLE}

To prove Theorem \ref{thm:allEqualLE} we first establish the following claims. 

\begin{claim}
If $r^V = 0$, then $r^B = 0$. 
\end{claim}
\begin{proof}[Proof of Claim]
Pursuing a contradiction, observe that if
$\lambda_1^B > -\infty$, then by the density of $V \subset B$, for $m$-a.e. $x$ there exists $v \in V \setminus \{ 0\}, \| v \|_V = 1$ so that $\frac1n \log\|A^n_x v\|_B \to \lambda_1^B$ as $n \to \infty$. Therefore,
\[
\liminf_n \frac1n \log\|A^n_x\|_V \geq \liminf_n \frac1n \log\|A^n_x v\|_V \geq
\lim_n \frac1n \log\|A^n_x v\|_B = \lambda_1^B > -\infty \,. 
\]
That $r^V = 0$ implies the LHS limit exists for $m$-a.e. $x$ and equals $-\infty$, a contradiction. 
\end{proof}

\begin{claim}\label{cla:LEequal} \
\begin{itemize}
\item[(a)] For all $i \geq 1$, we have that $r^V \geq i$ if and only if $r^B \geq i$. 
\item[(b)] If $r^V \geq i$, then $\lambda_j^V = \lambda_j^B$ and $d_j^V = d_j^B$ for all $1 \leq j \leq i$. 
\end{itemize}
\end{claim}

\begin{proof}[Proof of Claim \ref{cla:LEequal}]
We will prove below, by induction on $i$, that $r^V \geq i$ implies $r^B \geq i$ and 
$\lambda_j^V = \lambda_j^B, d_j^V = d_j^B$ for all $1 \leq j \leq i$. 
The proof that $r^B \geq i$ implies $r^V \geq i$ is identical on exchanging the roles
of $V$ and $B$ below; further details are omitted. 

 Assume first that $r^V \geq 1$: we will show that 
$r^B \geq 1, \lambda_1^B = \lambda_1^V$ and $d_1^B = d_1^V$. 
To start, by Proposition \ref{prop:mainProof}(a) we have
\begin{align}\label{eq:m1V}
\Sigma_{d_1^V}^V = \Sigma_{d_1^V}^B \, , 
\end{align}
 hence $\Sigma_{d_1^V}^B > -\infty$ and $r^B \geq 1$. Applying Proposition \ref{prop:mainProof}(b), we see that
 \begin{align} \label{eq:m1B}
 \Sigma_{d_1^B}^V = \Sigma_{d_1^B}^B \,. 
 \end{align}
 To proceed, assume $d_1^B \geq d_1^V$. 
Then, \eqref{eq:m1V} implies $d_1^V \lambda_1^V = d_1^V \lambda_1^B$, hence $\lambda_1^V = \lambda_1^B = \lambda_1$, while combining with \eqref{eq:m1B} gives
\[
(d_1^B - d_1^V) \lambda_1 = \chi_{d_1^V + 1}^V + \cdots + \chi_{d_1^B}^V \, . 
\]
However, if $d_1^B > d_1^V$ this gives a contradiction since $\chi_j < \lambda_1 = \lambda_1^V$ for all $j > d_1^V$. One can similarly rule out the case $d_1^B < d_1^V$; we conclude $\lambda_1^B = \lambda_1^V$ and $d_1^B = d_1^V$.

Assume now the induction hypothesis that for some $i \geq 1$, we have that $r^V \geq i$ implies $r^B \geq i$ and 
$\lambda_j^B = \lambda_j^V, d_j^B = d_j^V$ for all $j = 1, \cdots, i$. 
If $r^V < i + 1$, there is nothing to prove. If $r^V \geq i + 1$, we proceed: Proposition \ref{prop:mainProof}(a) implies 
\begin{align}\label{eq:Mi1V}
\Sigma^V_{M_{i+1}^V} = \Sigma_{M_{i + 1}^V}^B \, ;
\end{align}
 since $M_i^B = M_i^V$, it follows that $r^B \geq i + 1$, hence by Proposition \ref{prop:mainProof}(b) we have
\begin{align}\label{eq:Mi1B}
\Sigma^V_{M_{i+1}^B} = \Sigma_{M_{i + 1}^B}^B \, .
\end{align}
If $M_{i+1}^B \geq M_{i+1}^V$, then \eqref{eq:Mi1V} implies $d_{i+1}^V \lambda_{i + 1}^V = d_{i + 1}^V \lambda_{i + 1}^B$, hence $\lambda_{i+1}^V = \lambda_{i +1}^B = \lambda_{i+1}$. Applying this to \eqref{eq:Mi1B}, we obtain
\[
(d_{i+1}^B - d_{i+1}^V) \lambda_{i + 1} = \chi^V_{M_i^V + 1} + \cdots + \chi^V_{M_i^B} \,. 
\]
If $d_{i +1}^B > d_{i+1}^V$ or $d_{i + 1}^V > d_{i + 1}^B$, then as before we obtain a contradiction, and conclude $d_{i+1}^V = d_{i+1}^B =: d_{i +1}$. 
\end{proof}

\begin{proof}[Completing the proof of Theorem \ref{thm:allEqualLE}]
We have already shown that if $r^V = 0$, then $r^B = 0$ and the proof is complete. 
We now consider the case $r^V = \infty$; the case $0 < r^V <\infty$ is handled similarly and is omitted. 

In this case, Claim \ref{cla:LEequal} implies $r^B = \infty$ and 
\[
\lambda^V_i = \lambda^B_i =:  \lambda_i \quad d_i^V = d_i^B =: d_i 
\]
 for all $i \geq 1$. It remains to check that the identity
\[
F_i^V(x) = F_i^B(x) \cap V
\]
holds for a.e. $x \in X$. To start, assume $v \in F_i^V(x) \setminus \{ 0 \}$. Then, 
\[
\lambda_i \geq \lim_{n \to \infty} \frac1n \log \| A^n_x v\|_V \geq \limsup_{n \to \infty} \frac1n \log \| A^n_x v\|_B \, , 
\]
hence $v \in F_i^B$. We conclude $F_i^V(x) \subset F_i^B(x) \cap V$. 

For the opposite inclusion, assume $v \in V \setminus F_i^V(x)$; we will show $v \in F_i^B(x)$. For this, let $E \subset V$ be a shared $M_{i-1}$-dimensional complement to both of $F_i^V(x)$ and $F_i^B(x)$, and write $v = v_\perp + v_\| $ where $v_\| \in F_i^V(x), v_\perp \in E$, noting that $v \notin F_i^V(x) \implies v_\perp \neq 0$. Since $v_\perp \in E$, it holds that $v_\perp \notin F_i^B(x)$, and so $\lim_{n\to\infty} \frac1n \log \| A^n_x v_\perp\|_B \geq \lambda_{i-1}$. Meanwhile, 
$\lim_{n \to \infty} \frac1n \log \| A^n_x v_\|\|_B \leq \lim_{n \to \infty} \frac1n \log \| A^n_x v_\|\|_V \leq \lambda_i$, and so
\[
\lim_{n \to \infty} \frac1n \log \| A^n_x v\|_B \geq \max\left\{  \lim_{n \to \infty} \frac1n \log \|A^n_x v_\perp\|_B , \lim_{n \to \infty} \frac1n \log \| A^n_x v_\|\|_B \right\} \geq \lambda_{i-1} \,, 
\]
hence $v \notin F_i^B(x)$. This completes the proof. 
\end{proof}

{\begin{remark}[The nonergodic case]\label{rmk:nonergCase3}
We provide here a list of changes needed in the case when 
$(T, m)$ is not ergodic. In the following steps, the main difficulty is to deal
 with the possibility that
the values $r^W, \lambda_i^W, d_i^W$, etc., all depend on $x \in X$. 

\begin{enumerate}

	\item To start, we can reduce to the case where $r^V, r^B$ are both constant in $x$ by restricting the measure $m$ to sets of the form  
\[
\mathcal S = \{x \in X : r^V(x) = k^V, r^B(x) = k^B\}
\]
for arbitrary pairs $k^B, k^V \in \{ 0, 1, 2, \dots, \infty\}$. Sets of this form are $T$-invariant ($T^{-1} \Sc = \Sc$ up to $m$-measure zero sets) and so the restrictions $m_\Sc(K) := m(K \cap \Sc) / m(\Sc)$ are $T$-invariant. Since there are at-most countably many such sets $\Sc$, it suffices to prove Theorem \ref{thm:allEqualLE} for each $m_\Sc$ separately. 
	\item Suppose one has already restricted to a set of the form $\Sc$ for some fixed $k^B, k^V$. The volume rates $\Sigma^W_k$ and multiplicities $M_i^W$ are now functions of $X$, and so the analogue of Proposition \ref{prop:mainProof} is to show that
\begin{gather*}
	k^V \geq i \quad \implies \quad \Sigma_{M^V_i(x)}^V(x) = \Sigma_{M^V_i(x)}^B(x) \, , \\
	k^B \geq i \quad \implies \quad \Sigma_{M^B_i(x)}^B(x) = \Sigma_{M^B_i(x)}^V(x) \,. 
	\end{gather*}
As in the ergodic case, one starts by showing that $k^V \geq i$ allows to construct a cocycle $\hat A^n_x, x \in \mathcal S,$ quotienting along $F(x) = F_{i + 1}^V(x)$. One can build the quotient spaces $E(x)$ using Lemma \ref{lem:measSelectCompl} on restricting to sets of the form $\{ M^V_i(x) = \text{Const.}\}$. The proof of Lemma \ref{lem:quotientProjVoLGrowh} now proceeds with no real changes. The only change to the remainder of the proof of Proposition \ref{prop:mainProof}(a) is that one shows 
$\Sigma_{M^V_i(x)}^V(x) = \Sigma_{M^V_i(x)}^B(x)$ on a set of the form $U \cap T^{-1} U$, where $m(U) > 1 - \delta$. Taking $\delta \to 0$ completes the proof of part (a); part (b) is treated similarly. 
\item One checks that the case $k^B \neq k^V$ leads to a contradiction, implying $m\{ r^V = r^B\} = 1$. The arguments in Section \ref{subsec:proofOfEqualLE} now carry over without substantive changes. 
\end{enumerate}
\end{remark}
}

\section{Applications}\label{sec:Applications}

Here we outline in detail our two main applications of our main theorem to the problems of advection diffusion and the 2d Navier-Stokes equations.

\subsection{Preliminaries}
 Let $\T^d$ denote the $d$-dimensional torus, $d = 2$ or 3, parametrized by $[0,2 \pi)^d$. For $f : \T^d \to \R$, we define the 
standard $L^2$ inner product $\langle f, g \rangle_{L^2} = \int_{\T^d} f g\,\dee x$ with corresponding norm $\| f \|_{L^2} = \langle f, f \rangle^{1/2}_{L^2}$. Recall the Fourier transform $\Fc f = \hat f$ of an integrable $f : \T^d \to \R$ is given by 
    \[
    \hat f : \Z^d \to \C \, , \quad \hat f (k) = \frac{1}{(2 \pi)^d} \int_{\T^d} f(x) e^{- i k \cdot x} \dee x,
    \]
and that $\hat f(0) = 0$ if $f$ has zero mean; in this case we view
    $\hat f : \Z^d_0 \to \C$ where $\Z^d_0 = \Z^d \setminus \{ 0 \}$. 
We define the homogeneous $H^s$ norms for $s\in \R$ 
\[
\| f \|_{H^s} := \left(\sum_{k\in\Z^d_0}|k|^{2s}|\hat{f}(k)|^2\right)^{1/2},
\]
where $|k| := |k|_{\ell^2} = \left(\sum_{i=1}^d k_i^2\right)^{1/2}$.
This norm is equivalent to the $H^s$ norms when restricted to spaces of mean-zero functions. Since we will only consider mean zero functions in this paper, we will not make a point to distinguish the homogeneous and non-homogeneous Sobolev spaces. 
Lastly, when it is clear from context, all the above constructions will be applied to divergence free vector-valued functions in the usual way, namely $\mathbf{H}^s$ corresponds to the Hilbert space of velocity fields $u$ whose Fourier transform $\hat{u}(k) \in \C^d$ satisfies $\hat{u}(k)\cdot k = 0$ for each $k\in \Z^d_0$ and $\|u\|_{H^s}$ is defined as in the scalar case above with $|\cdot|$ instead denoting the norm on $\C^d$. 

Fix $\gamma > \frac{d}{2} + 1$ and let $\mathbf{H} = \mathbf{H}^\gamma$ be the Sobolev space of $H^\gamma$-regular, mean-zero, divergence-free vector fields on $\T^d$ with norm $\| \cdot\|_{H^\gamma}$ defined above. Note that $\gamma > \frac{d}{2} + 1$ implies the Sobolev embedding $H^{\gamma}\embeds W^{1,\infty}$, so all such velocity fields are at least (globally) Lipschitz. 
When it is clear from context, we will not distinguish when a norm is being applied to a scalar valued function or vector valued one.

\subsubsection{Skew product formulation}\label{subsub:skewprodform}

In what follows, we will describe a class of time dependent velocity fields that are subordinate to an ergodic measure preserving flow. The formulation we present is in the general setting of a \emph{skew-product}, defined below, which providing a natural dynamical framework for systems evolving on $\mathbf{H}$ driven by an external forcing, either random or deterministic.

Let $(\Omega, \mathscr F, \P)$ be a probability space and let $\theta^t : \Omega \circlearrowleft$ be a flow of measurable ergodic, $\P$-preserving transformations on $\Omega$. 
Assume that for each $t \geq 0$, we have a mapping $\tau^t : \Omega \times {\bf H} \circlearrowleft$ of the form
\[
\tau^t(\omega, u) = (\theta^t \omega, \Phi^t_\omega(u)),
\]
where $\Phi^t_{\cdot} : \Omega\times {\bf H} \to \Hbf$ is measurable and $u \mapsto \Phi^t_\omega(u)$ is continuous for all $\omega \in \Omega$. Moreover, we will assume 
$\Phi^t_\omega$ satisfies $\Phi^0_\omega(u) = u$ and the \emph{cocycle property}
\begin{align}\label{eq:cocycle}
\Phi^{t+r}_\omega = \Phi^r_{\theta_t\omega}\circ \Phi^t_{\omega},\quad t,r\geq 0 \, , 
\end{align}
for all $\omega \in \Omega$. 
Mappings of the form $\tau^t$ are referred to as \emph{skew product flows over $\theta^t : \Omega \to \Omega$}. 

Conceptually, for each $\omega \in \Omega$ we view the (potentially non-invertible) map $\Phi^t_\omega : {\bf H} \circlearrowleft$ as describing the evolution of a time-varying 
incompressible velocity field 
\[
u_t := \Phi^t_\omega(u_0).
\]
From this perspective, the set $\Omega$ encodes the set of nonautonomous driving or forcing paths $\omega = (\omega(t))_{t\geq 0}$ and $\theta^t$ denotes the time shift flow.  Equation \eqref{eq:cocycle} reflects that the forcing path evolving from time $t$ to time $t + r$ is given by $\theta^t \omega$. 

This framework includes a variety of evolution equations on $\Hbf$, including the 2d Navier-Stokes equations with either {\em stochastic} (e.g. white in time forcing) or {\em deterministic} (e.g. time-periodic) driving terms as  (see \cite{KS} Section 2.4.4 for a construction of such an RDS in the case of white in time forcing). Higher dimensional $(d\geq 3)$ examples of fluid motion can be considered by adding hyperviscosity or by Galerkin truncations. 
Other models can be formulated in terms of a skew-product flow and don't necessarily need to solve a fluid equation, e.g., time-stationary fields, time-periodic fields, and time dependent linear combinations 
$u_t(x) = \sum_{k} u_k(x) z^k(t)$
of fixed, time-independent vector fields $\{u_k\}$, where $\{z^k(t)\}$ are a collection of processes on $\R$ (e.g., Ornstein Uhlenbeck processes).


Lastly, throughout the following we will assume $\mathfrak m$ is a $\tau^t$-invariant measure on $(\Omega \times \Hbf, \mathscr F \times \operatorname{Bor}(\Hbf))$ such that 
\begin{align}\label{marginalEq}
\mathfrak m(A \times \mathbf{H}) = \P(A) \, ,  \quad A \in \mathscr F \, .
\end{align}
 The measure $\mathfrak m$ captures the statistics of typical velocity fields $(u_t)$ with respect to the model, while \eqref{marginalEq} reflects that $\P$ is the law of the 
 underlying driving. For additional discussion on ergodicity and invariance, see Section \ref{subsubsec:ergTheoryBackground}.

\subsection{Passive scalar advection linear cocycle}\label{subsec:passivescalar}

Let $\tau^t : \Omega \times \Hbf \to \Omega \times \Hbf$ be as above. 
Given a fixed initial $(\omega, u_0) \in \Omega \times \Hbf$ and $t > 0$, define $u_t := \Phi^t_\omega(u_0)$. For $\kappa > 0$, we are interested in solutions $(f_t)_{t \geq 0}$ to the passive scalar advection diffusion equation 
\begin{align}\label{eq:passiveScalarEq}
\pd_tf_t + u_t \cdot \nabla f_t = \kappa \Delta f_t \, 
\end{align}
for fixed initial mean-zero scalars $f_0 : \T^d \to \R$. 
Being a parabolic equation with Lipschitz velocity field, well-posedness of \eqref{eq:passiveScalarEq} $H^s$ for any $s\geq0$ is classical. One can extend to $H^{-s}$ by the density of $H^{s}$ in $H^{-s}$, linearity of the equation, and the $L^2$ duality of $H^s$ and $H^{-s}$. \footnote{A version of this argument is carried out in the proof of Lemma \ref{lem:regularityPSA2}.} 
 
\begin{proposition}\label{prop:well-posed}
For any $s \in \R$ and $(\omega, u_0) \in \Omega \times \Hbf$, there is a semiflow of compact linear operators $S^t_{\omega, u_0} : H^s \circlearrowleft, t \geq 0,$ such that $f_t := S^t_{\omega, u_0}f_0$ is a solution to \eqref{eq:passiveScalarEq} with initial data $f_0 \in H^s$. Moreover, the operators $S^t_{\omega, u_0}$ form a linear cocycle over $\tau^t$: for $r, t \geq 0$ we have
\[
S^{t + r}_{\omega, u_0} = S^r_{\theta^t \omega, u_t}\circ S^t_{\omega, u_0} 
\quad \text{ for all } \quad (\omega, u_0) \in \Omega \times \Hbf \,. 
\]
\end{proposition}

Our first result says that for measure-typical velocity fields $(u_t)$, the
asymptotic exponential growth (or decay) rate of $(f_t)$ exists for any initial scalar $f_0$, and that `typical' initial scalars see only a single exponential growth rate $\lambda_1$ independent of the $H^s$ norm one uses.

\begin{atheorem}\label{thm:ratesTheSamePSA}
Let $\gamma^\prime \geq \gamma > 1+d/2$ and assume that there is an invariant measure $\mathfrak m$ for $\tau^t : \Omega \times \Hbf^\gamma \circlearrowleft$ that satisfies the following mild moment condition: 
\begin{align}\label{eq:momentCondition11}
\int \left(\int_0^1 \| u_s \|_{H^{\gamma^\prime}} \dee s \right) \dee \mathfrak m(\omega, u_0) < \infty \,. 
\end{align}
Let $\kappa > 0$ be fixed. Then the following hold:
\begin{itemize}
    \item[(a)] For $\mathfrak m$-a.e. $(\omega, u_0) \in \Omega \times \Hbf^\gamma$ and for all $s \in [-\gamma^\prime,\gamma^\prime]$ and mean-zero $f_0 \in H^s$, the global solution $f_t = S^t_{\omega, u_0} f_0$ to \eqref{eq:passiveScalarEq} has the property that the limit
    \begin{align}\label{eq:limitPSA123}
    \lambda(\omega, u_0; f_0) := \lim_{t \to \infty} \frac1t \log \| f_t \|_{H^s} \in [-\infty, \infty)
    \end{align}
    exists and is independent of $s\in [-\gamma^\prime,\gamma^\prime]$. 
    
    \item[(b)] If $\mathfrak{m}$ is also ergodic, then there exists $\lambda_1 \in \R \cup \{ -\infty\}$ and $N \in \Z_{\geq 0}$, each depending only on $\kappa$,
     with the following property: for all $s\in[-\gamma^\prime,\gamma^\prime]$, for $\mathfrak{m}$-a.e. $(\omega, u_0) \in \Omega \times \Hbf$, and for
     all $f_0$ chosen {\em off} of an $N$-codimensional subspace of $H^s$, we have
     \[
     \lambda_1 = \lim_{t \to \infty} \frac1t \log \| f_t \|_{H^s}.
     \]
\end{itemize}
\end{atheorem}
\begin{remark}
In contrast to classical parabolic regularity theory, which gives $H^s$ regularity of $f_t$ for all $s\geq 0$ as long as $u_t\in \Hbf$ locally uniformly in $t$, Theorem \ref{thm:ratesTheSamePSA} requires more quantitative regularity estimates in terms of $u_t$. Consequently, the range of $s$ to which equality of exponents applies is constrained by the regularity of $u$ where certain moments are available.
\end{remark}

\begin{proof}
As an immediate implication of Lemmas \ref{lem:regularityPSA1} and \ref{lem:regularityPSA2} below, we obtain that for any $-\gamma^\prime \leq s  \leq \gamma^\prime$, we have that $S^t_{\omega, u_0}$ is a semiflow of compact linear operators in $H^s$. By Lemmas \ref{lem:regularityPSA1} and \ref{lem:regularityPSA2}, equation \eqref{eq:momentCondition11} implies the logarithmic moment estimate
\begin{align}\label{eq:logMomEstPSA}
\int \log^+ \| S^1_{\omega, u_0} \|_{H^{ s}} \dee \mathfrak m (\omega, u_0) < \infty \, ,  \, , 
\end{align}
which implies that the MET (Theorem \ref{thm:nonInvMET}) applies to $S^t_{\omega, u_0}$ as a linear cocycle over $(\tau^t, \mathfrak{m})$ along integer times $t$. 
For the limits \eqref{eq:limitPSA123} taken along integer times, parts (a) and (b) now follow immediately from Theorem \ref{thm:allEqualLE}. To pass from discrete to continuous-time limits in part (a), it suffices\footnote{This sufficient condition for passing from discrete to continuous time Lyapunov exponents is classical; see, e.g., \cite{lian2010lyapunov}.} that the cocycle $S^t_{\omega, u_0}$ satisfies
\begin{align}\label{discToContinuousTime}
\log^+ \sup_{t \in [0,1]} \| S^t_{\omega, u_0} \|_{H^s} , \quad 
\log^+ \sup_{t \in [0,1]} \| S^{1-t}_{\theta^t \omega, u_t} \|_{H^s} \quad \in L^1(\mathfrak{m}) \,. 
\end{align}
This too follows from Lemmas \ref{lem:regularityPSA1} and \ref{lem:regularityPSA2}. 
\end{proof}

As discussed in Section \ref{subsubsec:passiveScalarIntro}, at $\kappa = 0$ 
	equality of exponents does not hold. This suggests that as $\kappa \to 0$ the ``rate'' at which the Lyapunov exponent is realized in $H^s$ depends heavily on $s$. For example, although the exponents in $H^1$ and $L^2$ agree and are negative as $t \to \infty$, there is a $\kappa$-dependent transient timescale along which the $H^1$ norm \emph{increases} before decay starts \cite{MC18}, while the $L^2$ norm can only decrease. 
The following provides a way of quantifying this $\kappa$-dependence. 
For simplicity we state the result in the case when $(\tau^t, \mathfrak m)$ is ergodic and the comparison between $L^2$ and $H^s, s > 0$. 

For $\epsilon > 0, s \in [0, \gamma']$, define the Lyapunov regularity functions\footnote{Recall the definition of the minimal angle $\angle^{H^s}$ in \eqref{defn:angleHilbSpace}.}
\begin{align}\label{eq:defnPSALEregularity}
\overline{D}_{\e, \kappa}^{H^s}(\omega, u_0) = \sup_{n \in \Z_{\geq 0}} \frac{\| S^n_{\omega, u_0} \|_{H^s}}{e^{n (\lambda_1 + \e)}} \, , \quad 
\underline{D}_{\e, \kappa}^{H^s}(\omega, u_0) = \sup_{n \in \Z_{\geq 0}} \frac{e^{n (\lambda_1 - \e) }\sin \angle^{H^s} (v, F_{i + 1}^{H^s}(x)) }{\| S^n_{\omega, u_0} \|_{H^s}} \, .
\end{align}

\begin{acorollary}
Assume the setting of Theorem \ref{thm:ratesTheSamePSA}, and in addition, that $(\tau^t, \mathfrak m)$ is ergodic. Fix $p > 3$ and $0 < q < \frac{p^2 - 3 p}{p-1}$, and assume the moment condition
\[
\mathcal I := \int \left( \int_0^1(1 +  \| u_\tau\|_{H^{\gamma}}) d \tau \right)^p \dee \mathfrak{m}(\omega, u_0) <\infty \,. 
 \]
Then, for any $\delta, \kappa > 0$ and $- \gamma' \leq s' < s \leq \gamma'$, there exists a function $K_{\delta, \kappa}^{s', s} : \Omega \times \Hbf \to [1,\infty)$ such that for any 
 \begin{align*}
\overline{D}^{H^s}_{\e, \kappa}   \leq K_{\delta, \kappa}^{s', s} \, \overline{D}^{H^{s'}}_{\e + \delta, \kappa} \, , \qquad  
\underline{D}^{H^{s'}}_{\e, \kappa}  \leq K_{\delta, \kappa}^{s', s} \, \underline{D}^{H^s}_{\e + \delta, \kappa}  \, , 
\end{align*}
and the following moment condition holds: 
\[
\int (\log^+ K_{\delta, \kappa}^{s', s})^q \dee m \lesssim_{p, q} \delta^{- (p-q)} \left(  1 + (s-s') | \log \kappa| + \mathcal I \right) \,. 
\]
\end{acorollary}
The proof is a straightforward consequence of Lemmas \ref{lem:regularityPSA1},  \ref{lem:regularityPSA2} and Corollary \ref{cor:reguarlityFunctions}. 

\subsection{2d Navier-Stokes and its linearization cocycle}\label{subsec:NavierStokes}

We turn attention now to linearization along solutions to evolution equations governing the dynamics of the velocity field $u_t$ itself. While much of what we say here can be extended to different evolution equations, we focus in this manuscript on trajectories of the 2d incompressible Navier-Stokes equations on $\T^2$: 
\begin{equation}\label{eq:Navier-Stokes-general}
    \partial_t u + (u\cdot \nabla) u  =\nu \Delta u -\nabla p + F,\quad \Div u = 0.
\end{equation}
Here $p$ is the pressure enforcing the divergence free constraint, $\nu >0$ is the kinematic viscosity and $F$ is a spatially smooth body forcing which we will take to be either stochastic and white-in-time or periodic in time. We will assume throughout that the forcing $F$ and solutions $u_t$ are mean-zero on $\T^2$.

In what follows, we present two cases where the 2d Navier-Stokes equations give rise to a skew-product flow $\tau^t$ in the sense of Section \ref{subsub:skewprodform}:  periodic forcing and white in time stochastic forcing (both of which we assume to be additive). We will then study the cocyle associated to its linearization in vorticity form and present Theorem \ref{thm:ratesTheSameNSE} concerning Lyapunov exponents taken in $H^s$ as $s$ varies. 

\subsubsection{Periodic forcing}
Below, we formulate evolution by the Navier-Stokes equations in the skew product formulation of Section \ref{subsub:skewprodform} in the case of additive, time-periodic, spatially regular forcing. 
In this case, $\Omega = \Sb^1$, where the circle $\Sb^1$ is parametrized by $[0,1)$ with the endpoints identified. The time shift $\theta^t : \Omega\circlearrowleft$ is given by $\theta^t\omega = \omega + t \mod 1$, while the measure $\P$ is normalized Lebesgue measure. 

The following well-posedness and regularity results on Sobolev spaces are classical (see for instance \cite{Temam1997-xm},\cite{Robinson2001-tk},\cite{KS}).
\begin{proposition}[\cite{KS} Theorem 2.1.19] \label{prop:NSEwellPosedDet}\ 
\begin{itemize}
\item[(i)] Fix an integer $m \geq 2$ and $F \in L^2([0,\infty), \Hbf^{m-1})$. For each fixed initial $u_0\in \Hbf^0$ and for all $\e > 0$  there exists a unique solution $u \in C([\ep,T];H^m)\cap L^2([\ep,T];H^{m+1})$ for each $\ep >0$ and $T\geq 0$.
Moreover, there exists a constant $C_m$ such that the following inequality holds for each $0\leq t\leq T$: 
\begin{equation}\label{eq:NEreg-estimate}
\begin{aligned}
    t^m\|u_t\|_{H^m}^2 + &\int_0^t s^m\|u_s\|_{H^{m+1}}\ds \leq \int_0^t s^m\|F_s\|_{H^{m-1}}\ds\\
    & + C_m\left(\|u_0\|_{L^2} + \|u_0\|_{L^2}^{4m+2} + \|F\|_{L^2([0,T];L^2)}^2 + \|F\|^{4m+2}_{L^2([0,T];L^2)}\right)
    \end{aligned}
\end{equation}

\item[(ii)] For each $0 \leq r \leq m$, there exists a continuous mapping $\Phi^t_{\cdot}\,:\, \Omega\times \Hbf^r\to \Hbf^r$, over $\theta^t$, $(\Omega,\mathcal{F},\P)$ such that $u_t = \Phi^t_\omega(u_0)$ is the unique solution to \eqref{eq:Navier-Stokes-general} with initial data $u_0$ and forcing $F_{t+\omega}$.  Moreover, $\Phi^t_\omega:\Hbf^r\to \Hbf^r$ in injective and $C^1$ Fr\'ech\'et differentiable for all $\omega \in \Omega, t \geq 0$. 
\end{itemize}
\end{proposition}

Since $F_t$ is periodic, it is natural to consider the time-one map $\Phi^1_0 : L^2 \circlearrowleft$ since $\theta^1\omega = \omega$ and therefore $\Phi^n_0 = \Phi^1_0\circ\cdots\circ \Phi^1_0$ for $n \in \Z_{\geq 0}$. 
This mapping admits a compact \emph{global attractor} $\mathcal A$ to which solutions converge: 
\begin{corollary}
Assume $F_t \in \Hbf = \Hbf^{\gamma}$ for all $t \in [0,1]$. 
\begin{itemize}
\item[(a)] The mapping $\Phi^1_0$ admits a compact \emph{global attractor} $\mathcal A \subset \Hbf$. Precisely, (i) $\Phi^1_0(\mathcal A) = \mathcal A$, (ii) $\Phi^1_0 |_{\mathcal A} : \mathcal A \circlearrowleft$ is a homeomorphism, and (iii) for all $u_0 \in \Hbf$ and $\omega \in \Omega$, any subsequential limit $u_* = \lim u_{n_k}$ of the trajectory $(u_n)_{n \geq 0}$ belongs to $\mathcal A$. 
\item[(b)] There exist invariant probability measures $\mu$ for $\Phi_0^1$. 
Moreover, all such invariant measures are supported on $\mathcal A$. 
\end{itemize}
\end{corollary}
\begin{proof}
Equation \eqref{eq:NEreg-estimate} can be easily seen via a Gr\"onwall argument to show that $\Phi^1_0$ is \emph{dissipative} on $\mathbf{H} := H^\gamma$ in the sense that it admits a \emph{compact absorbing ball} $B = \{\|Du\|_{\mathbf{H}} \leq c\}\subseteq \mathbf{H}$. Specifically, there exists $c = c_\nu > 0$ and for each bounded subset $S\subseteq \Hbf$ a time $n_0 = n_0(S) \in \Z_{\geq 0}$ such that
\[
    \Phi^n(S) \subset B \quad \text{for all}\quad n \geq n_0(S).
\]
By standard theory of global attractors (see, e.g. \cite{Robinson2001-tk} Theorem 10.5), this implies that there is a unique, maximal, and compact global attractor $\mathcal{A}\subseteq \Hbf$ left invariant under $\Phi_0^1$ (i.e., $\Phi_0^1(\mathcal A) = \mathcal A)$. That $\Phi_0^1|_{\mathcal A}$ is a homeomorphism now follows from compactness of $\mathcal A$ and the (classical) fact that $\Phi_0^1$ is injective (see \cite{Robinson2001-tk} Theorem 10.6). 

The Krylov-Bogoliubov argument (c.f. Example \ref{exam:simpleInvMeasSetup}) implies 
existence of invariant measures for $\Phi_0^1$. That all such invariant measures are supported on $\mathcal A$ follows from the item (a)(iii); further details omitted. 
\end{proof}

 Given a $\Phi_0^1$-invariant measure $\mu$, we define an associated measure $\mathfrak{m}$ on $\Omega\times \mathbf{H}$ via
\begin{equation}\label{eq:disintigration-def}
    \dee \mathfrak{m}(\omega,u) = \dee \mu_{\omega}(u)\dee \P(\omega),\quad \text{where}\quad \mu_{\omega} := (\Phi^\omega_0)_*\mu,
\end{equation}
and $(\Phi^\omega_0)_*\mu := \mu \circ (\Phi^\omega_0)^{-1}$ is the pushforward of the measure $\mu$ under the map $\Phi^\omega_0$. We see that the measure $\mathfrak{m}$ is an ergodic invariant measure for the skew-product flow
$
    \tau^t(\omega,u) = (\theta^t\omega,\Phi^t_\omega(u))
$, 
and therefore we are in the general setup of section \ref{subsub:skewprodform}. The proof of this straightforward, but does not appear to be in the literature. We include it below for completeness.
\begin{proposition}
Let $\mu$ be a probability measure on $\mathbf{H}$, and define $\mathfrak{m}$ and $\mu_\omega$ as in \eqref{eq:disintigration-def}
\begin{enumerate}
\item If $\mu$ is $\Phi^1_0$-invariant, then $\mathfrak{m}$ is $\tau^t$-invariant.
\item If $\mu$ is $\Phi^1_0$-ergodic, then $\mathfrak{m}$ is $\tau^t$-ergodic
\end{enumerate}
\end{proposition}
\begin{proof}

For 1: Let $\mu$ be the $\Phi^1_0$-invariant and let $\mathfrak{m}$ be as in \ref{eq:disintigration-def}. We first show that $(\Phi^t_\omega)_* \mu_\omega = \mu_{\theta^t \omega}$:
\[
(\Phi^t_\omega)_* \mu_\omega = (\Phi^{t + \omega}_0)_* \mu = (\Phi^{\{ t + \omega\}}_0)_* \mu = \mu_{\{ t + \omega\}} = \mu_{\theta^t \omega}
\]
where here $\{ r\} = r - \lfloor r \rfloor$ is the fractional part of $r \in [0,\infty)$. Invariance of $\mathfrak{m}$ with respect to $\tau^t$ is now standard (see \cite{KS} Prop 1.2.27).

For 2: Assume $\mu$ is $\Phi^1_0$-ergodic. We begin by checking that $\mu_\omega$ is ergodic for the map $\Phi^1_\omega:\Hbf\circlearrowleft$ for all $\omega \in \Omega$. For this, let $\varphi : X \to \R$ be $\Phi^1_\omega$-invariant $\mu_\omega$-a.s. Then, 
$\hat \varphi := \varphi \circ \Phi^\omega_0$ is $\Phi^1_0$-invariant $\mu$-a.s., since
\[
\hat \varphi \circ \Phi^1_0 = \varphi \circ \Phi^\omega_0 \circ \Phi^1_0 
= \varphi \circ \Phi^1_\omega \circ \Phi^\omega_0
\]
and $\varphi$ is $\Phi^1_\omega$-invariant $\mu_\omega$-a.e.. We conclude $\hat \varphi$ is $\mu$-a.s. constant, hence $\varphi$ is $\mu_\omega$-a.s. constant. 

We now check that $\mathfrak{m}$ is $\tau^t$-ergodic. To start, fix a $\tau^t$-invariant $\psi : \Omega \times \mathbf{H} \to \R$. We see that for each $\omega \in \Omega$,
\[
\psi(\omega, u) = \psi(\theta^1 \omega, \Phi^1_\omega (u)) = \psi(\omega, \Phi^1_\omega(u)) \quad \mu_\omega\text{ -a.e.}
\]
We conclude that for each fixed $\omega$, the function $\varphi_\omega(u) = \psi(\omega, u)$ is $\Phi^1_\omega$-invariant, hence $m_\omega$-a.e. equivalent to a constant $\varphi_\omega$. Now, $\psi(\omega, u) = \varphi_\omega$ holds $\mathfrak{m}$-a.e., and as $\psi$ is $\tau^t$-invariant we have that $\omega \mapsto \varphi_\omega$ is $\theta^t$-invariant. By ergodicity of $\P$ for $\theta^t$, we obtain that $\varphi_\omega$ is $\P$-a.s. constant. In all, we conclude $\psi$ is $\mathfrak{m}$-a.e. constant and therefore $\mathfrak{m}$  is $\tau^t$-ergodic.
\end{proof}

\subsubsection{White-in-time forcing}

Next we consider the white-in-time stochastically forced case. Specifically, we will assume that the forcing $F$ is the time derivative of a Brownian process in $\mathbf{H}$: 
\begin{equation}\label{eq:white-noise-force}
    F = \partial_t\xi \, , \quad \xi(t,x) = \sum_{j=1}^\infty \sigma_j e_j(x) \beta_j(t) \, ,
\end{equation}
where $\{\beta_j\}$ are independent canonical 1d Wiener processes, $\{e_j\}$ forms an orthonormal basis for $\mathbf{H}$, and the coefficients ${\bf \sigma} = \{\sigma_j\}$ satisfy $ \|{\bf \sigma}\|_{\ell^2}^2 = \sum_{j=1}^\infty \sigma_j^2 < \infty$. This last condition ensures that the process $\xi_t = \xi(t,\cdot)$ is a continuous process in $\mathbf{H}$. Since $\xi_t$ is not differentiable in time, the Navier-Stokes equations must be interpreted in a time-integrated sense: 
\[
    u_t - u_0 + \int_0^t\left((u_s\cdot\nabla)u_s - \nu \Delta u_s + \nabla p_s\right)\ds = \xi_t,
\]
where equality holds in $H^{-1}$ with probability $1$.

In this setting we take $\Omega$ to be the space $C_0(\R_+;\Hbf)$ of continuous one-sided paths $\omega :\R_+ \to \Hbf$ vanishing at $0$, with the standard Borel sigma algebra $\mathscr{F}$ and equipped with a Gaussian measure $\P$ whose projection onto basis elements $e_j$ through the map $\omega \mapsto \langle e_j,\omega\rangle_{\Hbf}$ is the canonical Wiener measure. We define the semiflow $\theta^t:\Omega\circlearrowleft$ to be the shift map
\[
    (\theta^t\omega)_s = \omega_{t+s} - \omega_{t},\quad t,s\in \R_+,
\]
which is easily seen to leave the measure $\P$ invariant.

The following well-posedness, regularity and construction of an RDS is well-known (see e.g. \cite{KS} \S 2.4 ).
\begin{proposition}\label{prop:SNS-wellpossed}
Let $d = 2$, $\gamma> 1+ \frac{d}{2}=2$, and suppose that $F$ is of the form \eqref{eq:white-noise-force} where $\sum_{j} j^{2 (\gamma - 1)} \sigma_j^{2} < \infty$ (hence $\xi_t \in \Hbf^{\gamma - 1}$ with probability 1). Then, there exists a measurable mapping $\Phi^t_{\cdot}\,:\, \Omega\times \Hbf\to \Hbf$ such that $u_t = \Phi^t_\omega(u_0)$ is a strong pathwise solution to \eqref{eq:Navier-Stokes-general} (in the integral sense) with initial data $u_0$ and noise path $\omega = (\xi_t)$. Moreover, for $\P$-a.e. $\omega$, the mapping $\Phi^t_\omega: \Hbf \to \Hbf$ is injective and $C^1$ Fr\'ech\'et differentiable, and satisfies the cocycle property
\[
\Phi^{t + r}_\omega = \Phi^r_{\theta^t \omega} \circ \Phi^t_\omega \quad \text{ for all } r, t \geq 0, \text{and for all } \omega \in \Omega \, .
\] 
Lastly,  $\omega \mapsto \Phi^t_\omega$ only depends on $\omega|_{[0,t]}$.
\end{proposition}

That $\Phi^t_\omega$ depends only on $\omega|_{[0,t]}$ implies that it is \emph{Markovian}, in the sense that $u_t = \Phi^t_{\omega}$ is a Markov process.
We say that a probability measure $\mu$ on $\mathbf{H}$ is a \emph{stationary measure} for this Markov process if
\[
    \E(\Phi^t_\omega)_*\mu = \mu \, .
\]
Existence of such a measure for dissipative RDS is generally guaranteed by a simple Krylov Bogoliubov argument (see, e.g. \cite{KS}). However, in contrast to the deterministic forcing case, it is often the case that this measure is in fact unique under fairly mild conditions on the noise (see e.g. \cite{Flandoli1995-mo,HM06}). It is well known that stationary measures $\mu$ are in one-to-one correspondence with the invariant measure $\mathfrak{m} = \P\times\mu$ on $\Omega\times\mathbf{H}$ for the skew product flow $\tau^t(\omega,u) = (\theta^t\omega,\Phi^t_\omega(u))$.

\begin{theorem} \
\begin{itemize} 
\item[(a)] (Theorem 4.2.9 \cite{KS}) Let $\mu$ be a stationary probability measure for $(u_t)$. Then, 
$\mathfrak m = \P \times \mu$ is an invariant measure for the semiflow $\tau^t : \Omega \times \Hbf$. 
\item[(b)] (Theorem I.2.1 \cite{kifer2012ergodic}) If $\mu$ is the \emph{unique} stationary measure for $(u_t)$, then $\P\times\mu$ is ergodic for $\tau^t$. 
\end{itemize}
\end{theorem}

Moreover depending on the regularity of $F_t$ one can obtain moment estimates of higher Sobolev norms with respect the the stationary measure $\mu$. The following estimate is a consequence of \cite{Kuksin2003-st} (also c.f. Exercise 2.5.8 \cite{KS}).

\begin{proposition}\label{prop:moment-est}
Suppose $\sum_j j^{2r} \sigma_j^2 < \infty$ for some $r \geq \gamma$. Then, any  stationary measure $\mu$ for \eqref{eq:Navier-Stokes-general} satisfies the estimate
\[
    \int \left( \sup_{t \in [0,1]} \|u_t\|_{H^r}\right)^p\dee \mathfrak m(\omega, u_0) < \infty \, .
\]
for all $p \geq 0$. 
\end{proposition}

\subsubsection{Linearized Navier-Stokes}

In either the time-periodically forced or stochastically forced setting, we will assume below that $u_t = \Phi^t_\omega(u_0)$ on $\mathbf{H}$. Our main goal is to study the linearized Navier-Stokes cocycle given by the Fr\'ech\'et derivative $D_{u_0}\Phi^t_\omega$ which acts as a linear operator on divergence free velocity fields. Specifically, given an initial divergence free velocity $v_0 \in \mathbf{H}$ (viewed as an infinitesimal perturbation), the trajectory $v_t = (D_{u_0}\Phi^t_{\omega})v_0$ satisfies the {\em linearized} or {\em first variation equation}
\begin{equation}\label{eq:linns}
    \partial_t v + (u\cdot \nabla)v + (v\cdot\nabla)u = \nu \Delta v - \nabla q \, ,  \quad \Div v = 0 \, , 
\end{equation}
where $q$ is the pressure enforcing the divergence-free constraint on $v$. Hence $D_{u_0}\Phi^t_\omega$ is the solution operator to the above linear equation and defines a compact linear co-cycle on $\mathbf{H}$. By uniqueness of solutions to \eqref{eq:linns}, the $D_{u_0} \Phi^t_\omega$ satisfy the cocycle property
\begin{align}\label{cocyclePropNSE4}
D_{u_0}\Phi^{t + r}_{\omega} = D_{u_t}\Phi^r_{\theta^t \omega}\circ D_{u_0}\Phi^t_{\omega} \quad \text{ for all } r, t \geq 0 \,. 
\end{align}

To apply Theorem \ref{thm:allEqualLE} in this setting, we want to treat $D_{u_0} \Phi^t_\omega$ as a cocycle over $\Hbf^s$ for a range of $s$, given a \emph{fixed} base $\Phi^t_\omega$ on $\Hbf = \Hbf^\gamma$, where $\gamma > 2$ is fixed. 
 The following summarizes what is needed. 
 
\begin{proposition}\label{prop:lin-well-posed}
Fix $\gamma' \geq \gamma, \omega \in \Omega, u_0 \in \Hbf$ and suppose
 $u_t = \Phi^t_\omega (u_0)$ is a solution 
to the 2d Navier-Stokes equations \eqref{eq:Navier-Stokes-general} 
in the setting of either Proposition \ref{prop:NSEwellPosedDet} or 
\ref{prop:SNS-wellpossed}. Assume that for $\mathfrak m$-a.e. $(\omega, u_0) \in \Omega \times \Hbf$, the solution $u_t = \Phi^t_\omega(u_0)$ satisfies $(u_t) \in L_{\rm loc}^1([0,\infty), \Hbf^{\gamma' + 2})$. Then:  
\begin{itemize}
	\item[(a)] The mapping $D_{u_0}\Phi^t_\omega:\mathbf{H}\circlearrowleft$ extends\footnote{When there is no confusion, we will abuse notation somewhat and write $D_{u_0} \Phi^t_\omega : \Hbf^s \to \Hbf^s$ for the extended operator.} to a compact bounded linear operator on $\Hbf^s$ such that $v_t := D_{u_0}\Phi^t_\omega v_0$ is a solution to \eqref{eq:linns} with initial data $v_0 \in \Hbf^s$. Equation  \eqref{cocyclePropNSE4} is satisfied as a cocycle on $\Hbf^s$. 
	\item[(b)] The mapping $(\omega, u_0) \mapsto D_{u_0} \Phi^t_\omega$ is strongly measurable in $\Hbf^s$. 
\end{itemize}
\end{proposition}
\begin{proof}[Proof sketch]
Since \eqref{eq:linns} is parabolic (up to a compact perturbation), 
well-posedness in $\Hbf^s, s \geq 0$ is standard (see \cite{I_Udovich1989-jm}), while well-posedness in $\Hbf^{-s}$ can be proved via a duality argument and using linearity of the equation. 
A priori estimates sufficient to deduce these statements for a range of $s$ related to the regularity of $u$ are presented in Section \ref{subsubsec:linns-est}.
\end{proof}

We are now in position to state our results on the Lyapunov exponents of 
2d Navier-Stokes. 

\begin{atheorem}\label{thm:ratesTheSameNSE}
Let $\nu > 0$ be fixed. Let $\tau^t:\Hbf \times \Omega \circlearrowleft$ be the skew product semiflow associated to the Navier-Stokes equations and let $\mathfrak{m}$ be a $\tau^t$-invariant probability measure that satisfies the following moment condition for some $\gamma^\prime 
\geq \gamma > 1+ d/2 = 2$: 
\begin{equation}\label{eq:velocity-moment-est}
\int \left( \int_0^1 \|u_t\|_{H^{\gamma^\prime + 2}} \dee t\right) \, \dee \mathfrak m(\omega, u_0) < \infty \,. 
\end{equation}
Then: 
\begin{itemize}
    \item[(a)] For $\mathfrak m$-a.e. $(\omega, u_0) \in \Omega \times \Hbf$ and for all  $s \in [-\gamma^\prime + 1,\gamma^\prime + 1]$ and mean-zero divergence free velocity fields $v_0 \in \Hbf^s$ , the global solution $v_t = D_{u_0}\Phi^t_\omega v_0$ to \eqref{eq:passiveScalarEq} has the property that the limit
    \begin{align}\label{eq:LEequalNSE}
    \lambda(\omega, u_0; v_0) := \lim_{t \to \infty} \frac1t \log \| v_t \|_{H^s}
    \end{align}
    exists (note the limit $-\infty$ is possible) and is independent of $s\in [-\gamma^\prime + 1,\gamma^\prime + 1]$. 
    
    \item[(b)] If $\mathfrak{m}$ is also $\tau^t$-ergodic, then there exists $\lambda_1 \in \R \cup \{ -\infty\}$ and $d \in \Z_{\geq 0}$, each depending only on $\mathfrak m$ and $\nu$, with the following property: for all $s \in [-\gamma^\prime + 1,\gamma^\prime + 1]$, for $\mathfrak m$-a.e. $(\omega, u_0) \in \Omega \times \Hbf$, and for
     all $v_0$ chosen {\em off} of a $d$-codimensional subspace of $\Hbf^s$, we have
     \[
     \lambda_1 = \lim_{t \to \infty} \frac1t \log \| v_t \|_{H^s} \,. 
     \]
\end{itemize}
\end{atheorem}
\begin{proof}
Lemmas \ref{lem:LNS1} and \ref{lem:LNS2} and the estimate \eqref{eq:velocity-moment-est} imply immediately that $\| D_{u_0} \Phi^1_\omega\|_{H^s}$ satisfies the analogue of the logarithmic moment estimate \eqref{eq:logMomEstPSA} for any $s \in [- \gamma' + 1, \gamma' + 1]$.  
The MET (Theorem \ref{thm:nonInvMET}) and Theorem \ref{thm:allEqualLE} apply, implying convergence of the limits \eqref{eq:LEequalNSE} and independence from $s$ when taken along integer times. Passing from discrete to continuous time follows similarly, using the analogue of \eqref{discToContinuousTime}. 
\end{proof}

\begin{remark}
Note that in light of the regularizing properties of Navier-Stokes the moment condition \eqref{eq:velocity-moment-est} for solutions to Navier stokes is ultimately a condition on the regularity of the force. 
In the periodically forced case, it is sufficient for $F_t \in \Hbf^{\gamma+1}$ for all $t$,  due the to fact that there is an absorbing ball in $\Hbf^{\gamma+2}$.
However, in the stochastically forced case, we require that $F_t$ belongs to $\Hbf^{\gamma + 2}$, so that the moment bound \eqref{eq:velocity-moment-est} follows from Proposition \ref{prop:moment-est} under the condition that $\sum_j j^{2 ( \gamma^\prime + 2)} \sigma_j^2 < \infty$. 
\end{remark}

\bigskip

We now apply Corollary \ref{cor:reguarlityFunctions} concerning Lyapunov regularity functions to the 
Navier-Stokes cocycle. 
For $\epsilon, \nu > 0$ and $s \in [- \gamma' + 1, \gamma' + 1]$, let 
$\overline{D}^{H^s}_{\epsilon, \nu}, \underline{D}^{H^s}_{\epsilon, \nu} : \Omega \times \Hbf \to [1,\infty)$ denote the Lyapunov regularity functions for the cocycle $D_{u_0}\Phi^t_\omega$, defined analogously to \eqref{eq:defnPSALEregularity}. 

\begin{acorollary}
Assume the setting of Theorem \ref{thm:ratesTheSameNSE}, and in addition, that $(\tau^t, \mathfrak m)$ is ergodic. Fix $p > 3$ and $0 < q < \frac{p^2 - 3 p}{p-1}$, and assume the moment condition
\[
\mathcal I_p := \int \left( \int_0^1(1 +  \| u_\tau\|_{H^{\gamma + 2}}) \dee \tau \right)^p \dee \mathfrak m(\omega, u_0) <\infty \,. 
 \]
Then, for any $\delta, \nu > 0$ and $- \gamma' + 1 \leq s' < s \leq \gamma' + 1$, there exists a function $K_{\delta, \nu}^{s', s} : \Omega \times \Hbf \to [1,\infty)$ such that 
 \begin{align*}
\overline{D}^{H^s}_{\e, \kappa}   \leq K_{\delta, \kappa}^{s', s} \, \overline{D}^{H^{s'}}_{\e + \delta, \kappa} \, , \qquad  
\underline{D}^{H^{s'}}_{\e, \kappa}  \leq K_{\delta, \kappa}^{s', s} \, \underline{D}^{H^s}_{\e + \delta, \kappa}  \, , 
\end{align*}
and the following moment condition holds: 
\[
\int (\log^+ K_{\delta, \kappa}^{s', s})^q \dee \mathfrak m \lesssim_{p, q} \delta^{- (p-q)} \left(  1 + (s - s') | \log \nu| + \mathcal I_p \right) \,. 
\]

\end{acorollary}

\subsection{Verifying the moment conditions}

In this section we record and prove the estimates needed to verify the moment condition \eqref{eq:regularizingCond} to apply Theorem \ref{thm:allEqualLE} to advection diffusion and the 2d linearized Navier-Stokes equations described above. The techniques are straightforward, employing tools from Fourier multipliers and paradifferential calculus. It is likely that the stability estimates \eqref{eq:regPropPSA1}, \eqref{eq:regPropPSA2} and \eqref{eq:NSE-Stability-1s}, \eqref{eq:NSE-Stability-2s} are not sharp and could be improved with more work.

\medskip

\subsubsection{Preliminary estimates}

For each $s\in\R$ we define the fractional derivative operator $\Lambda^{s}$ 
to be the Fourier multiplier
\[
    \mathcal{F}[\Lambda^{s}f](k) := |k|^{s}\hat{f}(k).
\]

We begin by proving a fundamental commutator estimate for the advection operator $u\cdot \nabla$.

\begin{lemma}\label{lem:commutator}
Let $\gamma > \frac{d}{2} + 1$, $s\in [0,\gamma]$. Then, there exists a constant $C$ depending on $\gamma,d$ such that for all mean-zero vector fields $u\in \Hbf = \Hbf^{\gamma}$ and mean-zero scalars $f\in H^s$,  we have  
\[
	\|[\Lambda^s,u\cdot \nabla]f\|_{L^2}\leq C \|u\|_{H^{\gamma}}\|f\|_{H^s} \, .
\]
Here, $[A, B] = A B - B A$ denotes the commutator of two operators $A, B$. 
\end{lemma}
\begin{proof} By an approximation argument, it suffices to consider the case when $f, u$ are both $C^\infty$. Fixing such $f, u$, note first that the Fourier transform of $[\Lambda^s,u\cdot \nabla]f$ is given by
\begin{align*}
	\mathcal{F}[\Lambda^s,u\cdot \nabla]f(k) & = i\sum_{\ell\in\Z^d_0} \left( (|k|^s - |k-\ell|^s)\hat{f}(k-\ell) \right)  (k-\ell)\cdot\hat{u}(\ell) \, . 
\end{align*}
By Parseval's identity, it suffices to bound this in $\ell^2(\Z^d_0)$. For this, we split this sum up into two regions $|\ell|< |k|/2$ and $|\ell|\geq |k|/2$; we label the $\sum_{|\ell| < |k|/2}$ term $I(k)$ and the $\sum_{|\ell| \geq |k|/2}$ term $II(k)$. 

When $|\ell|<|k|/2$, it holds that $|k - \ell| \approx |k|$. It then follows from the mean value theorem that $||k|^s - |k-\ell|^s|\lesssim |k-\ell|^{s-1}|\ell|$, hence
\[
|I(k)| \leqc \sum_{\ell\in \Z^d_0} |\ell||\hat{u}(\ell)||k-\ell|^s|\hat{f}(k-\ell)| \, .
\]
By Young's inequality, it follows that the $\ell^2$ norm of $(I(k))_{k \in \Z^d_0}$ is  bounded by
\[
	\Big\||\ell| |\widehat{u}(\ell)|\Big\|_{\ell^1}\Big\| |k|^s |\widehat{f}(k)|\Big\|_{\ell^2} \leq \Big\||\ell|^{-r}\Big\|_{\ell^2}\Big\||\ell|^{r+1} |\widehat{u}(\ell)|\Big\|_{\ell^2}\Big\||k|^s|\widehat{f}(k)|\Big\|_{\ell^2} \leqc \|u\|_{H^{\gamma}}\|f\|_{H^{s}},
\]
where we used the fact that $|\ell|^{-r}$ belongs to $\ell^2(\Z^d_0)$ if $r>d/2$ and $\gamma \geq r+1$ so that $\|u\|_{H^{r+1}}\leq \|u\|_{H^{\gamma}}$.

When $|\ell|\geq |k|/2$, we instead have that $|k - \ell| \lesssim |\ell|$. Therefore, $||k|^s - |k-\ell|^s|\lesssim |\ell|^s$ and $|\ell|^{s-\gamma} \lesssim |k-\ell|^{s-\gamma}$ since $s \leq \gamma$. 
This gives, 
\[
|II(k)| \leqc \sum_{\ell\in \Z^d_0} |\ell|^{\gamma}|\widehat{u}(\ell)||k-\ell|^{1-\gamma+s}|\widehat{f}(k-\ell)|.
\]
Again by Young's inequality this implies that the $\ell^2$ norm of $(II(k))_{k \in \Z^d_0}$ is bounded by
\[
	\Big\||\ell|^{\gamma} |\widehat{u}(\ell)|\Big\|_{\ell^2}\Big\||k|^{1-\gamma+s}|\widehat{f}(k)|\Big\|_{\ell^1} \leq \|u\|_{H^{s+\gamma}}\Big\||k|^{1-\gamma}\Big\|_{\ell^2}\Big\||k|^s|\widehat{f}(k)|\Big\|_{\ell^2} \leqc \|u\|_{H^{\gamma}}\|f\|_{H^{s}},
\]
where we used the fact that $|k|^{1-\gamma}\in \ell^2(\Z^d_0)$ since $\gamma -1 > d/2$.
\end{proof}
When dealing with the compact term that arises in linearized Navier-Stokes equation and its adjoint, we will also require the following Lemma.

\begin{lemma}\label{lem:preEst2}
Let $d = 2, \gamma > 1 + d/2 = 2$, and $s\in [0,\gamma]$. Then, there exists a constant $C$ depending on $\gamma$ such that for all $u\in \Hbf$ and mean-zero $f\in H^s$, the following estimates hold: 
\begin{equation}\label{eq:est-LNs1}
	\|\Lambda^{s}(\Delta u\cdot \nabla) \Lambda^{-2}f\|_{L^2}\leq C \|u\|_{H^{\gamma+2}}\|f\|_{H^{s}},
\end{equation}
and
\begin{equation}\label{eq:est-LNs2}
	\|\Lambda^{s-2}(\Delta u\cdot \nabla) f\|_{L^2}\leq C \|u\|_{H^{\gamma+2}}\|f\|_{H^{s}}.
\end{equation}

\end{lemma}

\begin{proof}
We first prove the estimate on $\|\Lambda^{s}(\Delta u\cdot \nabla \Lambda^{-2} \xi)\|_{L^2}$. As in the previous Lemma, we consider the Fourier transform
\[
\Fc[ \Lambda^s ( \Delta u \cdot \nabla \Lambda^{-2} f)] (k) = i  \sum_{\ell \in \Z^d_0} |k|^s |\ell|^2 |k-\ell|^{-2} \hat f(k-\ell)\, (k-\ell) \cdot \hat u (\ell)
\]
and decompose the sum into $I(k) + II(k)$, corresponding $|\ell| < |k|/2$ and $|\ell| > |k|/2$.  As we had earlier, for the $I(k)$ terms we have that $|\ell|\leqc |k| \approx |k - \ell|$, hence
\[
|I(k)| \lesssim \sum_{\ell \in \Z^d_0}  |\ell| |k-\ell|^{s} |\hat f(k-\ell)|\, |\hat u (\ell)|,
\]
and so by Young's inequality, 
\[
\| (I(k))_{k \in \Z^d_0} \|_{\ell^2} \lesssim \| |\ell| u(\ell) \|_{\ell^1} \left\| |\ell|^2 |\hat f(\ell)| \right\|_{\ell^2}
\leqc \|u \|_{H^\gamma} \| f \|_{H^s}.
\]
For the $II(k)$ term, we have $|k|\leqc |\ell|$ and $|k - \ell|\leqc |\ell|$, and so 
\[
|II(k)| \lesssim \sum_{\ell \in \Z^d_0} |\ell|^{s + 2} |k-\ell|^{-1}|\hat f(k - \ell)| | \hat u(\ell)| \lesssim \sum_{\ell \in \Z^d_0} |\ell|^{\gamma+2}  |k-\ell|^{s-\gamma -1} |\hat f(k - \ell)||\hat u(\ell)| \, ,
\]
where we used that $|\ell|^{s-\gamma}\leqc |k-\ell|^{s-\gamma}$ since $\gamma \geq s$. Again using Young's inequality, we conclude
\[
\| (II(k))_{k \in \Z^d_0}\|_{\ell^2} \lesssim \| u\|_{H^{\gamma+2}} \| f \|_{H^{s-2}}\leq\| u\|_{H^{\gamma+2}} \| f \|_{H^{s}}  \,,
\]
thereby proving \eqref{eq:est-LNs1}.

The proof of \eqref{eq:est-LNs2} follows along similar lines and is omitted for brevity.
\end{proof}

\subsubsection{Advection diffusion}

Lets first consider the advection diffusion equation \eqref{eq:passiveScalarEq} on $\T^d$ associated to some arbitrary time dependent velocity field $u\in L^\infty([0,1];\mathbf{H})$.
Our first step is to prove the following quantitative $L^2\to H^s$ regularity estimate:
\begin{lemma}\label{lem:regularityPSA1}
Let $u : [0,1] \times \T^d \to \R^d$ be a time-varying, divergence-free vector field with $u \in L^\infty([0,1], \Hbf)$. Let $(f_t)_{t\in[0,1]}$ be the solution to \eqref{eq:passiveScalarEq} with $\kappa \in (0,1)$ and mean-zero $f_0 \in L^2$. Then, for all $s\in[0,\gamma]$, we have that
\begin{align}\label{eq:regImprovePSA1}
	 \|f_1\|_{H^s} \leq \kappa^{-s/2}\exp\left(c \int_0^1 (1+\|u_\tau \|_{H^{\gamma}})\,\dee \tau\right)\|f_0\|_{L^2} \,. 
\end{align}
If $f_0 \in H^s$, then 
\begin{align}\label{eq:regPropPSA1}
\sup_{t \in [0,1]} \| f_t\|_{H^s} \leq \exp \left( c \int_0^1 (1 + \| u_\tau\|_{H^{\gamma}} ) \dee \tau \right) \| f \|_{H^s} \,. 
\end{align}
\end{lemma}
\begin{proof}
We prove below the regularization bound \eqref{eq:regImprovePSA1}; the propagation 
bound \eqref{eq:regPropPSA1} follows similarly and its proof is omitted. 

By an approximation argument, it suffices to prove the above estimate when $f_0$ is $C^\infty$ and $u_t$ is a $C^\infty$ vector field for all $t \in [0,1]$. 
For each $t \in [0,1]$, we consider the time-dependent operator $\Lambda^{rt}$ and note that $\kappa^{st/2}\Lambda^{st}f_t$ satisfies
\[
	\partial_t (\kappa^{st/2}\Lambda^{st} f_t) = -\kappa^{st/2}\Lambda^{st}(u_t\cdot\nabla f_t) + \kappa^{st/2}(s\log(\sqrt{\kappa}\Lambda^{1}) + \kappa \Delta) \Lambda^{st} f_t,
\]
where the operator $\log(\sqrt{\kappa}\Lambda^1)$ is defined by the Fourier multiplier
\[
	\mathcal{F}[\log(\sqrt{\kappa}\Lambda^1)f](k) = \log(\sqrt{\kappa}|k|)\widehat{f}(k) \,,
\]
which is defined on the space of mean-zero $f$. Note that there exists a $C(s) > 0$, independent of $k$ or $\kappa$, such that $s\log(\sqrt{\kappa}|k|) \leq \kappa |k|^2 + C(s)$, so that for $t\in [0,1]$ we have by Parseval's identity
\[
	\langle \Lambda^{rt} f_t,(s\log(\sqrt{\kappa}\Lambda^{1}) + \kappa \Delta) \Lambda^{st} f_t\rangle_{L^2} = \sum_{k\in\Z^d_0}(s\log(\sqrt{\kappa}|k|) - \kappa |k|^2)|k|^{2st}|\widehat{f_t}(k)|^2\lesssim \|f_t\|_{H^{st}}^2.
\]
Using this and that $f_t$ is smooth and $u_t$ is divergence-free gives the following energy estimate
\begin{equation}
\begin{aligned}
	\frac{\dee}{\dt}\left(\tfrac{1}{2}\kappa^{st}\|f_t\|_{H^{st}}^2\right) &= -\kappa^{st}\langle \Lambda^{st} f_t, \Lambda^{st}(u_t\cdot\nabla f_t)\rangle_{L^2} + \kappa^{st}\langle \Lambda^{st} f_t,(r\log(\sqrt{\kappa}\Lambda^{1}) + \kappa \Delta) \Lambda^{st} f_t\rangle_{L^2}\\
	&\leqc -\kappa^{st}\langle \Lambda^{st} f_t, [\Lambda^{st},u_t\cdot\nabla] f_t\rangle_{L^2} + \kappa^{st}\|f_t\|_{H^{st}}^2\\
	&\leqc \kappa^{st}\left(\|f_t\|_{H^{st}}\|[\Lambda^{st},u_t\cdot\nabla] f_t\|_{L^2} + \|f_t\|_{H^{st}}^2\right).
\end{aligned}
\end{equation}
Above, we have used the fact that since $u$ is divergence-free, $u \cdot \nabla$ is skew-symmetric in the $L^2$ inner product, hence $\langle g, u \cdot \nabla g \rangle = 0$ for all smooth $g$. 
Applying the commutator Lemma \ref{lem:commutator} with $s = r t$, assuming $t\in [0,1]$ and using that $\gamma > \frac{d}{2} + 1$, we conclude that
\[
\begin{aligned}
	\frac{\dee}{\dt}\left(\kappa^{st}\|f_t\|_{H^{st}}^2\right) &\leqc (1+\|u_t\|_{H^{\gamma}})\left(\kappa^{st}\|f_t\|_{H^{st}}^2\right).
	\end{aligned}
\]
In particular, 
\[
\frac{\dee}{\dee t} \log \left( \kappa^{s t} \| f_t\|_{H^{s t}}^2 \right) \lesssim 1 + \| u_t\|_{H^\gamma} \,,
\]
and integrating $t$ from $0$ to $1$ completes the proof.
\end{proof}

\begin{lemma}\label{lem:regularityPSA2}
Let $u : [0,1] \times \T^d \to \R^d$ be a time-varying, divergence-free vector field with $u \in L^\infty([0,1], \Hbf)$. Let $f\in C([0,1];H^{-s})$ be a solution to \eqref{eq:passiveScalarEq} with $\kappa \in (0,1]$ and initial $f_0 \in H^{-s}$ for some $s \in [0,\gamma]$. Then, 
$f_1 \in L^2$, and
\begin{align}\label{eq:regImprovePSA2}
	\|f_1\|_{L^2} \leq \kappa^{-s/2}\exp\left(c \int_0^1 (1+\|u_\tau \|_{H^{\gamma}})\,\dee \tau\right)\|f_0\|_{H^{-s}}
\end{align}
and
\begin{align}\label{eq:regPropPSA2}
\sup_{t \in [0,1]} \| f_t\|_{H^{-s} }\leq \exp \left( c \int_0^1 \| u_\tau\|_{H^{\gamma}} \dee \tau \right) \| f \|_{H^{-s}} \,. 
\end{align}
\end{lemma}

\begin{proof}
As before, we focus below on the proof of \eqref{eq:regImprovePSA2} and omit that of  \eqref{eq:regPropPSA2}. 
By an approximation argument, we can assume $f_0, u_t$ are $C^\infty$. Our proof will use the $L^2$ duality of $H^{-s}$ with $H^s$. To see this, let $g_0$ be a smooth, mean-zero function and let $(g_t)$ solve the time reversed equation
\begin{equation}\label{eq:dualPassiveAdvection}
\partial_t g_t - u_{1 - t} \cdot \nabla g_t = \kappa \Delta g_t \,. 
\end{equation}
We compute
\begin{align*}
\frac{\dee}{\dee t} \langle g_t, f_{1 - t} \rangle 
& = \langle \partial_t g_t , f_{1 - t} \rangle = - \langle g_t, \partial_t f_{1 - t} \rangle  = \langle  u_{1 - t} \cdot \nabla g_t , f_{1 - t} \rangle = - \langle g_t, - u_{1 - t} \cdot \nabla f_{1 - t} \rangle = 0
\end{align*}
using (i) that $\Delta$ is self adjoint in the $L^2$ inner product and (ii) $u \cdot \nabla$ is skew-adjoint when $u$ is divergence free. We conclude that
\[
\langle f_1, g_0 \rangle = \langle f_0, g_1 \rangle \,. 
\]
Now, 
\begin{align*}
\| f_1 \|_{L^2} &= \sup_{\| g_0\|_{L^2} = 1} \langle f_1, g_0 \rangle  = \sup_{\| g_0\|_{L^2} = 1} \langle f_0, g_1 \rangle  \leq \| f_0\|_{H^{-s}} \sup_{\| g_0\|_{L^2} = 1} \| g_1 \|_{H^s} \,, 
\end{align*}
treating the $g_0$ under the $\sup$ as an initial condition for \eqref{eq:dualPassiveAdvection}. By Lemma \ref{lem:regularityPSA1}, it holds that
\[
\| g_1 \|_{H^s} \lesssim \kappa^{- s / 2} \exp \left( c \int_0^1 (1 + \| u_\tau\|_{H^\gamma}) \dee \tau \right) 
\]
and so 
\[
\| f_1\|_{L^2} \leq \kappa^{- s / 2} \exp \left( c \int_0^1 (1 + \| u_\tau\|_{H^\gamma}) \dee \tau \right) \| f_0 \|_{H^{-s}}
\]
as desired. 
\end{proof}

\subsubsection{Linearized Navier-Stokes}\label{subsubsec:linns-est}

It is convenient to work with Navier-Stokes in vorticity form
\[
\partial_t w + u \cdot \nabla w = \nu \Delta w + \mathrm{curl}\, F
\]
where $w = \mathrm{curl}\, u$, and the velocity $u$ is recovered by the Biot-Savart law $u = \Lambda^{-2} ( \nabla^\perp w) =: Kw$, where 
here $\nabla^\perp= (- \partial_y, \partial_x)$ denotes the skew gradient. In this form, the first variation equation becomes
\begin{equation}\label{eq:LSN-vort}
\partial_t \eta + u \cdot \nabla \eta + v \cdot \nabla w = \nu \Delta \eta
\end{equation}
where $v_t = K \eta_t$. 
\begin{lemma}\label{lem:LNS1}
Let $u \in L^\infty([0,1], \Hbf) \cap 
L^1([0,1], \Hbf^{\gamma + 2})$, and let $\eta \in L^\infty([0,1], L^2)$ be the solution to \eqref{eq:LSN-vort} with initial $\eta_0\in L^2$. Then for each $s\in[0,\gamma]$, $\eta_1 \in H^s$, and satisfies
\[
\| \eta_1 \|_{H^s} \lesssim \nu^{- s / 2} \exp \left( c \int_0^1 (1 + \| u_\tau \|_{H^{\gamma+2}}) \dee \tau \right) \| \eta_0\|_{L^2} \,. 
\]
If $\eta_0 \in H^s$, then 
\begin{equation}\label{eq:NSE-Stability-1s}
\sup_{t \in [0,1]} \| \eta_t \|_{H^s} \lesssim \exp \left( c \int_0^1 (1 + \| u_\tau \|_{H^{\gamma+2}}) \dee \tau \right) \| \eta_0\|_{H^s} \,. 
\end{equation}
\end{lemma}

\begin{proof}
Note that since $w = \Delta\psi$ and $\nabla^\perp \psi = u$, we can also rewrite the second term in \ref{eq:LSN-vort} in the following more useful form
\begin{equation}\label{eq:LSN-vort-alt}
\partial_t \eta + u \cdot \nabla \eta + \Delta u \cdot \nabla \Lambda^{-2} \eta = \nu \Delta \eta.
\end{equation}
Repeating previous computations, we have
\begin{align*}
\frac{\dee}{\dee t} \frac12 \nu^{s t} \| \eta_t\|_{H^{s t}}^2 & =  
\nu^{s t} \langle \Lambda^{s t} \eta_t, \big(s \log ( \sqrt{\nu} \Lambda ) + \nu \Delta\big) \Lambda^{s t} \eta_t \rangle \\
& - \nu^{s t} \langle \Lambda^{s t} \eta_t, [\Lambda^{s t} , u_t \cdot \nabla ] \eta_t \rangle \\
& - \nu^{s t} \langle \Lambda^{s t} \eta_t, \Lambda^{s t} (\Delta u_t \cdot \nabla) \Lambda^{-2} \eta_t \rangle 
\end{align*}
The first and second terms are bounded $\lesssim \nu^{s t} (1 + \| u_t \|_{H^\gamma}) \| \eta_t\|_{H^{s t}}^2$ as before, while 
by Lemma \ref{lem:preEst2} we have that the third term is
\[
\lesssim \nu^{s t}\|\eta\|_{H^{s t}}^2 \| u_t\|_{H^{\gamma+2}} \,. 
\]
Applying these inequalities and combining all terms, we have shown that 
\[
\frac{\dee}{\dee t} \nu^{s t} \| \eta_t\|_{H^{s t}}^2 \lesssim 
\nu^{s t} \| \eta_t\|_{H^{s t}}^2 \| u_t\|_{H^{\gamma+2}} \,. 
\]
The desired conclusion follows as before. The estimate when $\eta_0 \in H^s$ is similar and omitted. 
\end{proof}

We also have the analogue of Lemma \ref{lem:regularityPSA2}. \begin{lemma}\label{lem:LNS2}
Let $u \in L^\infty([0,1], H^\gamma)$ for $\gamma > 1+ d/2$ and let $\eta \in L^\infty([0,1], L^2)$ be the solution to \eqref{eq:LSN-vort} with initial $\eta_0\in H^{-s}$ for some $s\in[0,\gamma]$. Then, $\eta_1 \in L^2$, and satisfies the estimate
\[
\| \eta_1 \|_{L^2} \leq \nu^{- s / 2} \exp \left( c \int_0^1 (1 + \| u_\tau \|_{H^{\gamma+2}}) \dee \tau \right) \| \eta_t\|_{H^{-s}} \,. 
\]
If $\eta_0 \in H^{-s}$, then 
\begin{equation}\label{eq:NSE-Stability-2s}
\sup_{t \in [0,1]} \| \eta_t \|_{H^{-s}} \leq \exp \left( c \int_0^1 (1 + \| u_\tau \|_{H^{\gamma+2}}) \dee \tau \right) \| \eta_t\|_{H^{-s}} \,. 
\end{equation}
\end{lemma}
\begin{proof}

As in Lemma \ref{lem:regularityPSA2} we will also find it useful to consider the associated time-reversed adjoint equation
\begin{equation}\label{eq:LSN-vort-alt-adj}
\partial_t \zeta - u_{1-t} \cdot \nabla \zeta - \Lambda^{-2}(\Delta u_{1-t} \cdot \nabla) \zeta = \nu \Delta \zeta.
\end{equation}
with smooth initial data satisfying $\|\zeta_0\|_{L^2} = 1$. Since 
\[
	\langle \eta_1,\zeta_0\rangle = \langle \eta_0,\zeta_1\rangle,
\]
it suffices to prove the following estimate: 
\[
	\|\zeta_{1}\|_{H^s}\leq \nu^{-s/2}\exp \left( c \int_0^1 (1 + \| u_\tau \|_{H^{\gamma+2}}) \dee \tau \right) \,. 
\]
We will prove this in a manner similar to Lemma \ref{lem:LNS1}. Specifically, by repeating these computations, we obtain that
\begin{align*}
\frac{\dee}{\dee t} \nu^{s t} \| \zeta_t\|_{H^{s t}}^2 & =  
\nu^{s t} \langle \Lambda^{s t} \zeta_t, \big(s \log ( \sqrt{\nu} \Lambda ) + \nu \Delta\big) \Lambda^{s t} \zeta_t \rangle \\
& + \nu^{s t} \langle \Lambda^{s t} \zeta_t, [\Lambda^{s t} , u_{1-t} \cdot \nabla ] \zeta_t \rangle \\
& + \nu^{s t} \langle \Lambda^{s t} \zeta_t, \Lambda^{st-2}(\Delta u_{1-t} \cdot \nabla)  \zeta_t \rangle.
\end{align*}
Again the first two terms are $\leqc \nu^{st} (1 + \|u_{1-t}\|_{H^\gamma})\|\zeta_t\|_{H^{st}}^2$, while by Lemma \ref{lem:preEst2} the third term is
\[
	\leqc \nu^{st}\|\eta_{t}\|_{H^{st}}^2\|u_{1-t}\|_{H^{\gamma + 2}}.
\]
This gives
\[
	\frac{\dee}{\dt} \nu^{st}\|\zeta_{t}\|_{H^{st}}^2\leqc \nu^{st}\|\zeta_t\|^2_{H^{st}}\|u_{1-t}\|_{H^{\gamma+2}}.
\]
The estimate follows by integrating the differential inequality from $0$ to $1$ exactly as in previous computations. The estimate on propagation of $H^{-s}$ regularity is similar and omitted. 
\end{proof}

\bibliographystyle{abbrv}
\bibliography{bibliography}

\begin{thebibliography}{10}

\bibitem{barreira2002lyapunov}
L.~Barreira and Y.~B. Pesin.
\newblock {\em Lyapunov exponents and smooth ergodic theory}, volume~23.
\newblock American Mathematical Soc., 2002.

\bibitem{beck2013metastability}
M.~Beck and C.~E. Wayne.
\newblock Metastability and rapid convergence to quasi-stationary bar states
  for the two-dimensional navier--stokes equations.
\newblock {\em Proceedings of the Royal Society of Edinburgh Section A:
  Mathematics}, 143(5):905--927, 2013.

\bibitem{BBPS19II}
J.~Bedrossian, A.~Blumenthal, and S.~Punshon-Smith.
\newblock Almost-sure enhanced dissipation and uniform-in-diffusivity
  exponential mixing for advection-diffusion by stochastic {N}avier-{S}tokes.
\newblock {\em arXiv preprint arXiv:1911.01561}, 2019.

\bibitem{bedrossian2019almost}
J.~Bedrossian, A.~Blumenthal, and S.~Punshon-Smith.
\newblock Almost-sure exponential mixing of passive scalars by the stochastic
  navier-stokes equations.
\newblock {\em arXiv preprint arXiv:1905.03869}, 2019.

\bibitem{bedrossian2018lagrangian}
J.~Bedrossian, A.~Blumenthal, and S.~Punshon-Smith.
\newblock Lagrangian chaos and scalar advection in stochastic fluid mechanics.
\newblock {\em J. Eur. Math. Soc.}, Jan. 2022.

\bibitem{BCZ2017enhanced}
J.~Bedrossian and M.~Coti~Zelati.
\newblock Enhanced dissipation, hypoellipticity, and anomalous small noise
  inviscid limits in shear flows.
\newblock {\em Archive for Rational Mechanics and Analysis}, 224(3):1161--1204,
  2017.

\bibitem{bedrossian2016invariant}
J.~Bedrossian, M.~Coti~Zelati, and N.~Glatt-Holtz.
\newblock Invariant measures for passive scalars in the small noise inviscid
  limit.
\newblock {\em Communications in Mathematical Physics}, 348(1):101--127, 2016.

\bibitem{bernoff1994rapid}
A.~J. Bernoff and J.~F. Lingevitch.
\newblock Rapid relaxation of an axisymmetric vortex.
\newblock {\em Physics of Fluids}, 6(11):3717--3723, 1994.

\bibitem{blumenthal2015volume}
A.~Blumenthal.
\newblock A volume-based approach to the multiplicative ergodic theorem on
  banach spaces.
\newblock {\em Discrete \& Continuous Dynamical Systems}, 36(5):2377, 2016.

\bibitem{blumenthal2017lyapunov}
A.~Blumenthal, J.~Xue, and L.-S. Young.
\newblock Lyapunov exponents for random perturbations of some area-preserving
  maps including the standard map.
\newblock {\em Annals of Mathematics}, 185(1):285--310, 2017.

\bibitem{blumenthal2017entropy}
A.~Blumenthal and L.-S. Young.
\newblock Entropy, volume growth and {SRB} measures for {B}anach space
  mappings.
\newblock {\em Inventiones mathematicae}, 207(2):833--893, 2017.

\bibitem{bowen2021multiplicative}
L.~Bowen, B.~Hayes, and Y.~F. Lin.
\newblock A multiplicative ergodic theorem for von neumann algebra valued
  cocycles.
\newblock {\em Communications in Mathematical Physics}, 384(2):1291--1350,
  2021.

\bibitem{busemann1947intrinsic}
H.~Busemann.
\newblock Intrinsic area.
\newblock {\em Annals of Mathematics}, pages 234--267, 1947.

\bibitem{chandrasekhar2013hydrodynamic}
S.~Chandrasekhar.
\newblock {\em Hydrodynamic and hydromagnetic stability}.
\newblock Courier Corporation, 2013.

\bibitem{Chicone1999-mi}
C.~Chicone and Y.~Latushkin.
\newblock {\em Evolution Semigroups in Dynamical Systems and Differential
  Equations}.
\newblock American Mathematical Soc., 1999.

\bibitem{chirikov1979universal}
B.~V. Chirikov.
\newblock A universal instability of many-dimensional oscillator systems.
\newblock {\em Physics reports}, 52(5):263--379, 1979.

\bibitem{constantin1983global}
P.~Constantin and C.~Foias.
\newblock Global lyapunov exponents, kaplan-yorke formulas and the dimension of
  the attractors for 2d navier-stokes equations.
\newblock 1983.

\bibitem{constantin2008diffusion}
P.~Constantin, A.~Kiselev, L.~Ryzhik, and A.~Zlato{\v{s}}.
\newblock Diffusion and mixing in fluid flow.
\newblock {\em Annals of Mathematics}, pages 643--674, 2008.

\bibitem{CrisantiEtAl1991}
A.~Crisanti, M.~Falcioni, A.~Vulpiani, and G.~Paladin.
\newblock Lagrangian chaos: transport, mixing and diffusion in fluids.
\newblock {\em La Rivista del Nuovo Cimento (1978-1999)}, 14(12):1--80, 1991.

\bibitem{crisanti1993intermittency}
A.~Crisanti, M.~Jensen, A.~Vulpiani, and G.~Paladin.
\newblock Intermittency and predictability in turbulence.
\newblock {\em Physical review letters}, 70(2):166, 1993.

\bibitem{crovisier2019problem}
S.~Crovisier and S.~Senti.
\newblock A problem for the 21st/22nd century.
\newblock {\em EMS Newsletter}, (114):8--13, 2019.

\bibitem{doering2006multiscale}
C.~R. Doering and J.-L. Thiffeault.
\newblock Multiscale mixing efficiencies for steady sources.
\newblock {\em Physical Review E}, 74(2):025301, 2006.

\bibitem{dragivcevic2018spectral}
D.~Dragi{\v{c}}evi{\'c}, G.~Froyland, C.~Gonzalez-Tokman, and S.~Vaienti.
\newblock A spectral approach for quenched limit theorems for random expanding
  dynamical systems.
\newblock {\em Communications in Mathematical Physics}, pages 1--67, 2018.

\bibitem{dragivcevic2020spectral}
D.~Dragi{\v{c}}evi{\'c}, G.~Froyland, C.~Gonz{\'a}lez-Tokman, and S.~Vaienti.
\newblock A spectral approach for quenched limit theorems for random hyperbolic
  dynamical systems.
\newblock {\em Transactions of the American Mathematical Society},
  373(1):629--664, 2020.

\bibitem{drazin2004hydrodynamic}
P.~G. Drazin and W.~H. Reid.
\newblock {\em Hydrodynamic stability}.
\newblock Cambridge university press, 2004.

\bibitem{dubrulle1994scaling}
B.~Dubrulle and S.~Nazarenko.
\newblock On scaling laws for the transition to turbulence in uniform-shear
  flows.
\newblock {\em EPL (Europhysics Letters)}, 27(2):129, 1994.

\bibitem{Dyatlov2015-uj}
S.~Dyatlov and M.~Zworski.
\newblock Stochastic stability of {Pollicott--Ruelle} resonances.
\newblock {\em Nonlinearity}, 28(10):3511, Sept. 2015.

\bibitem{eckmann1985ergodic}
J.-P. Eckmann and D.~Ruelle.
\newblock Ergodic theory of chaos and strange attractors.
\newblock {\em The theory of chaotic attractors}, pages 273--312, 1985.

\bibitem{Feng_2019}
Y.~Feng and G.~Iyer.
\newblock Dissipation enhancement by mixing.
\newblock {\em Nonlinearity}, 32(5):1810--1851, apr 2019.

\bibitem{Flandoli1995-mo}
F.~Flandoli and B.~Maslowski.
\newblock Ergodicity of the {$2$-D} {Navier-Stokes} equation under random
  perturbations.
\newblock {\em Commun. Math. Phys.}, 172(1):119--141, 1995.

\bibitem{foias2001navier}
C.~Foias, O.~Manley, R.~Rosa, and R.~Temam.
\newblock {\em Navier-Stokes equations and turbulence}, volume~83.
\newblock Cambridge University Press, 2001.

\bibitem{froyland2010semi}
G.~Froyland, S.~Lloyd, and A.~Quas.
\newblock A semi-invertible oseledets theorem with applications to transfer
  operator cocycles.
\newblock {\em arXiv preprint arXiv:1001.5313}, 2010.

\bibitem{froyland2010coherent}
G.~Froyland, S.~Lloyd, and N.~Santitissadeekorn.
\newblock Coherent sets for nonautonomous dynamical systems.
\newblock {\em Physica D: Nonlinear Phenomena}, 239(16):1527--1541, 2010.

\bibitem{froyland2013metastability}
G.~Froyland and O.~Stancevic.
\newblock Metastability, lyapunov exponents, escape rates, and topological
  entropy in random dynamical systems.
\newblock {\em Stochastics and Dynamics}, 13(04):1350004, 2013.

\bibitem{furstenberg1963noncommuting}
H.~Furstenberg.
\newblock Noncommuting random products.
\newblock {\em Transactions of the American Mathematical Society},
  108(3):377--428, 1963.

\bibitem{furstenberg1960products}
H.~Furstenberg and H.~Kesten.
\newblock Products of random matrices.
\newblock {\em The Annals of Mathematical Statistics}, 31(2):457--469, 1960.

\bibitem{gonzalez2014concise}
C.~Gonz{\'a}lez-Tokman and A.~Quas.
\newblock A concise proof of the multiplicative ergodic theorem on banach
  spaces.
\newblock {\em arXiv preprint arXiv:1406.1955}, 2014.

\bibitem{gonzalez2014semi}
C.~Gonz{\'a}lez-Tokman and A.~Quas.
\newblock A semi-invertible operator oseledets theorem.
\newblock {\em Ergodic Theory and Dynamical Systems}, 34(4):1230--1272, 2014.

\bibitem{gonzalez2015concise}
C.~Gonz{\'a}lez-Tokman and A.~Quas.
\newblock A concise proof of the multiplicative ergodic theorem on banach
  spaces.
\newblock {\em Journal of Modern Dynamics}, 9(1):237--255, 2015.

\bibitem{gonzalez2021stability}
C.~Gonz{\'a}lez-Tokman and A.~Quas.
\newblock Stability and collapse of the lyapunov spectrum for perron--frobenius
  operator cocycles.
\newblock {\em Journal of the European Mathematical Society},
  23(10):3419--3457, 2021.

\bibitem{HM06}
M.~Hairer and J.~C. Mattingly.
\newblock Ergodicity of the 2{D} {N}avier-{S}tokes equations with degenerate
  stochastic forcing.
\newblock {\em Ann. of Math.}, 164(3):993--1032, 2006.

\bibitem{I_Udovich1989-jm}
V.~I. I\_Udovich.
\newblock {\em The Linearization Method in Hydrodynamical Stability Theory}.
\newblock American Mathematical Soc., Dec. 1989.

\bibitem{kaimanovich1989lyapunov}
V.~A. Kaimanovich.
\newblock Lyapunov exponents, symmetric spaces, and a multiplicative ergodic
  theorem for semisimple lie groups.
\newblock {\em Journal of Soviet Mathematics}, 47(2):2387--2398, 1989.

\bibitem{karlsson1999multiplicative}
A.~Karlsson and G.~A. Margulis.
\newblock A multiplicative ergodic theorem and nonpositively curved spaces.
\newblock {\em Communications in mathematical physics}, 208(1):107--123, 1999.

\bibitem{kato2013perturbation}
T.~Kato.
\newblock {\em Perturbation theory for linear operators}, volume 132.
\newblock Springer Science \& Business Media, 2013.

\bibitem{kifer2012ergodic}
Y.~Kifer.
\newblock {\em Ergodic theory of random transformations}, volume~10.
\newblock Springer Science \& Business Media, 2012.

\bibitem{kryloff1937theorie}
N.~Kryloff and N.~Bogoliouboff.
\newblock La th{\'e}orie g{\'e}n{\'e}rale de la mesure dans son application
  {\`a} l'{\'e}tude des syst{\`e}mes dynamiques de la m{\'e}canique non
  lin{\'e}aire.
\newblock {\em Annals of mathematics}, pages 65--113, 1937.

\bibitem{Kuksin2003-st}
S.~Kuksin and A.~Shirikyan.
\newblock Some limiting properties of randomly forced two-dimensional
  {Navier--Stokes} equations.
\newblock {\em Proceedings of the Royal Society of Edinburgh Section A:
  Mathematics}, 133(4):875--891, Aug. 2003.

\bibitem{KS}
S.~Kuksin and A.~Shirikyan.
\newblock {\em Mathematics of two-dimensional turbulence}, volume 194.
\newblock Cambridge University Press, 2012.

\bibitem{latini2001transient}
M.~Latini and A.~J. Bernoff.
\newblock Transient anomalous diffusion in poiseuille flow.
\newblock {\em Journal of Fluid Mechanics}, 441:399--411, 2001.

\bibitem{lian2010lyapunov}
Z.~Lian and K.~Lu.
\newblock {\em Lyapunov exponents and invariant manifolds for random dynamical
  systems in a Banach space}.
\newblock American Mathematical Soc., 2010.

\bibitem{lin2011optimal}
Z.~Lin, J.-L. Thiffeault, and C.~R. Doering.
\newblock Optimal stirring strategies for passive scalar mixing.
\newblock {\em Journal of Fluid Mechanics}, 675:465--476, 2011.

\bibitem{lu2013strange}
K.~Lu, Q.~Wang, and L.-S. Young.
\newblock {\em Strange attractors for periodically forced parabolic equations},
  volume 224.
\newblock American Mathematical Soc., 2013.

\bibitem{lundgren1982strained}
T.~Lundgren.
\newblock Strained spiral vortex model for turbulent fine structure.
\newblock {\em The Physics of Fluids}, 25(12):2193--2203, 1982.

\bibitem{mackay1991appraisal}
R.~MacKay.
\newblock An appraisal of the ruelle-takens route to turbulence.
\newblock In {\em The Global Geometry of Turbulence}, pages 233--246. Springer,
  1991.

\bibitem{mane1983lyapounov}
R.~Man{\'e}.
\newblock Lyapounov exponents and stable manifolds for compact transformations.
\newblock In {\em Geometric dynamics}, pages 522--577. Springer, 1983.

\bibitem{mathew2005multiscale}
G.~Mathew, I.~Mezi{\'c}, and L.~Petzold.
\newblock A multiscale measure for mixing.
\newblock {\em Physica D: Nonlinear Phenomena}, 211(1-2):23--46, 2005.

\bibitem{mierczynski2020lyapunov}
J.~Mierczy{\'n}ski, S.~Novo, and R.~Obaya.
\newblock Lyapunov exponents and oseledets decomposition in random dynamical
  systems generated by systems of delay differential equations.
\newblock {\em Communications on Pure \& Applied Analysis}, 19(4):2235, 2020.

\bibitem{MC18}
C.~J. Miles and C.~R. Doering.
\newblock Diffusion-limited mixing by incompressible flows.
\newblock {\em Nonlinearity}, 31(5):2346, 2018.

\bibitem{noethen2022well}
F.~Noethen.
\newblock Well-separating common complements for sequences of subspaces of the
  same codimension are generic in hilbert spaces.
\newblock {\em Analysis Mathematica}, pages 1--21, 2022.

\bibitem{oakley2021mix}
B.~W. Oakley, J.-L. Thiffeault, and C.~R. Doering.
\newblock On mix-norms and the rate of decay of correlations.
\newblock {\em Nonlinearity}, 34(6):3762, 2021.

\bibitem{oseledets1968multiplicative}
V.~I. Oseledets.
\newblock A multiplicative ergodic theorem. {C}haracteristic {L}japunov
  exponents of dynamical systems.
\newblock {\em Trudy Moskovskogo Matematicheskogo Obshchestva}, 19:179--210,
  1968.

\bibitem{pesin2010open}
Y.~Pesin and V.~Climenhaga.
\newblock Open problems in the theory of non-uniform hyperbolicity.
\newblock {\em Discrete Contin. Dyn. Syst}, 27(2):589--607, 2010.

\bibitem{pesin1977characteristic}
Y.~B. Pesin.
\newblock Characteristic lyapunov exponents and smooth ergodic theory.
\newblock {\em Uspekhi Matematicheskikh Nauk}, 32(4):55--112, 1977.

\bibitem{raghunathan1979proof}
M.~S. Raghunathan.
\newblock A proof of oseledec’s multiplicative ergodic theorem.
\newblock {\em Israel Journal of Mathematics}, 32(4):356--362, 1979.

\bibitem{rhines1983rapidly}
P.~B. Rhines and W.~R. Young.
\newblock How rapidly is a passive scalar mixed within closed streamlines?
\newblock {\em Journal of Fluid Mechanics}, 133:133--145, 1983.

\bibitem{Robinson2001-tk}
J.~C. Robinson.
\newblock {\em {Infinite-Dimensional} Dynamical Systems: An Introduction to
  Dissipative Parabolic {PDEs} and the Theory of Global Attractors}.
\newblock Cambridge University Press, Apr. 2001.

\bibitem{ruelle1979ergodic}
D.~Ruelle.
\newblock Ergodic theory of differentiable dynamical systems.
\newblock {\em Publications Math{\'e}matiques de l'Institut des Hautes
  {\'E}tudes Scientifiques}, 50(1):27--58, 1979.

\bibitem{ruelle1982characteristic}
D.~Ruelle.
\newblock Characteristic exponents and invariant manifolds in hilbert space.
\newblock {\em Annals of Mathematics}, pages 243--290, 1982.

\bibitem{ruelle1971nature}
D.~Ruelle and F.~Takens.
\newblock On the nature of turbulence.
\newblock {\em Les rencontres physiciens-math{\'e}maticiens de
  Strasbourg-RCP25}, 12:1--44, 1971.

\bibitem{schaumloffel1991multiplicative}
K.-U. Schauml{\"o}ffel and F.~Flandoli.
\newblock A multiplicative ergodic theorem with applications to a first order
  stochastic hyperbolic equation in a bounded domain.
\newblock {\em Stochastics: An International Journal of Probability and
  Stochastic Processes}, 34(3-4):241--255, 1991.

\bibitem{schmid2002stability}
P.~J. Schmid, D.~S. Henningson, and D.~Jankowski.
\newblock Stability and transition in shear flows. applied mathematical
  sciences, vol. 142.
\newblock {\em Appl. Mech. Rev.}, 55(3):B57--B59, 2002.

\bibitem{shaw2007stirring}
T.~A. Shaw, J.-L. Thiffeault, and C.~R. Doering.
\newblock Stirring up trouble: multi-scale mixing measures for steady scalar
  sources.
\newblock {\em Physica D: Nonlinear Phenomena}, 231(2):143--164, 2007.

\bibitem{Temam1997-xm}
R.~Temam.
\newblock {\em {Infinite-Dimensional} Dynamical Systems in Mechanics and
  Physics}.
\newblock Springer New York, Apr. 1997.

\bibitem{temam2012infinite}
R.~Temam.
\newblock {\em Infinite-dimensional dynamical systems in mechanics and
  physics}, volume~68.
\newblock Springer Science \& Business Media, 2012.

\bibitem{thieullen1987fibres}
P.~Thieullen.
\newblock Fibr{\'e}s dynamiques asymptotiquement compacts exposants de
  lyapounov. entropie. dimension.
\newblock In {\em Annales de l'Institut Henri Poincar{\'e} C, Analyse non
  lin{\'e}aire}, volume~4, pages 49--97. Elsevier, 1987.

\bibitem{van2012origin}
E.~Van~Sebille, M.~H. England, and G.~Froyland.
\newblock Origin, dynamics and evolution of ocean garbage patches from observed
  surface drifters.
\newblock {\em Environmental Research Letters}, 7(4):044040, 2012.

\bibitem{varzaneh2021oseledets}
M.~G. Varzaneh and S.~Riedel.
\newblock Oseledets splitting and invariant manifolds on fields of banach
  spaces.
\newblock {\em Journal of Dynamics and Differential Equations}, pages 1--31,
  2021.

\bibitem{viana2016foundations}
M.~Viana and K.~Oliveira.
\newblock {\em Foundations of ergodic theory}.
\newblock Number 151. Cambridge University Press, 2016.

\bibitem{vukadinovic2015averaging}
J.~Vukadinovic, E.~Dedits, A.~C. Poje, and T.~Sch{\"a}fer.
\newblock Averaging and spectral properties for the 2d advection--diffusion
  equation in the semi-classical limit for vanishing diffusivity.
\newblock {\em Physica D: Nonlinear Phenomena}, 310:1--18, 2015.

\bibitem{walters1993dynamical}
P.~Walters.
\newblock A dynamical proof of the multiplicative ergodic theorem.
\newblock {\em Transactions of the American Mathematical Society},
  335(1):245--257, 1993.

\bibitem{walters2000introduction}
P.~Walters.
\newblock {\em An introduction to ergodic theory}, volume~79.
\newblock Springer Science \& Business Media, 2000.

\bibitem{wilkinson2017lyapunov}
A.~Wilkinson.
\newblock What are lyapunov exponents, and why are they interesting?
\newblock {\em Bulletin of the American Mathematical Society}, 54(1):79--105,
  2017.

\bibitem{wojtaszczyk1996banach}
P.~Wojtaszczyk.
\newblock {\em Banach spaces for analysts}.
\newblock Number~25. Cambridge University Press, 1996.

\bibitem{yaglom2012hydrodynamic}
A.~M. Yaglom.
\newblock {\em Hydrodynamic instability and transition to turbulence}, volume
  100.
\newblock Springer Science \& Business Media, 2012.

\bibitem{yamada1988lyapunov}
M.~Yamada and K.~Ohkitani.
\newblock Lyapunov spectrum of a model of two-dimensional turbulence.
\newblock {\em Physical review letters}, 60(11):983, 1988.

\bibitem{young2002srb}
L.-S. Young.
\newblock What are srb measures, and which dynamical systems have them?
\newblock {\em Journal of Statistical Physics}, 108(5):733--754, 2002.

\bibitem{zelati2020relation}
M.~C. Zelati, M.~G. Delgadino, and T.~M. Elgindi.
\newblock On the relation between enhanced dissipation timescales and mixing
  rates.
\newblock {\em Communications on Pure and Applied Mathematics},
  73(6):1205--1244, 2020.

\bibitem{zlatovs2010diffusion}
A.~Zlato{\v{s}}.
\newblock Diffusion in fluid flow: dissipation enhancement by flows in 2d.
\newblock {\em Communications in Partial Differential Equations},
  35(3):496--534, 2010.

\end{thebibliography}

\end{document}